\pdfoutput=1
\documentclass{article}

\usepackage{graphicx}
\usepackage{indentfirst,csquotes}
\usepackage{amssymb,amsthm,amsmath}
\usepackage{xcolor,paralist,hyperref,fancyhdr,etoolbox,cleveref}
\usepackage{orcidlink}
\usepackage[square,sort,comma,numbers]{natbib}
\usepackage{algorithm}
\usepackage{algpseudocode}
\usepackage{subcaption}
\usepackage{tikz-cd}
\usetikzlibrary{positioning, fit, backgrounds, arrows.meta}
\usepackage{enumitem}
\allowdisplaybreaks
\numberwithin{equation}{section}
\hypersetup{colorlinks=true, linkcolor=black, filecolor=black, urlcolor=black}

\newcommand{\RR}{{\mathbb{R}}} % reals
\newcommand{\NN}{{\mathbb{N}}} % natural numbers
\newcommand{\calA}{{\mathcal{A}}}
\newcommand{\calF}{{\mathcal{F}}}
\newcommand{\calL}{{\mathcal{L}}}
\newcommand{\calK}{{\mathcal{K}}}
\newcommand{\calN}{{\mathcal{N}}}

\newcommand{\calT}{{\mathcal{T}}}
\newcommand{\calX}{{\mathcal{X}}}
\providecommand{\argmin}{\operatorname*{argmin}}
\newcommand{\inner}[2]{\left\langle {#1} , {#2} \right\rangle} % inner product
\newcommand{\proj}[2]{P_{#1} \left[ {#2} \right]} % Projection 
\newcommand{\rev}[1]{#1}

\theoremstyle{plain}
\newtheorem{theorem}{Theorem}

\newtheorem{lemma}[theorem]{Lemma}
\newtheorem{proposition}[theorem]{Proposition}

\newtheorem{assumption}{Assumption}

\newcommand{\footremember}[2]{%
	\footnote{#2}
	\newcounter{#1}
	\setcounter{#1}{\value{footnote}}%
}
\newcommand{\footrecall}[1]{%
	\footnotemark[\value{#1}]%
} 

\newif\ifnotes\notestrue
\def\hfabian#1{}

\def\hpierre#1{}

\title{A structured quasi-Newton augmented Lagrangian algorithm for nonlinear least squares with linear and nonlinear constraints} %%%%%%%%%%%%

\author{Pierre Borie \orcidlink{0009-0000-1043-5057}\footremember{udem}{Université de Montréal, Department of Computer Science and Operations Research, Montreal, QC, Canada} \\ \texttt{pierre.borie@umontreal.ca} \and Fabian Bastin~\orcidlink{0000-0003-1323-6787}\footrecall{udem} \\ \texttt{bastin@iro.umontreal.ca} \and Stéphane Dellacherie~\orcidlink{0009-0005-9043-9328}\footremember{hq}{Hydro-Québec, Montreal, QC, Canada}\footremember{uqam}{Université du Québec à Montréal, Montreal, QC, Canada} \\ \texttt{dellacherie.stephane@hydroquebec.com}}

\date{October 27, 2026}

\begin{document}

	\maketitle
	
	\begin{abstract} 
		\noindent We present an algorithm for solving nonlinear least-squares problems subject to a mix of nonlinear constraints, linear equalities, and bounds. The nonlinear constraints are handled through an augmented Lagrangian reformulation of the objective, while the linear constraints are kept explicit. Each outer iteration therefore requires the approximate minimization of the augmented Lagrangian over a polyhedron, which we carry out with a gradient projection technique combined with a trust-region step computation. To capture the second-order information specific to the least-squares structure, we approximate the Hessian of the augmented Lagrangian by a Gauss--Newton term corrected by a structured quasi-Newton update, activated by a hybrid switching rule. The correction is built from a secant equation that we derive for the augmented Lagrangian itself. The outer iteration follows the augmented Lagrangian framework of Conn, Gould, Sartenaer and Toint for linearly constrained subproblems, for which we provide an explicit inner solver: a trust-region method whose Cauchy point is computed along nested faces of the polyhedron, without any projection onto it, and safeguarded by classical active-set devices, namely the release of bounds based on least-squares multiplier estimates and $\varepsilon$-active bounds. We prove that this inner solver satisfies the stopping rule required by the outer analysis, which yields global convergence of the overall method to a first-order critical point, and report numerical experiments. Our method outperforms an augmented Lagrangian solver built on the same outer update rules but without least-squares structure, and is competitive in computation time with a state-of-the-art solver for nonlinear programming.
	\end{abstract}
	
	\vspace{1em}\noindent\textbf{Keywords:} constrained nonlinear least squares, augmented Lagrangian method, structured quasi-Newton update.
	
	\section{Introduction}

We consider the constrained nonlinear least-squares (NLS) problem
\begin{equation}
	\label{pb:cnls}
	\begin{aligned}
		\min_{x\in \RR^n} \quad & \dfrac{1}{2} \|r(x)\|^2 \\
		\text{s.t.} \quad & c(x) = 0 \\
		& x \in \Omega,
	\end{aligned}
\end{equation}
where $\|\cdot\|$ denotes the $\ell_2$ norm. The residual and constraint functions $r\colon \RR^n \to \RR^{n_r}$ and $c\colon \RR^n \to \RR^{n_c}$ are assumed to be twice continuously differentiable with respective Jacobians $J$ and $C$. We may also write the least-squares objective as a general function $f$ with gradient 
\[\nabla f(x) = J(x)^\top r(x).\] 
The set $\Omega$ consists of linear constraints and is defined as
\begin{equation}\label{eq:linear_constraints}
	\Omega = \left\{ x \in \RR^n \ | \ Ax=b,\ l \le x \le u\right\},
\end{equation}
with $A\in\RR^{m\times n}$, $m < n$, $b\in \RR^m$, and $l,u \in \RR^n$. Some components of the latter two vectors may be infinite, to model unbounded variables. For convenience, we also assume that the linear equalities do not involve single-variable constraints. Formulation~\eqref{pb:cnls} is also suitable for problems with nonlinear inequality constraints of the form $g(x) \ge 0$. The latter are transformed into equality constraints by adding nonnegative slack variables, which gives the new constraints
\begin{equation*}
	g(x) - \nu = 0,\quad \nu \ge 0.
\end{equation*}
The lower and upper bounds associated with these slack variables are thus $0$ and $\infty$ respectively. Linear inequality constraints can be converted into equalities similarly.

	The Lagrangian for problem~\eqref{pb:cnls}, with respect to the nonlinear constraints, is defined by 
	\begin{equation*}
		\calL(x,\lambda) = f(x) + \inner{\lambda}{c(x)},
	\end{equation*}
	where $\lambda \in \RR^{n_c}$ is the vector of Lagrange multipliers and $\inner{\cdot}{\cdot}$ is the standard inner product.

	An algorithm for solving problem~\eqref{pb:cnls} aims to find a first-order critical point $(x_*,\lambda_*)$ that satisfies the necessary condition
	\begin{equation*}
		\begin{aligned}
			&x_* \in \Omega\ \text{and}\ c(x_*)=0 \\
			&\proj{\Omega}{x_*-\nabla_x \calL(x_*,\lambda_*)}=x_*,
		\end{aligned}
	\end{equation*}
	where $\proj{\Omega}{\cdot}$ is the projection operator onto $\Omega$, defined, for any vector $x$, as
	\begin{equation*}
		\proj{\Omega}{x}= \argmin_{v\in \Omega} \ \|v-x\|^2. 
	\end{equation*}

The NLS structure arises in data fitting, parameter estimation, and scientific computing, and its efficient resolution has attracted a sustained research effort~\cite{dennisschnabel:1996,nocedalwright:2006,bjork:2024}. 
A distinctive feature of NLS problems is that the Hessian of the objective decomposes into a first-order part $J(x)^\top J(x)$, where $J(x)$ is the Jacobian of $r$, and a second-order part involving the individual Hessians of the residuals weighted by their values. Exploiting this decomposition is the central challenge in designing
efficient algorithms.
The Gauss--Newton (GN) method retains only the first-order term, which is effective when residuals are small at the solution but can lead to poor convergence otherwise.
The Levenberg--Marquardt method, initiated by Levenberg~\cite{levenberg:1944} and Marquardt~\cite{marquardt:1963} and further studied, for instance, in~\cite{bellavia-etal:2018, bergou-etal:2021}, regularizes the GN matrix by a scalar multiple of the identity, providing robustness at the cost of potentially slower convergence on large-residual problems.
A richer class of approaches, known as structured quasi-Newton (SQN) methods, approximates the second-order part through secant equations tailored to the least-squares structure. The NL2SOL package~\cite{dennisetal:1981,dennis-etal:1989} is a popular implementation of this idea, complementing the GN term with a structured quasi-Newton correction updated via a dedicated secant equation. This algorithm also switches between the GN model and the model including the structured quasi-Newton correction, depending on the quality of the step. Other hybrid strategies with a switching criterion based on the residual magnitude have also been explored in~\cite{albaalifletcher:1985, fletcherxu:1987, zhouchen:2010}. Further SQN formulations, including factorized updates and sizing strategies, have been proposed in~\cite{biggs:1977, betts:1976, huschens1994, yabetakahashi:1991, zhang-etal:2000}. A comprehensive numerical comparison of these techniques can be found in~\cite{lucksan-etal:2019}.

	Extending structured NLS methods to the constrained setting has received comparatively less attention, though it remains an active area of research. Such formulations arise, for example, in data-fitting problems where the constraints enforce the physical consistency of an underlying model~\cite{grenieretal:2006, delbos-etal2006, pu-etal:2015}. The convex case, especially the bound-constrained case, is studied in~\cite{bellavia-etal:2004, bellavia-etal:2011, porcelli:2013, goncalvesmenezes:2020}.

For more general nonlinear constraints, most methods operate within a sequential quadratic programming (SQP) framework. In that paradigm, each iteration solves a quadratic subproblem obtained by linearizing the constraints and approximating the Hessian of the Lagrangian. An SQP method whose Hessian approximation is based on a secant equation derived from that of Biggs~\cite{biggs:1977} is proposed by Li et al.~\cite{li-etal:2002}. A structured BFGS update of the projected Hessian of the Lagrangian is used in~\cite{tjoabiegler:1991}, while a structured approximation of the projected Hessian of an exact penalty function is updated within the null space of the active constraints in~\cite{amiribartels:1989}.
For problems with equality constraints only, more recent work includes regularized approaches. A Levenberg--Marquardt approximation of the Lagrangian Hessian is combined with a trust-region method in~\cite{bergou-etal:2021}, and a regularized Karush--Kuhn--Tucker (KKT) system is derived in~\cite{orbansiqueira:2020} by adding primal and dual regularization terms.
Structured approximations of the Hessian within general constrained formulations have also been studied from a theoretical perspective in~\cite{tapia:1988}.
When the constraints can be split into two categories, as in formulation~\eqref{pb:cnls}, an alternative to SQP methods is the augmented Lagrangian (AL) framework, also known as the method of multipliers, independently introduced by Hestenes~\cite{hestenes:1969} and Powell~\cite{powell:1969}, with further theoretical foundations provided by Rockafellar~\cite{rockafellar:1973}.
Rather than solving a sequence of quadratic subproblems, AL methods incorporate the nonlinear constraints into the objective as a penalty term and solve a sequence of subproblems in which only the easier constraints remain.
For problem~\eqref{pb:cnls}, this amounts to approximately minimizing, at each outer iteration, the AL function over the polyhedral set $\Omega$, thereby reducing the problem to a linearly constrained one. This strategy is particularly attractive as the linear part can be handled directly and efficiently by gradient projection techniques, without sacrificing the ability to enforce nonlinear constraints.
The theoretical underpinnings of this approach, including global convergence results, are developed in~\cite{conn-etal:1988a,conn-etal:1991, conn-etal:1993, conn-etal:1996b} and summarized in~\cite[Chapter~14]{conn-etal:2000}.
A major implementation of this philosophy is the LANCELOT solver~\cite{conn-etal:1992}, which handles bound-constrained subproblems using a gradient projection technique~\cite{conn-etal:1988b}.
More recent AL implementations have broadened this framework in several directions. Convergence is established under weaker constraint qualifications in~\cite{andreani-etal:2008}. A primal-dual AL method that behaves like a stabilized SQP method in the local regime is proposed in~\cite{gillrobinson:2012}, a matrix-free AL approach is implemented in~\cite{arreckx-etal:2016}, and an adaptive AL algorithm with inexact subproblem solves is developed in~\cite{curtis-etal:2015}.
For the inner minimization that arises at each outer AL iteration, gradient projection methods~\cite{moretoraldo:1991, linmore:1999a} are particularly effective for bound-constrained problems as they allow for rapid changes in the active set throughout the iterations. The linear equality case can be handled efficiently, as the projections can be computed by solving a system of normal equations~\cite{gould-etal:2001}.
When both bounds and linear equalities are present, as in our setting, direct projection onto $\Omega$ is not available in closed form. Nevertheless, we will see in the next section that the norm of the reduced gradient, complemented by a sign check on least-squares estimates of the bound multipliers as in active-set solvers such as MINOS~\cite{murtaghsaunders:1978} and LSNNO~\cite{tointtuyttens:1992}, provides a practical criticality measure, as already noted in~\cite[Section~7.2]{conn-etal:1996b}.

It is worth pointing out what this combination leaves uncovered in the existing literature. Structured quasi-Newton methods for least squares are, with the exceptions discussed above, developed for the unconstrained case. Augmented Lagrangian solvers such as LANCELOT~\cite{conn-etal:1992} and Percival~\cite{arreckx-etal:2016} handle bound-constrained subproblems, the other linear constraints being incorporated into the AL, and treat the objective as a generic nonlinear function. The framework of~\cite{conn-etal:1996b} keeps the linear constraints explicit in the subproblems, but leaves their solution to existing linearly constrained solvers and does not exploit a least-squares structure. The structured methods available for constrained least squares operate within an SQP framework~\cite{amiribartels:1989, tjoabiegler:1991, li-etal:2002} and do not keep bounds and linear equalities explicit in their subproblems. Among the methods surveyed above, we are not aware of one combining the three features at once, namely an approximation exploiting the least-squares structure, a feasible set mixing nonlinear equalities with a general polyhedron whose linear constraints are neither penalized nor linearized, and a supporting global convergence analysis. It is this combination that our algorithm addresses.

The present work brings the SQN philosophy of nonlinear least squares into the AL framework. Our contributions are the following.
\begin{enumerate}[label=(\roman*), leftmargin=*, itemsep=2pt]
	\item We derive a structured secant equation for the Hessian of the AL, given in~\eqref{eq:structured_secant_eq}, which simultaneously captures the curvature carried by the residuals and that carried by the nonlinear constraints, the latter being weighted by the shifted multipliers $\bar{\lambda} = \lambda + \mu c$, built from the multiplier estimate $\lambda$ and the penalty parameter $\mu$ of the outer iteration. Correcting the GN term of the AL by a structured update satisfying this equation yields an approximation that satisfies the standard secant equation up to second-order terms in the step, the equation holding exactly if the residuals and constraints are affine. To the best of our knowledge, this is the first SQN approximation designed for the AL of a least-squares problem. Following~\cite{tjoabiegler:1991}, we complement it with a hybrid switching rule driven by the relative decrease of the AL itself, rather than of the objective, so that the decision to activate the second-order correction reflects the penalized function that is actually minimized.
	\item We provide an explicit inner solver for the subproblems over the polyhedron~\eqref{eq:linear_constraints}, in which bounds and linear equalities coexist, by adapting to the polyhedral case the gradient projection approach of~\cite{conn-etal:1988b} that underlies the bound-constrained inner solver of LANCELOT. Unlike the generalized Cauchy point procedure of~\cite{conn-etal:1993}, which requires the approximate solution of a linear program at every trial step, or the projected gradient methods of~\cite{calamaimore:1987}, which require projections onto $\Omega$, our Cauchy point is computed along nested faces of $\Omega$, with projections onto subspaces only. This path is safeguarded by two classical devices, namely the release of the bounds whose least-squares multiplier estimates have the wrong sign, as in gradient projection methods since~\cite{rosen:1960}, and $\varepsilon$-active bounds, as in the $\varepsilon$-active sets of~\cite{zoutendijk:1960,topkisveinott:1967,bertsekas:1982} and the dominant constraints of~\cite{conn-etal:1996b}. We prove the global convergence of the resulting inner iterations, and then that of the overall method to a first-order critical point of problem~\eqref{pb:cnls}. Our analysis follows the framework of~\cite{conn-etal:1996b}, of which our problem is a particular case\rev{. To apply it}, we show that the measure used in our implementation, based on the reduced gradient and on least-squares multiplier estimates, certifies the criticality test of~\cite{conn-etal:1996b} \rev{without any assumption on the signs of the bound multipliers}.
	\item We provide TRAULLS (Trust Region AUgmented nonLinear Least-squares Solver), an open-source Julia~\cite{bezanson-etal:2017} implementation available at \url{https://github.com/UncertainLab/Traulls.jl}. All the experiments of Section~\ref{sec:numerical_experiments} can be reproduced from the \texttt{benchmark} folder of that repository, in the version tagged \texttt{coap-2026}.
\end{enumerate}
Together, these ingredients yield a solver specifically designed for NLS problems while handling constraints of general form, including nonlinear equalities and inequalities alongside linear equalities and bounds.

We end this section by introducing some notation. Components of a vector $x\in\RR^n$ are denoted $x_i$ for $i=1,\ldots,n$. Vectors $e_i$ denote the columns of the $n \times n$ identity matrix written $I_n$ -- or simply $I$ when the dimension is clear from the context. When describing iterative processes to solve problem~\eqref{pb:cnls}, we index vectors and matrices by the iteration number. Our method has a nested structure (Section~\ref{sec:algorithm}): we use an uppercase index $K$ for its outer iterations and a lowercase index $k$ for the inner iterations. For instance, at outer iteration $K$, $(x_K,\lambda_K)$ is the current primal-dual iterate, while for functions evaluated at an iterate $x_k$ we use the shorthand notation $r_k = r(x_k)$, $c_k =c(x_k)$, $J_k=J(x_k)$, etc., and similarly at outer iterates. To avoid interfering with the previous notation for components, we add parentheses to make the distinction clear. For instance, $(x_k)_i$ denotes the $i$-th component of the vector $x_k$. For any sequence $(u_k)_k$, the limit of a converging subsequence indexed by $\calK \subset \NN$ will be written $\lim_{k \in \calK} u_k$. The orthogonal complement of a subspace $V \subset \RR^n$, i.e., the set $\{ w\ | \ \inner{w}{v}=0\ \text{for all}\ v \in V\}$, will be written $V^\perp$.

	The rest of the paper is organized as follows. In Section~\ref{sec:algorithm}, we present our algorithmic framework and discuss implementation details. We study global convergence in Section~\ref{sec:convergence_analysis}. Numerical experiments are shown in Section~\ref{sec:numerical_experiments} and we conclude with perspectives.

	\section{The algorithm}\label{sec:algorithm}

	\subsection{Augmented Lagrangian framework}\label{subsec:al_framework}

We formally state the first assumptions about problem~\eqref{pb:cnls}.
\begin{assumption}\label{assumption:functions_C2}
	Residuals $r$ and constraints $c$ are twice continuously differentiable on $\RR^n$.
\end{assumption}
\begin{assumption}\label{assumption:full_rank_A}
	The equality constraint matrix $A$ has full row rank.
\end{assumption}
\begin{assumption}\label{assumption:feasible_linear_cons}
	The feasible region for the linear constraints $\Omega$ is nonempty.
\end{assumption}

Our method is based on the AL function defined as
\begin{equation}\label{eq:al}
	\Phi(x,\lambda,\mu) := \dfrac{1}{2}\|r(x)\|^2 + \inner{\lambda}{c(x)} + \dfrac{\mu}{2} \|c(x)\|^2,
\end{equation}
where $\mu > 0$ is a penalty parameter.

We maintain the linear constraints in the formulation of the problem and only penalize the violation of the nonlinear constraints.
The gradient is given by 
\begin{equation*}
	\nabla_x \Phi(x,\lambda,\mu) = J(x)^\top r(x) + C(x)^\top\bar{\lambda}(x,\lambda,\mu),
\end{equation*}
where
\begin{equation}\label{eq:1st_order_mult}
	\bar{\lambda}(x,\lambda,\mu):=\lambda + \mu c(x)
\end{equation} 
is the vector of first-order estimates of the Lagrange multipliers, which we call shifted multipliers. The AL and Lagrangian gradients are related by the identity
\begin{equation}\label{eq:identity_al_lag_grad}
	\nabla_x \Phi(x,\lambda,\mu) = \nabla_x \calL(x,\bar{\lambda}(x,\lambda,\mu)).
\end{equation}
The Hessian is given by
\begin{equation*}
	\nabla^2_{xx} \Phi(x,\lambda,\mu) = J(x)^\top J(x) + \mu C(x)^\top C(x) +  S(x,\lambda,\mu) 
\end{equation*}
with the second-order terms explicitly given by
\begin{equation}\label{eq:al_hessian_2nd_terms}
	S(x,\lambda,\mu) =\sum_{i=1}^{n_r} r_i(x) \nabla^2r_i(x) + \sum_{i=1}^{n_c} \nabla^2 c_i(x) \bar{\lambda}_i(x,\lambda,\mu).
\end{equation}
For fixed $\lambda$ and $\mu$, reformulating problem~\eqref{pb:cnls} with function~\eqref{eq:al} gives the linearly constrained problem
\begin{equation}\label{pb:cnls_al_reformulation} 
	\begin{aligned}
		\min_{x} \quad& \Phi(x,\lambda,\mu)  \\
		\text{s.t.}  \quad & x \in \Omega.
	\end{aligned}	
\end{equation}
Moving the nonlinear constraints into the objective simplifies the problem since only linear constraints are left, which enables one to use iterative methods for linearly constrained optimization~\cite{gould-etal:2001,conn-etal:1988b} while still improving feasibility of the nonlinear constraints. Since $\Omega$ is a convex set, a first-order critical point of~\eqref{pb:cnls_al_reformulation} can also be characterized by the condition
\begin{equation*}
	\proj{\Omega}{x_*-\nabla_x \Phi(x_*,\lambda,\mu)} = x_*.
\end{equation*}
In our method, as in standard AL methods~\cite{conn-etal:1991,conn-etal:2000} and more precisely as in the framework of~\cite{conn-etal:1996b} for linearly constrained subproblems, each iterate $x_K$ is an approximate solution of a subproblem of the form~\eqref{pb:cnls_al_reformulation} satisfying
\begin{equation}\label{eq:algo_traulls_criticality_test}
	\left\|\proj{\calT_{\varepsilon_K}(x_K)}{-\nabla_x\Phi(x_K, \lambda_K, \mu_K)}\right\| \le \omega_K, \qquad \varepsilon_K = \min(\varepsilon_{\max},\omega_K),
\end{equation}
for a fixed vector of multipliers $\lambda_K$, penalty parameter $\mu_K$ and tolerance $\omega_K$, and a constant $\varepsilon_{\max} > 0$. Here $\calT_\varepsilon(x)$ is the tangent cone at $x$ of the linear equalities and of the bounds that lie within distance $\varepsilon$ of $x$, called dominant in~\cite{conn-etal:1996b}, which we define precisely in~\eqref{eq:tangent_cone_eps} below. For $\varepsilon = 0$, $\calT_0(x)$ is the tangent cone of $\Omega$ at $x$, and~\eqref{eq:algo_traulls_criticality_test} implies $\|\proj{\Omega}{x_K-\nabla_x\Phi(x_K, \lambda_K, \mu_K)}-x_K\| \le \omega_K$. Taking $\varepsilon_K > 0$, with $\varepsilon_K \to 0$ as $\omega_K \to 0$, is what allows the inner iterations to terminate (Section~\ref{subsec:inner_convergence}).
If the iterate $x_K$ sufficiently reduces the violation of the nonlinear constraints, a dual ascent step is taken on the Lagrange multipliers by setting
\begin{equation*}
	\lambda_{K+1} = \lambda_K + \mu_K c_K = \bar{\lambda}(x_K,\lambda_K,\mu_K).
\end{equation*}
If, on the contrary, feasibility has not been improved, the minimization~\eqref{pb:cnls_al_reformulation} is restarted using an increased penalty parameter $\mu_{K+1} > \mu_K$ to enforce feasibility. The method is outlined in Algorithm~\ref{algo:traulls}. Following our notation for iteration-dependent quantities, we respectively write $\nabla_x\Phi(x_K, \lambda_K, \mu_K)$ and $\bar{\lambda}(x_K,\lambda_K,\mu_K)$ in the compact form $\nabla_x \Phi_K$ and $\bar{\lambda}_K$.
\begin{algorithm}
	\caption{Augmented Lagrangian algorithm}\label{algo:traulls}
	\begin{algorithmic}[1]
		\Require Initial point $x^s_0 \in \Omega$, multiplier estimate $\lambda_0$, penalty parameter $\mu_0 > 1$, increase factor $\tau > 1$, tolerance update constants $\omega,\eta,\kappa_\omega,\kappa_\eta,\beta_\omega,\beta_\eta > 0$, stopping tolerances $\omega_*,\eta_* > 0$, threshold $\varepsilon_{\max} > 0$
		\State{Set initial tolerances $\omega_0\gets \omega \mu_0^{-\kappa_\omega}$ and $\eta_0\gets \eta \mu_0^{-\kappa_\eta}$}
		\For{$K=0,1,2,\ldots$}
		\State\label{line:inner_iteration} Starting from $x_K^s$, approximately solve $\min_{x\in \Omega} \Phi(x,\lambda_K,\mu_K)$ to find $x_K \in \Omega$ satisfying~\eqref{eq:algo_traulls_criticality_test}
		\If{$\|c(x_K)\| \le \eta_K$}
		\If{\(\|\proj{\calT_{\varepsilon_*}(x_K)}{-\nabla_x\Phi_K}\| \le \omega_*\), with $\varepsilon_* = \min(\varepsilon_{\max},\omega_*)$, and \(\left\Vert c(x_K)\right\Vert \le \eta_*\)}
		\State{\textbf{stop} and} \Return approximate solution $(x_K,\bar{\lambda}_K)$
		\EndIf
		\State Update the Lagrange multipliers $\lambda_{K+1}\gets \bar{\lambda}_K$ 
		\State Set the next starting point $x_{K+1}^s \gets x_K $
		\State Leave the penalty parameter unchanged \(\mu_{K+1} \gets \mu_K\)
		\State Decrease the tolerances $\omega_{K+1}\gets \omega_K \mu_{K+1}^{-\beta_\omega}$ and $\eta_{K+1}\gets \eta_K \mu_{K+1}^{-\beta_\eta}$
		\Else{}
		\State Set the next starting point $x_{K+1}^s \gets x_K$ and leave the multipliers unchanged $\lambda_{K+1} \gets \lambda_K$
		\State Increase the penalty parameter $\mu_{K+1} \gets \tau\mu_K$
		\State Update the tolerances $\omega_{K+1}\gets \omega \mu_{K+1}^{-\kappa_\omega}$ and $\eta_{K+1}\gets \eta \mu_{K+1}^{-\kappa_\eta}$
		\EndIf
		\EndFor
	\end{algorithmic}
\end{algorithm}
One first notices that Algorithm~\ref{algo:traulls} has an outer-inner iteration structure, in the sense that each iteration requires approximately solving an optimization problem (line~\ref{line:inner_iteration}), which also involves an iterative process. The techniques used for the latter are detailed in the following sections.

Tolerances in Algorithm~\ref{algo:traulls} are updated in such a way that both sequences $\omega_K$ and $\eta_K$ tend to $0$ when $K \to \infty$. The update scheme that we outline follows the rules described in~\cite[Algorithm~14.4.2]{conn-etal:2000}. The conditions on the parameters are given in Section~\ref{subsec:implementation_details}, and the values used in our experiments in Section~\ref{sec:numerical_experiments}.

The projection onto the cone $\calT_\varepsilon(x)$ in~\eqref{eq:algo_traulls_criticality_test}, like the projection onto $\Omega$, is not available in closed form, since it amounts to solving a quadratic program. For our implementation, we use instead a measure based on the reduced gradient and on least-squares estimates of the bound multipliers, as suggested in~\cite[Section~7.2]{conn-etal:1996b}, which we define now.

For $x \in \Omega$, $\calA^+(x)$ and $\calA^-(x)$ respectively denote the indices of upper and lower bounds satisfied as equalities (active) at $x$, and we set $\calA(x) := \calA^+(x) \cup \calA^-(x)$, so that
\begin{equation*}
	i \in \calA(x) \iff x_i \in \{l_i,u_i\}.
\end{equation*}
We also introduce the notion of tangent space at a point $x$, written $T(x)$ and defined as the set of directions $d$ such that the same constraints are satisfied with equality at both $x$ and $x+d$. In our case, this corresponds to the linear equality constraints and the active bounds. Formally,
\begin{equation}\label{eq:tangent_space}
	T(x) := \left\{ d \in \RR^n \ | \ Ad = 0,\ d_i=0\ \text{for all}\ i\in \calA(x)\right\}.
\end{equation}
At a given point \((x,\lambda)\) and parameter $\mu$, the reduced gradient is defined as
\begin{equation*}
	\proj{T(x)}{\nabla_x \Phi(x,\lambda,\mu)},
\end{equation*}
where $\proj{T(x)}{\cdot}$ denotes the orthogonal projection operator onto the tangent space $T(x)$.
Provided that the multipliers associated with the active bounds have the appropriate sign, the quantity $\|\proj{T(x)}{\nabla_x \Phi(x,\lambda,\mu)}\|$ is a suitable criticality measure, used in solvers for linearly constrained optimization such as MINOS~\cite{murtaghsaunders:1978} and LSNNO~\cite{tointtuyttens:1992}.
Projections onto the tangent space also have the advantage of being available in closed form. The subspace $T(x)$ is the null space of the matrix
\begin{equation}\label{eq:matrix_tangent_space}
	\tilde{A} = \begin{pmatrix}
		A \\ E_\calA^\top
	\end{pmatrix},
\end{equation}
where $E_\calA$ is formed by the columns $(e_i)_{i \in \calA(x)}$ of the identity. We assume throughout that these active constraints remain nondegenerate.
\begin{assumption}\label{assumption:full_rank_active}
	For every $x \in \Omega$, the matrix~\eqref{eq:matrix_tangent_space} formed from the bounds active at $x$ has full row rank.
\end{assumption}
Assumption~\ref{assumption:full_rank_active} strengthens Assumption~\ref{assumption:full_rank_A}: it is a linear independence (nondegeneracy) condition on the active linear constraints, which holds if and only if the columns of $A$ associated with the variables that are not at an active bound form a matrix of full row rank. Under this assumption, projecting onto $T(x)$ is equivalent to applying the orthogonal projection operator
\begin{equation}\label{eq:minor_subpb_projector}
	\tilde{P} = I - \tilde{A}^\top\left(\tilde{A}\tilde{A}^\top\right)^{-1}\tilde{A}.
\end{equation}
We will provide more details on how the projections are calculated in practice in Section~\ref{subsec:implementation_details}. Under Assumption~\ref{assumption:full_rank_active}, every vector $v \in \RR^n$ admits the unique decomposition
\begin{equation}\label{eq:ls_decomposition}
	v = \proj{T(x)}{v} + A^\top \xi + E_\calA \theta,
\end{equation}
where $(\xi,\theta) \in \RR^m \times \RR^{|\calA(x)|}$ solves the least-squares problem $\min \|v - \tilde{A}^\top (\xi,\theta)\|$. For $v = \nabla_x\Phi(x,\lambda,\mu)$, the components $\theta_i$, $i \in \calA(x)$, are least-squares estimates of the multipliers of the active bounds. They come at no extra cost, since $(\xi,\theta)$ is the auxiliary vector of the system solved to compute $\tilde{P}v$ via the normal equations approach (see Section~\ref{subsec:implementation_details}). With the decomposition~\eqref{eq:ls_decomposition}, at a first-order critical point, the multiplier of an active lower bound is nonnegative and that of an active upper bound is nonpositive. We measure the violation of these sign conditions by the vector $\theta^{\mathrm{w}}$ with components
\begin{equation*}
	\theta^{\mathrm{w}}_i = \left\{\begin{aligned}
		& \max(0,-\theta_i) & & \text{if } i \in \calA^-(x) \setminus \calA^+(x), \\
		& \max(0,\theta_i) & & \text{if } i \in \calA^+(x) \setminus \calA^-(x), \\
		& 0 & & \text{if } i \in \calA^+(x) \cap \calA^-(x),
	\end{aligned}\right.
\end{equation*}
the last case corresponding to fixed variables, whose multiplier may have any sign, and we define the criticality measure
\begin{equation}\label{eq:criticality_measure_pi}
	\pi(x,\lambda,\mu) := \max\left( \left\|\proj{T(x)}{\nabla_x\Phi(x,\lambda,\mu)}\right\|_\infty, \|\theta^{\mathrm{w}}\|_\infty \right).
\end{equation}
The reduced gradient alone vanishes at every point where $\nabla_x\Phi$ lies in the row space of $\tilde{A}$, whatever the signs of the bound multipliers, in particular at every vertex of $\Omega$. It cannot certify criticality on its own, and an algorithm stopping on it may terminate at a point that is not critical.

More generally, for $\varepsilon \ge 0$, we call $\varepsilon$-active the bounds of
\[\calA^-_\varepsilon(x) = \{i \mid x_i - l_i \le \varepsilon\}, \qquad \calA^+_\varepsilon(x) = \{i \mid u_i - x_i \le \varepsilon\}, \qquad \calA_\varepsilon(x) = \calA^-_\varepsilon(x)\cup\calA^+_\varepsilon(x),\]
and we define $T_\varepsilon(x)$, $\tilde{A}_\varepsilon$, the decomposition~\eqref{eq:ls_decomposition} and the vector $\theta^{\mathrm{w}}$ as above, with $\calA_\varepsilon(x)$ in place of $\calA(x)$. For a vector $v \in \RR^n$, we denote by $\pi_\varepsilon(x,v)$ the right-hand side of~\eqref{eq:criticality_measure_pi} computed in this way for $v$ in place of $\nabla_x\Phi$, and we write $\pi_\varepsilon(x,\lambda,\mu) = \pi_\varepsilon(x,\nabla_x\Phi(x,\lambda,\mu))$, so that $\pi_0 = \pi$.

Assumption~\ref{assumption:full_rank_active} only concerns the bounds that are active at a given point, whereas the $\varepsilon$-active bounds at $x$ need not be simultaneously active anywhere in $\Omega$. For an arbitrary $\varepsilon$, the matrix $\tilde{A}_\varepsilon$ may then be rank deficient, so that the multiplier estimates $\theta$ of the decomposition~\eqref{eq:ls_decomposition}, and hence the measure $\pi_\varepsilon$, are not uniquely defined, and a variable that is not fixed may be $\varepsilon$-active at both of its bounds. The next lemma shows that neither situation occurs when $\varepsilon$ lies below a threshold $\bar\varepsilon$ that only depends on $\Omega$. This threshold bounds the constant $\varepsilon_{\max}$ of~\eqref{eq:algo_traulls_criticality_test} in our convergence analysis (Section~\ref{sec:convergence_analysis}). For a set $S$ of bounds, each identified by an index and a side, we write $\Omega_S$ for the face of $\Omega$ on which all the bounds of $S$ are active.
\begin{lemma}\label{lemma:eps_threshold}
	Suppose that Assumption~\ref{assumption:full_rank_active} holds. There is a constant $\bar\varepsilon \in (0,+\infty]$, depending only on $A$, $b$, $l$ and $u$, such that, for all $\varepsilon \in [0,\bar\varepsilon]$ and $x \in \Omega$, with $S$ the set of the bounds that are $\varepsilon$-active at $x$:
	\begin{enumerate}[label=(\roman*), leftmargin=*, itemsep=2pt]
		\item the face $\Omega_S$ is nonempty;
		\item the matrix $\tilde{A}_\varepsilon$ formed by $A$ and the rows $e_i^\top$, $i \in \calA_\varepsilon(x)$, has full row rank;
		\item if $l_i < u_i$, the variable $x_i$ is $\varepsilon$-active at one of its bounds at most, and lies then at distance at least $3\varepsilon$ from the other one.
	\end{enumerate}
\end{lemma}
\begin{proof}
	The threshold $\bar\varepsilon$ must exclude the two situations described before the lemma, namely bounds that are never simultaneously active on $\Omega$ being all $\varepsilon$-active at the same point, and a variable that is not fixed being $\varepsilon$-active at both of its bounds. To quantify the first one, we denote by $r_\beta(x)$ the gap at $x$ of a bound $\beta$, that is, $x_i - l_i$ if $\beta$ is the lower bound of $x_i$ and $u_i - x_i$ if it is its upper bound. For $x \in \Omega$, the gap $r_\beta(x)$ is nonnegative, it vanishes if and only if $\beta$ is active at $x$, and $\beta$ is $\varepsilon$-active at $x$ if and only if $r_\beta(x) \le \varepsilon$. For a nonempty set $S$ of finite bounds, the largest gap $\max_{\beta\in S} r_\beta(x)$ thus measures how far $x$ is from having all the bounds of $S$ active, and we denote its smallest value over $\Omega$ by
	\[\rho_S = \min_{x\in\Omega}\ \max_{\beta\in S} r_\beta(x).\]
	This minimum is attained, since $\rho_S$ is the optimal value of the linear program that minimizes $t$ over the pairs $(x,t)$ such that $x \in \Omega$ and $r_\beta(x) \le t$ for all $\beta \in S$, which is feasible by Assumption~\ref{assumption:feasible_linear_cons} and bounded below by zero. Consequently, $\rho_S = 0$ if and only if some point of $\Omega$ has all the bounds of $S$ active, that is, if and only if $\Omega_S \ne \emptyset$. When $\Omega_S$ is empty, $\rho_S$ is therefore a positive lower bound, uniform over $\Omega$, on the largest gap of the bounds of $S$. Since there are finitely many sets of finite bounds, we may define
	\[\bar\varepsilon = \frac{1}{4}\min\left(\left\{\rho_S \mid \Omega_S = \emptyset\right\} \cup \left\{\tfrac{1}{2}(u_i-l_i) \mid -\infty < l_i < u_i < +\infty\right\}\right),\]
	where the first set accounts for the first situation and the second one, made of the half-widths of the bounded intervals $[l_i,u_i]$, for the second situation. As the minimum of a finite set of positive numbers, with the convention $\min\emptyset = +\infty$, the threshold satisfies $\bar\varepsilon \in (0,+\infty]$, and it only depends on $A$, $b$, $l$ and $u$.

	Let now $\varepsilon \in [0,\bar\varepsilon]$ and $x \in \Omega$, and let $S$ be the set of the bounds that are $\varepsilon$-active at $x$, so that $r_\beta(x) \le \varepsilon$ for all $\beta \in S$. We prove (i) by contradiction. If $\Omega_S$ were empty, $\rho_S$ would belong to the first set in the definition of $\bar\varepsilon$, and we would have
	\[\max_{\beta\in S} r_\beta(x) \le \varepsilon \le \bar\varepsilon \le \tfrac{1}{4}\rho_S < \rho_S,\]
	which contradicts the definition of $\rho_S$ as a minimum over $\Omega$. Statement (ii) is then a consequence of Assumption~\ref{assumption:full_rank_active} at a point of the face $\Omega_S$. By (i), there is a point $y \in \Omega_S$, at which all the bounds of $S$ are active. The rows of $\tilde{A}_\varepsilon$ are therefore among those of the matrix~\eqref{eq:matrix_tangent_space} formed from the bounds active at $y$. Since the latter are linearly independent, so are the former, and $\tilde{A}_\varepsilon$ has full row rank. Finally, to prove (iii), let $i$ be such that $l_i < u_i$. If one of these bounds is infinite, $x_i$ can only be $\varepsilon$-active at the other one, and there is nothing to prove. Otherwise, the second set in the definition of $\bar\varepsilon$ gives $u_i - l_i \ge 8\bar\varepsilon$. If $x_i$ is $\varepsilon$-active at its lower bound, that is, $x_i - l_i \le \varepsilon \le \bar\varepsilon$, then
	\[u_i - x_i = (u_i - l_i) - (x_i - l_i) \ge 7\bar\varepsilon.\]
	Since $\bar\varepsilon$ is positive and finite in that case, this distance is larger than $\varepsilon$, so that $x_i$ is not $\varepsilon$-active at its upper bound, and it is at least $7\varepsilon \ge 3\varepsilon$. The case of the upper bound is symmetric.
\end{proof}
We finally denote by
\begin{equation}\label{eq:tangent_cone_eps}
	\calT_\varepsilon(x) = \left\{ d \in \RR^n \ \middle| \ \begin{aligned} & Ad = 0, \\ & d_i \ge 0 \text{ for } i \in \calA^-_\varepsilon(x)\setminus\calA^+_\varepsilon(x), \\ & d_i \le 0 \text{ for } i \in \calA^+_\varepsilon(x)\setminus\calA^-_\varepsilon(x), \\ 
	& d_i = 0 \text{ for } i \in \calA^-_\varepsilon(x)\cap\calA^+_\varepsilon(x) \end{aligned} \right\}
\end{equation}
the tangent cone of the linear equalities and of the $\varepsilon$-active bounds, and by $\calN_\varepsilon(x)$ its polar cone, generated by the rows of $A$ and their opposites, the vectors $-e_i$, $i \in \calA^-_\varepsilon(x)$, and the vectors $e_i$, $i \in \calA^+_\varepsilon(x)$, i.e.
\begin{equation}\label{eq:normal_cone_eps}
	\calN_\varepsilon(x) = \bigg\{ A^\top\xi + \sum_{i \in \calA^+_\varepsilon(x)} \theta^+_i e_i - \sum_{i \in \calA^-_\varepsilon(x)} \theta^-_i e_i \ \Big|\ \xi\in\RR^m,\ \theta^+_i, \theta^-_i\ge0\bigg\}.
\end{equation}
For $\varepsilon = 0$, they are the tangent and normal cones of $\Omega$ at $x$, respectively written $\calT_\Omega(x)$ and $\calN_\Omega(x)$. We recall, for any vector $v \in \RR^n$, the Moreau decomposition
\begin{equation}\label{eq:moreau_decomposition}
	v = \proj{\calT_\varepsilon(x)}{v} + \proj{\calN_\varepsilon(x)}{v}.
\end{equation}
We stop the inner iterations of outer iteration $K$ when
\begin{equation}\label{eq:criticality_cond_reduced_grad}
	\pi_{\varepsilon_K}(x_K,\lambda_K,\mu_K) \le \omega_K,
\end{equation}
with $\varepsilon_K$ as in~\eqref{eq:algo_traulls_criticality_test}. Likewise, the final test of Algorithm~\ref{algo:traulls} is carried out as $\pi_{\varepsilon_*}(x_K,\lambda_K,\mu_K) \le \omega_*$ with $\varepsilon_* = \min(\varepsilon_{\max},\omega_*)$.
The role of the threshold $\varepsilon_K$, which keeps the Cauchy step from being arbitrarily short near the bounds, is explained in Section~\ref{subsec:inner_convergence}. We will see in Section~\ref{sec:convergence_analysis} that the measure $\pi$ vanishes exactly at the first-order critical points of~\eqref{pb:cnls_al_reformulation}. 

	\subsection{Inner minimization process}\label{subsec:inner_iteration}

In this section, we describe the inner minimization process used to produce the outer iterates of Algorithm~\ref{algo:traulls}.
Since the current multiplier estimates $\lambda$ and the penalty parameter $\mu$ are fixed, the objective function for the inner minimization only depends on $x$ so we use the shorthand notation $\varphi : x\mapsto \Phi(x,\lambda,\mu)$.
Given a point $x_0 \in \Omega$, multiplier estimates $\lambda$ and a penalty parameter $\mu$, we aim to find an approximate solution of the problem
\begin{equation}\label{eq:inner_itr_subpb}
	\begin{aligned}
		\min_x \quad & \varphi(x)\\
		\text{s.t.} \quad & x \in \Omega,
	\end{aligned}
\end{equation}
and such that
\begin{equation*}
	\pi_\varepsilon(x,\nabla \varphi(x)) \le \omega,
\end{equation*}
for a tolerance $\omega > 0$ and a threshold $\varepsilon > 0$, with the measure $\pi_\varepsilon$ of Section~\ref{subsec:al_framework}. This task also involves an iterative process, for which the iterates will be indexed with subscript $k$ but are distinct from the outer iterates of Algorithm~\ref{algo:traulls}. At each iteration $k$, we compute a step $s_k$ and recursively update the iterate by $x_{k+1}=x_k+s_k$. The step is obtained after minimizing a model of the objective function $\varphi$ around $x_k$. We consider the quadratic model
\begin{equation}\label{eq:al_quadratic_model}
	m_k(x) = \varphi_k + \inner{g_k}{x-x_k} + \dfrac{1}{2}\inner{x-x_k}{H_k(x-x_k)},
\end{equation}
where $\varphi_k:=\varphi(x_k)$, $g_k:=\nabla \varphi_k$ and $H_k$ is a symmetric approximation of the true Hessian $\nabla^2_{xx} \varphi_k$. Using the latter would negatively impact the performance of the algorithm since computing the second-order terms~\eqref{eq:al_hessian_2nd_terms} requires too much time and storage in practice. We discuss the choice of $H_k$ in Section~\ref{subsec:hessian_approx}. 
Model~\eqref{eq:al_quadratic_model} is appropriate to describe the algorithm in terms of the iterate. When discussing subproblems, we prefer to emphasize the step and then use the following model of the primal reduction $\varphi(x_k+s)-\varphi_k$, given by
\begin{equation*}
	q_k(s) = \dfrac{1}{2}\inner{s}{H_ks} + \inner{g_k}{s}.
\end{equation*}
The two formulations are related by $q_k(x-x_k)=m_k(x)-m_k(x_k)$.

We seek to compute the step $s_k$ as an approximate solution of the program
\begin{align*}
	\min_{s} \quad& q_k(s)  \\
	\text{s.t.}  \quad & x_k+s \in \Omega \\ 
	& \|s\|_\infty \le \Delta_k.
\end{align*}	
Since $x_0$ is feasible, the constraints $x_k+s\in \Omega$ are satisfied at every iteration as long as the steps satisfy
\[As = 0, \qquad l - x_k \le s \le u-x_k,\]
for all $k$.

In our method, we also incorporate a trust-region strategy to control and assess the quality of a step. This is the role of the constraint $\|s\|_\infty \le \Delta_k$, with a radius $\Delta_k > 0$, which reflects the domain on which we believe the model approximates the true function well. Since $\|s\|_\infty \le \Delta_k$ amounts to imposing $s_i \in [-\Delta_k,\Delta_k]$ for all $i$, we can rewrite this subproblem as
\begin{subequations}\label{pb:quadratic_subproblem} 
	\begin{align}
		\min_{s} \quad& q_k(s)  \\
		\text{s.t.}  \quad & As=0 \\
		& l_k \le s \le u_k,
	\end{align}	
\end{subequations}
where $l_k$ and $u_k$ are iteration-dependent bounds having components $(l_k)_i = \max\left(-\Delta_k,l_i-(x_k)_i\right)$ and $(u_k)_i = \min\left(\Delta_k,u_i-(x_k)_i\right)$ for all $i=1,\ldots,n$. Therefore, the intersection of the resulting trust region with $\Omega$ is again a polyhedron of the form~\eqref{eq:linear_constraints}: the radius constraint only tightens the bounds and leaves the linear equalities untouched, so the subproblem retains the structure of a problem with bounds and linear equalities that can be handled directly by projection machinery.

The success of an iteration and the update of the trust region are evaluated by the ratio
\begin{equation}\label{eq:tr_ratio}
	\rho_k = \dfrac{\varphi(x_k+s_k)-\varphi_k}{q_k(s_k)}.
\end{equation}
We follow the standard step acceptance criteria~\cite{conn-etal:2000}, which require constants $\eta_1, \eta_2, \gamma_1, \gamma_2$ such that 
\begin{equation}\label{eq:cond_constants_step_acceptance}
	0< \eta_1 \le \eta_2 < 1 \text{ and } 0< \gamma_1 \le \gamma_2 < 1.
\end{equation}
If $\rho_k \ge \eta_1$, the step is accepted, and the radius may be enlarged when $\rho_k \ge \eta_2$. Otherwise, the step is rejected and the radius is reduced. A typical scheme to update the radius $\Delta_k$ would be to set
\begin{equation}\label{eq:tr_basic_update}
	\Delta_{k+1} \in \left\{\begin{aligned}
		& \left[\Delta_k,\infty\right) & &\text{if } \rho_k \ge \eta_2 \\
		& \left[\gamma_2\Delta_k, \Delta_k\right]  & &\text{if } \rho_k \in [\eta_1,\eta_2) \\
		& \left[\gamma_1\Delta_k, \gamma_2\Delta_k\right]  & &\text{if }  \rho_k<\eta_1.
	\end{aligned}\right.
\end{equation}
The procedure employed to solve problem~\eqref{eq:inner_itr_subpb} is described in Algorithm~\ref{algo:trinner_iteration} and the latter is analyzed in the rest of this section. 
\begin{algorithm}
	\caption{Trust-region inner iteration algorithm}
	\label{algo:trinner_iteration}
	\begin{algorithmic}[1]
		\Require Initial point $x_0 \in \Omega$, radius $\Delta_0 >0$, threshold $\varepsilon > 0$, tolerance $\omega > 0$, constants $\eta_1, \eta_2, \gamma_1, \gamma_2$ satisfying~\eqref{eq:cond_constants_step_acceptance}
		\State Set $k \gets 0$ 
		\Repeat{}
		\State Form the model $m_k$
		\State Compute a step $s_k$ that sufficiently reduces the model
		\State Compute the ratio $\rho_k$~\eqref{eq:tr_ratio}
		\If{$\rho_k \ge \eta_1$} set $x_{k+1}\gets x_k+s_k$ 
		\Else{} set $x_{k+1} \gets x_k$ \EndIf
		\State Set $\Delta_{k+1}$ according to~\eqref{eq:tr_basic_update}
		\State Increment $k\gets k+1$
		\Until{$\pi_\varepsilon(x_k,\nabla\varphi(x_k)) \le \omega$}
	\end{algorithmic}
\end{algorithm}

	\subsection{Step computation}\label{subsec:step_computation}

	We consider the current feasible iterate $x_k$ and its associated approximate quadratic model $m_k$~\eqref{eq:al_quadratic_model}. The step is computed in two phases. 

\paragraph*{Cauchy step} We first compute a Cauchy step $s_k^C$ that ensures a decrease of the objective function sufficient to establish global convergence of the inner minimization algorithm. Next, we further minimize the objective function by exploring the subspace defined by the constraints active at the Cauchy point $x_k^C:=x_k+s_k^C$. This approach is commonly used in gradient projection techniques~\cite[Section~16.7]{nocedalwright:2006} and there are different strategies to compute a Cauchy point~\cite{conn-etal:1988b,moretoraldo:1991,linmore:1999a}. Because of the structure of the linear constraints, we cannot project directly onto the whole set $\Omega$ in practice, but we can exploit the active bounds to follow a path built from projections onto subspaces only. That is the reason why we compute our Cauchy step by finding the first local minimizer of the model along a piecewise linear path that follows the projected steepest descent direction on nested faces of the feasible set, a bound being kept fixed once it is reached. The path starts from a working set selected among the $\varepsilon$-active bounds by the signs of their multiplier estimates (Section~\ref{subsec:inner_convergence} and Appendix~\ref{appendix:cauchy_point_computation}). When $\Omega$ is a box and $\varepsilon = 0$, as in our implementation for bound-constrained problems (Section~\ref{subsec:implementation_details}), this path is the projected gradient path
\begin{equation}\label{eq:projected_gradient_path}
	s_k(t)=\proj{\Omega}{x_k-tg_k}-x_k \text{ for } t\ge 0.
\end{equation}
Over a general polyhedron, the active sets along~\eqref{eq:projected_gradient_path} need not be nested, and the two paths may differ.
Our procedure -- described in Appendix~\ref{appendix:cauchy_point_computation} -- is an adaptation to the polyhedral case of the SBMIN algorithm~\cite{conn-etal:1988b} used for the inner minimization phase of the LANCELOT solver~\cite{conn-etal:1992}.

\paragraph*{Subspace minimization} In order to improve the rate of convergence of the algorithm, we want the total step to achieve a better reduction than the Cauchy step. To do so, we build the next iterate $x_{k+1}$ after $M$ minor iterations producing the minor iterates $x_{k,1},\ldots,x_{k,M+1}$. The sequence starts at the Cauchy point and ends at the next iterate, i.e., $x_{k}+s_k^C=x_{k,1}$ and $x_{k+1}=x_{k,M+1}$. This type of approach has been shown to be effective for general optimization with bound constraints~\cite{linmore:1999a} and has also been incorporated into other AL algorithms for the inner loop minimization~\cite{conn-etal:1992,arreckx-etal:2016}.

Each minor iterate is defined after the previous one and is decomposed into 
\[x_{k,j+1}=x_{k,j}+w_{k,j},\]
where $w_{k,j}$ is a descent direction for the quadratic model $q_k$. For each minor iterate, we require
\begin{equation}\label{eq:minor_iterate_feasible}
	\begin{aligned}
		x_{k,j} \in \Omega, & & \|x_{k,j} - x_k \|_\infty \le \Delta_k, & & \calA(x_k^C) \subseteq \calA(x_{k,j}).
	\end{aligned}
\end{equation}
The first two conditions are merely that each minor iterate is feasible and the associated step lies within the trust region, while the last condition means that we can only add active bounds during this process. Note that in this context, the active bounds are those of the quadratic subproblem~\eqref{pb:quadratic_subproblem}, as a bound may become active because of the trust region, and not only because of the original bounds on the variables. Likewise, the variables fixed at their current values at the start of the Cauchy step, because of an $\varepsilon$-active bound, remain fixed during the minor iterations.
We also require the model to decrease between two successive minor iterates
\begin{equation}\label{eq:minor_iterate_decrease}
	m_k(x_{k,j+1}) \le m_k(x_{k,j}), \qquad j = 1,\ldots,M.
\end{equation}
After each minor iteration, the corresponding step is $s_{k,j}:=x_{k,j}-x_k$.

	The search direction $w_{k,j}$ is an approximate minimizer of the subproblem
	\begin{subequations}\label{pb:minor_subpb}
		\begin{align}
			\min_w \quad & m_k(x_{k,j}+w) \\
			\text{s.t.} \quad & Aw=0 \label{subeq:minor_subpb_eq_cons} \\
			& w_i = 0, \qquad i \in \calA(x_{k,j}) \label{subeq:minor_subpb_fix_bounds}.
		\end{align}
	\end{subequations}
	At a minor iteration, the free variables, indexed by $\calF(x_{k,j})$, are implicitly subject to the bounds 
	\begin{equation}\label{subeq:minor_subpb_bounds}
		\left(l^{(k,j)}\right)_i \le w_i \le \left(u^{(k,j)}\right)_i, \qquad i \in \calF(x_{k,j})
	\end{equation}
	where $l^{(k,j)}:= l_k - s_{k,j}$ and $u^{(k,j)}:= u_k - s_{k,j}$.

The minor subproblem~\eqref{pb:minor_subpb} is solved by applying the projected conjugate gradient (CG) method~\cite{gould-etal:2001} with the previous minor iterate $x_{k,j}$ as a starting point. Three termination cases can occur. First, a direction can be generated such that one of the bounds~\eqref{subeq:minor_subpb_bounds} is violated for one of the free variables. When this happens, the search direction is scaled so that the associated component lies at the bound and the CG iterations stop. The second case occurs when we generate a direction of negative curvature and is handled similarly to the first one, i.e., we modify the direction so that all the variables remain within their bounds. The third case is the normal termination one and occurs when a local minimizer -- with respect to a given tolerance -- has been found. In all cases, we return the obtained search direction $w_{k,j}$. The handling of the termination cases differs in the continuation of the procedure. The last two cases (negative curvature and local optimality) terminate the minor iteration loop. By contrast, if the CG iterations stopped because a bound was hit, we go back to solving~\eqref{pb:minor_subpb} using the projected conjugate gradient method with $x_{k,j+1}$ as a starting point and $\calA(x_{k,j+1}) \supset \calA(x_{k,j})$ as a new set of fixed components.
The specification of the projected conjugate gradient method applied to solving~\eqref{pb:minor_subpb} is outlined in Algorithm~\ref{algo:projected_cg_method}.
\begin{algorithm}
	\caption{The projected conjugate gradient method applied to~\eqref{pb:minor_subpb}}\label{algo:projected_cg_method}
	\begin{algorithmic}[1]
		\State\label{line:pcg_initial_projection}Set $w \gets 0,\ r \gets H_k(x_{k,j}-x_k)+g_k,\ v\gets \tilde{P}r,\ p \gets -v$
		\State Set tolerance $\varepsilon_{\mathrm{cg}} \gets \kappa_{\mathrm{cg}}\|v\|$
		\Repeat
		\If{$\inner{p}{H_kp}\le 0$}
		\State Set $w^+ \gets w + \gamma p$ where $\gamma$ is the smallest factor such that $w^+_i \in \{l^{(k,j)}_i,u^{(k,j)}_i\}$ for some $i \in \calF(x_{k,j})$
		\State \textbf{stop}
		\EndIf
		\State Set $\alpha \gets \inner{r}{v} / \inner{p}{H_kp}$
		\State Set $w^+ \gets w+\alpha p$
		\If{$w^+$ violates a bound}
		\State Set $w^+ \gets w + \gamma p$ where $\gamma$ is the smallest factor such that $w^+_i \in \{l^{(k,j)}_i,u^{(k,j)}_i\}$ for some $i \in \calF(x_{k,j})$
		\State \textbf{stop}
		\EndIf
		\State Set $r^+ \gets r + \alpha H_kp$
		\State Set $v^+\gets \tilde{P}r^+$ \label{line:pcg_iteration_projection} 
		\If{$\sqrt{\inner{r^+}{v^+}} < \varepsilon_{\mathrm{cg}}$} \textbf{stop}
		\EndIf
		\State Set $\beta \gets \inner{r^+}{v^+} / \inner{r}{v}$ 
		\State Set $p \gets -v^+ + \beta p$
		\State Set $w\gets w^+,\ r \gets r^+,\ v \gets v^+$
		\Until{$2(n-m-|\calA(x_{k,j})|)$ iterations have been done} 
		\State{}
		\Return $w^+$
	\end{algorithmic}
\end{algorithm} 
The operations occurring at lines~\ref{line:pcg_initial_projection}~and~\ref{line:pcg_iteration_projection} are projections onto the null space defined by the constraints~\eqref{subeq:minor_subpb_eq_cons}--\eqref{subeq:minor_subpb_fix_bounds}.

An interesting feature would be to complement each search direction by a steplength computed by a projected search~\cite{moretoraldo:1991,linmore:1999a}. This allows one to add more than one constraint at each minor iteration to the active set. The downside is that it would also require computing the projection of the gradient direction~\eqref{eq:projected_gradient_path} onto the feasible set $\Omega$ at every new trial value of the steplength. When $\Omega$ only contains bound constraints, the projection is trivial and cheap to compute. The linear equality case is more costly but the projections can also be computed directly. When $\Omega$ is of the form~\eqref{eq:linear_constraints}, the projection is less trivial, as it amounts to solving a minimum-distance quadratic program over $\Omega$ at every trial steplength. Performing such a projection repeatedly would require a dedicated quadratic programming solver -- an external dependency whose per-trial cost would outweigh the benefit of enlarging the active set more aggressively -- which we prefer to avoid.

We stop the minor-iteration procedure when the reduced gradient associated with the current active set is small enough. More formally, assume we performed the $j+1$ minor minimizations and let $T_{k,j}$ be the tangent space at $x_{k,j}$ with respect to the active bounds in $\calA(x_{k,j})$. We stop the minor iteration loop whenever
\begin{equation*}
	\left\Vert \proj{T_{k,j}}{\nabla m_k(x_{k,j+1})} \right\Vert \le \kappa_{\mathrm{mlt}} \left\Vert \proj{T_{k,j}}{\nabla m_k(x_k)}\right\Vert,
\end{equation*}
with $\kappa_{\mathrm{mlt}} \in (0,1)$. This inequality estimates whether relative progress can still be made in the tangent subspace spanned by the free variables. 

	\subsection{Structured Hessian approximation}\label{subsec:hessian_approx}

We now discuss how the Hessian of the AL is iteratively updated when we define the quadratic model of iteration $k$ in Algorithm~\ref{algo:trinner_iteration}. To set up the context and notation of this paragraph, we wish to approximate the true Hessian
\[\nabla^2 \varphi(x_k) = J_k^\top J_k + \mu C_k^\top C_k + S_k\]  
by a symmetric matrix $H_k$.
We recall that the need for an approximation mainly comes from the fact that evaluating the second-order terms in $S_k$ requires too much time and storage to be done in practice, especially for problems with a large number of variables and residuals.

The first approximation we can think of, and the simplest one, is obtained after merely linearizing the residuals and constraints in expression~\eqref{eq:al}, which gives
\begin{equation}\label{eq:hessian_gn_approx}
	H_k = J_k^\top J_k + \mu C_k^\top C_k.
\end{equation}
The GN approximation~\eqref{eq:hessian_gn_approx} -- in reference to its unconstrained counterpart -- is computationally cheap as it involves first-order derivative quantities already required to evaluate the gradient. This matrix is always positive semidefinite. It is positive definite -- which guarantees the convexity of the quadratic subproblems~\eqref{pb:quadratic_subproblem} -- if and only if the stacked matrix $\left(\begin{smallmatrix} J_k \\ \sqrt{\mu}\,C_k \end{smallmatrix}\right)$ has full column rank, that is, $\ker J_k \cap \ker C_k = \{0\}$. However, it is known to exhibit a slower rate of convergence on problems with nonzero residuals at the solution. This downside is amplified when we are to approximate the Hessian of the Lagrangian, or the AL in our case. Indeed, setting $S_k$ to the zero matrix not only neglects the contribution of the residuals to the curvature of the model, but also neglects the contribution of the Lagrange multipliers, which have no a priori reason to vanish.

	We thus need to take into account the full Hessian to form an approximation. Looking at the literature on this subject, one can observe that there is a variety of approaches focusing on the unconstrained case~\cite{biggs:1977,dennisetal:1981,huschens1994,yabetakahashi:1991,zhang-etal:2000,lucksan-etal:2019}, as it is a major challenge in improving the efficiency of algorithms for problems with nonzero residuals at the solution. Nevertheless, relevant parallels can be drawn with the AL situation, or the constrained case in general, since these studies explore techniques to deal with second-order terms. Because the first-order terms are readily available and provide curvature information, only the second-order components need to be approximated. Therefore, we look for an approximation of the form 
	\begin{equation}\label{eq:hessian_full_approx}
		H_k = B^{\mathrm{GN}}_k + B_k,
	\end{equation}
	where $B^{\mathrm{GN}}_k$ is the right-hand side of~\eqref{eq:hessian_gn_approx} and $B_k$ is a symmetric approximation of $S_k$. 

Structured approximations of the AL Hessian have been studied for general objective functions~\cite{tapia:1988} but our approach is closer to the one employed in~\cite{li-etal:2002}. In the latter, the authors base their approximation on a secant equation derived from a heuristic initially introduced in the unconstrained case~\cite{biggs:1977} which they adapt to approximate the Hessian of the Lagrangian and use in an SQP algorithm for constrained nonlinear least squares. This strategy is also at the heart of the SQN method from the popular package for unconstrained nonlinear least squares NL2SOL~\cite{dennisetal:1981,dennis-etal:1989}. The idea is the following. If we were to approximate each matrix term of the sum~\eqref{eq:al_hessian_2nd_terms}, we would have
\begin{equation}\label{eq:sum_hessian_approximations}
	B_{k+1} = \sum_{i=1}^{n_r}  r_i(x_{k+1}) B^{r_i}_{k+1} + \sum_{i=1}^{n_c} \bar{\lambda}_i(x_{k+1},\lambda,\mu) B^{c_i}_{k+1},
\end{equation}
with $B^{r_i}_{k+1} \approx \nabla^2r_i(x_{k+1})$ and $B^{c_i}_{k+1} \approx  \nabla^2 c_i(x_{k+1})$. To get an accurate approximation, given a step $s_k$, it is reasonable to require
\begin{equation}\label{eq:components_secant_equations}
	B^{r_i}_{k+1}s_k = \nabla r_i(x_{k+1}) - \nabla r_i(x_k) \qquad \text{and} \qquad B^{c_i}_{k+1}s_k = \nabla c_i(x_{k+1}) - \nabla c_i(x_k),
\end{equation}
i.e., that each Hessian maps the change in the variables to the change in the gradients. Noticing that, for each $i$, the vectors $\nabla r_i(x_{k+1}) - \nabla r_i(x_k)$ and $\nabla c_i(x_{k+1}) - \nabla c_i(x_k)$ are respectively the $i$-th columns of $\left(J_{k+1}-J_k\right)^\top$ and $\left(C_{k+1}-C_k\right)^\top$, multiplying the equations of~\eqref{eq:components_secant_equations} by $r_i(x_{k+1})$ and $\bar{\lambda}_i(x_{k+1},\lambda,\mu)$, respectively, and summing over the residual and constraint indices gives, by~\eqref{eq:sum_hessian_approximations}, the structured secant equation
\begin{equation}\label{eq:structured_secant_eq}
	B_{k+1}s_k = \left(J_{k+1}-J_k\right)^\top r_{k+1} + \left(C_{k+1}-C_k\right)^\top\bar{\lambda}_{k+1}.
\end{equation}
To follow the conventional notation of quasi-Newton methods, we will denote the right-hand side of~\eqref{eq:structured_secant_eq} by $\tilde{y}_k$ to distinguish it from the standard term $y_k:= g_{k+1}-g_k$. A matrix $B_{k+1}$ satisfying~\eqref{eq:structured_secant_eq} does not, in general, make the approximation $H_{k+1} = B^{\mathrm{GN}}_{k+1} + B_{k+1}$ satisfy the standard secant equation $H_{k+1}s_k = y_k$, and we denote the resulting error by
\begin{equation}\label{eq:secant_error}
	\delta^{\mathrm{sec}}_k := y_k - H_{k+1}s_k = y_k - B^{\mathrm{GN}}_{k+1}s_k - \tilde{y}_k.
\end{equation}
Since $\lambda$ and $\mu$ are fixed during the inner minimization, the gradient reads $g_k = J_k^\top r_k + C_k^\top(\lambda + \mu c_k)$, and the terms involving $J_{k+1}$ and $C_{k+1}$ cancel in $y_k - \tilde{y}_k$. This yields the exact identity
\[y_k - \tilde{y}_k = J_k^\top\left(r_{k+1}-r_k\right) + \mu\, C_k^\top\left(c_{k+1}-c_k\right).\]
Its right-hand side is the variation of the gradient between $x_k$ and $x_{k+1}$ when the Jacobians are frozen at their values at $x_k$. The error $\delta^{\mathrm{sec}}_k$ compares this variation with the one predicted by the GN matrix at $x_{k+1}$, namely $B^{\mathrm{GN}}_{k+1}s_k$, and can thus be seen as a linearization error of the GN term. Expanding the residuals and constraints around $x_k$ gives $r_{k+1}-r_k = J_ks_k + O(\|s_k\|^2)$ and $c_{k+1}-c_k = C_ks_k + O(\|s_k\|^2)$, while the Jacobians satisfy $J_{k+1} = J_k + O(\|s_k\|)$ and $C_{k+1} = C_k + O(\|s_k\|)$. Consequently,
\[y_k - \tilde{y}_k = B^{\mathrm{GN}}_ks_k + O\left((1+\mu)\|s_k\|^2\right) \quad \text{and} \quad B^{\mathrm{GN}}_{k+1}s_k = B^{\mathrm{GN}}_ks_k + O\left((1+\mu)\|s_k\|^2\right),\]
so that $\delta^{\mathrm{sec}}_k = O((1+\mu)\|s_k\|^2)$. As in SQN methods for unconstrained least squares~\cite{lucksan-etal:2019}, the approximation $H_{k+1}$ therefore satisfies the standard secant equation up to second-order terms in the step, and exactly if the residuals and the constraints are affine, since $\delta^{\mathrm{sec}}_k$ then vanishes.

One could instead enforce the standard secant equation exactly, by replacing $\tilde{y}_k$ with $y_k - B^{\mathrm{GN}}_{k+1}s_k = \tilde{y}_k + \delta^{\mathrm{sec}}_k$ in the update of $B_k$. This is the default choice of~\cite{denniswalker:1981} for structured secant updates, which, in the experience of~\cite{dennis-etal:1989}, is inferior to choices exploiting the structure of each application, such as the NL2SOL choice that~\eqref{eq:structured_secant_eq} extends. In the AL setting, it has moreover a drawback without counterpart in the unconstrained case. Since $B_{k+1}$ is meant to approximate the second-order term $S_{k+1}$, each of the two right-hand sides should be compared with $S_{k+1}s_k$, and the exact choice adds $\delta^{\mathrm{sec}}_k$ to the structured one. Assume that the second derivatives of $r$ and $c$ are Lipschitz continuous. Once again, expanding $r$ and $c$ around $x_k$, now to second order, gives
\begin{equation}\label{eq:exact_secant_difference}
	\begin{aligned}
		\delta^{\mathrm{sec}}_k = {}& -\sum_{i=1}^{n_r}(J_ks_k)_i\,\nabla^2 r_i(x_k)s_k - \tfrac{1}{2}J_k^\top q^r_k + O\big((1+\mu)\|s_k\|^3\big)\\
		& -\mu\Big(\sum_{i=1}^{n_c}(C_ks_k)_i\,\nabla^2 c_i(x_k)s_k + \tfrac{1}{2}C_k^\top q^c_k\Big),
	\end{aligned}
\end{equation}
where $(q^r_k)_i = \inner{s_k}{\nabla^2 r_i(x_k)s_k}$ and $(q^c_k)_i = \inner{s_k}{\nabla^2 c_i(x_k)s_k}$. The right-hand side is quadratic in $s_k$, of the form $M(s_k)s_k$ with $\|M(s_k)\| = O((1+\mu)\|s_k\|)$, and its second line is weighted by the penalty parameter itself. By contrast, $\tilde{y}_k - S_{k+1}s_k$ is a combination of the vectors $\nabla r_i(x_{k+1}) - \nabla r_i(x_k) - \nabla^2 r_i(x_{k+1})s_k$ and $\nabla c_i(x_{k+1}) - \nabla c_i(x_k) - \nabla^2 c_i(x_{k+1})s_k$, each of order $\|s_k\|^2$, weighted by $r_i(x_{k+1})$ and $\bar{\lambda}_i(x_{k+1},\lambda,\mu)$ respectively. It is thus of order $(\|r_{k+1}\| + \|\bar{\lambda}_{k+1}\|)\|s_k\|^2$, and depends on $\mu$ only through the shifted multipliers $\bar{\lambda}_{k+1}$. We expect the latter not to grow with the penalty parameter, but this expectation is heuristic, as it transposes to the inner iterates a property of the outer ones. The shifted multipliers $\bar{\lambda}_K$ at the outer iterates converge to $\lambda_*$ whatever the behavior of $\mu_K$ (see Lemma~\ref{lemma:key_lemma}), and those at the inner iterates approach $\bar{\lambda}_K$ as these iterates approach $x_K$. Enforcing $H_{k+1}s_k = y_k$ exactly would thus feed into $B_{k+1}$ the linearization error $\delta^{\mathrm{sec}}_k$ of the GN term, whose constraint part is proportional to $\mu$, which matters when the penalty parameter has to be increased to reach feasibility. In the unconstrained case, $\mu = 0$ and only the first line of~\eqref{eq:exact_secant_difference} remains, a term of order $\|s_k\|^2$ without any large factor.

The closest existing proposal to the approximation~\eqref{eq:structured_secant_eq}  is the one in~\cite{li-etal:2002}. However, the two works differ in the algorithmic framework within which the approximation is used. In~\cite{li-etal:2002}, the authors operate within an SQP method, so their approximation targets the Hessian of the Lagrangian and is used to build a quadratic subproblem in which the constraints appear linearized, whereas ours targets the Hessian of the AL and is used within a minimization over a fixed polyhedron. In particular, the curvature contributed by the nonlinear constraints enters our secant equation through the shifted multipliers $\bar{\lambda}_{k+1} = \lambda + \mu c(x_{k+1})$, built from $\lambda$ and $\mu$, which are fixed during the inner minimization, and not through Lagrange multiplier estimates updated after every iteration.

We can now derive an update formula from equation~\eqref{eq:structured_secant_eq}. We focus on the SR1 update
\begin{equation}\label{eq:structured_sr1_update}
	B_{k+1} = B_k + \dfrac{(\tilde{y}_k-B_ks_k)(\tilde{y}_k-B_ks_k)^\top}{\inner{\tilde{y}_k-B_ks_k}{s_k}},
\end{equation}
if $\inner{\tilde{y}_k-B_ks_k}{s_k} \neq 0$. When the denominator in~\eqref{eq:structured_sr1_update} is zero, we simply set $B_{k+1}=B_k$. To avoid numerical errors, we apply the update only when
\begin{equation}\label{eq:structured_sr1_safeguard}
	\left|\inner{\tilde{y}_k-B_ks_k}{s_k}\right| > \kappa_{\mathrm{sds}}\left(1 + \|s_k\|\|\tilde{y}_k-B_ks_k\|\right),
\end{equation}
for $\kappa_{\mathrm{sds}} \in (0,1)$. Inequality~\eqref{eq:structured_sr1_safeguard} adds an absolute term to the most common safeguard used in SR1 methods~\cite[Section~6.2]{nocedalwright:2006}, so that the update is also skipped when $\|s_k\|\|\tilde{y}_k-B_ks_k\|$ is negligible.

The choice of the SR1 update formula is motivated by its robustness in the least-squares context, as the exhaustive benchmarks in~\cite{lucksan-etal:2019} show, and its tendency to better approximate the true Hessian with each iteration~\cite{conn-etal:1991a}. The SR1 approximation may, however, be indefinite, which can be seen as a weakness compared to BFGS or DFP updates, which are guaranteed to be positive definite as long as this is the case for the initial approximation $B_0$. However, as highlighted in Algorithm~\ref{algo:projected_cg_method}, the use of the conjugate gradient method in a trust-region framework handles, and even exploits, the potential nonconvexity of the subproblems. Moreover, this update does not require a curvature condition $\inner{s_k}{\tilde{y}_k} > 0$ to be satisfied. In other works on the constrained case, SQN methods are used to approximate the projected Hessian of the Lagrangian~\cite{tjoabiegler:1991} -- or of a merit function~\cite{amiribartels:1989} -- in the null space of the active constraints. The BFGS update is then preferred because this matrix is, under standard assumptions, positive semidefinite, according to the second-order necessary conditions.

We end our description of the Hessian approximation by discussing hybrid updates, a feature commonly used in SQN methods for nonlinear least squares, either in the unconstrained~\cite{dennisetal:1981,albaalifletcher:1985,fletcherxu:1987} or constrained~\cite{tjoabiegler:1991,li-etal:2002} case. The idea is to test at every iteration whether the problem has small or large residuals at the solution. In the latter case, a structured update is used to increase the accuracy of the approximated Hessian whereas in the former, the GN approximation is employed, as it is more likely to correspond to the true Hessian. Note that the update, such as~\eqref{eq:structured_sr1_update}, is still computed at every successful iteration to continue accumulating curvature information. We use a strategy inspired by~\cite{fletcherxu:1987} adapted to the constrained case in~\cite{tjoabiegler:1991} that is based on the asymptotic behavior of the ratio
\begin{equation*}
	\zeta_k = \dfrac{\varphi(x_{k})-\varphi(x_{k+1})}{\varphi(x_{k})}.
\end{equation*}
The approximated Hessian is then updated following the switching rule
\begin{equation}\label{eq:hybrid_update_strategy}
	H_{k+1} = \left\{\begin{aligned}
		&B^{\mathrm{GN}}_{k+1} + B_{k+1} & &\text{if } \rev{\varphi(x_k)-\varphi(x_{k+1}) \le \kappa_{\mathrm{hyb}}\,\varphi(x_k)} \\
		&B^{\mathrm{GN}}_{k+1}  & & \text{otherwise}, 
	\end{aligned}\right.
\end{equation}
where $\kappa_{\mathrm{hyb}}$ is a parameter in $(0,1)$. \rev{When $\varphi(x_k) > 0$, the test reads $\zeta_k \le \kappa_{\mathrm{hyb}}$; when $\varphi(x_k) \le 0$, which can happen because of the term $\inner{\lambda}{c(x)}$, any decrease of $\varphi$ selects the GN approximation.} In~\cite{tjoabiegler:1991}, the ratio is computed with the AL used as a merit function by their SQP method. As we will see in Section~\ref{sec:numerical_experiments}, this switching rule has an important impact on the performance of the algorithm.

	\subsection{Summary and implementation details}\label{subsec:implementation_details}

We end our discussion by summarizing the relations between outer, inner and minor iterates and the conditions on the parameters of the algorithm presented in this section. Recall that the outer iterates are indexed with $K$ and the inner iterates with $k$.
The outer iterates $x_K$ are the main iterates of the algorithm and are approximate minimizers of the AL. Each outer iterate is formed after a sequence of inner iterates $(x_k)_k$, linked by $x_{k+1}=x_k+s_k$, where $s_k$ is an approximate solution of the quadratic subproblem~\eqref{pb:quadratic_subproblem}. In turn, each inner iterate is formed after a sequence of $M$ minor iterations producing the minor iterates $x_{k,j}$, linked by $x_{k,j+1}=x_{k,j}+w_{k,j}$, where $w_{k,j}$ is a descent direction obtained by applying the projected conjugate gradient Algorithm~\ref{algo:projected_cg_method}. 
The relations between the three levels of iterations are illustrated in Figure~\ref{fig:iteration_hierarchy}.
\begin{figure}[ht]
	\centering
	\begin{tikzpicture}[
		>=Stealth,
		every node/.style={font=\small},
		outerbox/.style={draw=blue!60, fill=blue!5, rounded corners=8pt, thick, inner sep=10pt},
		innerbox/.style={draw=teal!70, fill=teal!5, rounded corners=6pt, thick, inner sep=8pt},
		minorbox/.style={draw=orange!70, fill=orange!5, rounded corners=5pt, thick, inner sep=6pt},
		stepbox/.style={draw=gray!60, fill=gray!8, rounded corners=3pt, inner sep=5pt, align=center, font=\footnotesize},
		]
		% Outer iteration frame
		\node[outerbox, minimum width=12.9cm, minimum height=10.5cm, label={[font=\small\bfseries, anchor=north]north:Outer iterations (Algorithm~\ref{algo:traulls})}] (outer) {};
		\node[font=\footnotesize, anchor=north] at ([yshift=-4mm]outer.north) {Outer iterate $x_K$: approx.\ minimize AL $\Phi$, update $\lambda_{K+1}$, $\mu_{K+1}$, tolerances};
		
		% Inner iteration frame
		\node[innerbox, minimum width=12cm, minimum height=8.2cm, label={[font=\small\bfseries, anchor=north]north:Inner iterations (Algorithm~\ref{algo:trinner_iteration})}] (inner) at ([yshift=-6mm]outer.center) {};
		\node[font=\footnotesize, anchor=north] at ([yshift=-4mm]inner.north) {Trust-region steps: $x_{k+1}=x_k+s_k$};
		
		% Left column: Cauchy + TR update + Hessian update + convergence check
		\node[stepbox, minimum width=3.8cm] (cauchy) at ([xshift=-2.2cm, yshift=1.8cm]inner.center) {Cauchy step\\[1pt]{\itshape Path on nested faces}};
		\node[stepbox, minimum width=3.8cm, below=8mm of cauchy] (trupdate) {Trust-region update\\[1pt]{\itshape Accept/reject via $\rho_k$}};
		\node[stepbox, minimum width=3.8cm, below=8mm of trupdate] (hessupdate) {Hessian update};
		\node[stepbox, minimum width=3.8cm, below=8mm of hessupdate] (convcheck) {$\pi_\varepsilon(x_k) \le \omega$\,?};
		
			% Minor iteration frame
			\node[minorbox, minimum width=4.5cm, minimum height=6.2cm, right=8mm of cauchy, anchor=north west, yshift=8mm, label={[font=\small\bfseries, anchor=north]north:Minor iterations}] (minor) {};
			\node[font=\footnotesize, anchor=north] at ([yshift=-4mm]minor.north) {$x_{k,j+1}=x_{k,j}+w_{k,j}$};
		
		% Inside minor: PCG + termination
		\node[stepbox, minimum width=3.5cm] (pcg) at ([yshift=6mm]minor.center) {Projected CG\\[1pt]{\itshape Algorithm~\ref{algo:projected_cg_method}}};
		\node[stepbox, minimum width=3.5cm, below=7mm of pcg] (boundhit) {Bound hit};
		\node[stepbox, minimum width=3.5cm, below=7mm of boundhit] (negcurv) {Neg.\ curvature / converged};
		
			% Arrows left column
			\draw[->, thick, gray] (cauchy) -- (trupdate);
			\draw[->, thick, gray] (trupdate) -- (hessupdate);
			\draw[->, thick, gray] (hessupdate) -- (convcheck);
		
			% Arrow Cauchy -> minor
			\draw[->, thick, gray] (cauchy.east) -- ([xshift=-2mm]pcg.west|-cauchy.east) -- ([xshift=-2mm]pcg.west) -- (pcg.west);
		
		% Arrows inside minor
		\draw[->, thick, gray] (pcg) -- node[right, font=\footnotesize]{or} (boundhit);
		\draw[->, thick, gray] (boundhit) -- node[right, font=\footnotesize]{or} (negcurv);
		
		% Restart loop for bound hit
		\draw[->, thick, gray, dashed] (boundhit.east) -- ++(6mm,0) |- node[right, font=\footnotesize, pos=0.25]{restart} (pcg.east);
		
		\end{tikzpicture}
		\caption{Hierarchy of the three iteration levels in the AL algorithm. The outer loop updates the Lagrange multipliers and penalty parameter. Each outer iteration triggers a sequence of trust-region inner iterations, and each inner iteration computes its step via a Cauchy point followed by minor iterations using the projected conjugate gradient method.}
		\label{fig:iteration_hierarchy}
	\end{figure}

	\paragraph*{Computing projections}

From our description of the step computation, we can see that whether during the Cauchy step calculation phase or during the minor iteration loop, the active set -- and thus the projection operator $\tilde{P}$ at~\eqref{eq:minor_subpb_projector} -- changes over the course of an iteration. It is therefore worth examining how this operator is updated. One can compute the projection of a vector $v$, i.e., $\tilde{P}v$, by first solving, for an auxiliary variable $y$, the linear system
\begin{equation}\label{eq:proj_normal_eq_auxiliary_syst}
	\left(\tilde{A}\tilde{A}^\top\right)y = \tilde{A}v,
\end{equation}
and then retrieving the projection as
\begin{equation}
	v - \tilde{A}^\top y.
\end{equation}
This approach, referred to as the \textit{normal equations} approach~\cite{gould-etal:2001}, requires solving the system~\eqref{eq:proj_normal_eq_auxiliary_syst}. This is done by using the Cholesky decomposition of $\tilde{A}\tilde{A}^\top = LL^\top$ -- with $L$ lower triangular -- and system~\eqref{eq:proj_normal_eq_auxiliary_syst} is thus reduced to two successive triangular systems, with $L$ and $L^\top$. The latter factorization is well defined since the matrix $\tilde{A}\tilde{A}^\top$ is symmetric by construction and positive definite as long as $\tilde{A}$ has full row rank, which Assumption~\ref{assumption:full_rank_active} and Lemma~\ref{lemma:eps_threshold} guarantee for the bounds that are active or $\varepsilon$-active at the current point, with $\varepsilon \le \bar\varepsilon$. A bound whose activation would make $\tilde{A}$ rank deficient, which can only happen when it is reached together with other bounds or because of the trust region, is detected by a small pivot in the bordered update described below and left out of the factor. Moreover, the block structure of $\tilde{A}\tilde{A}^\top$ can be exploited to incrementally maintain the $L$ factor -- instead of forming it from scratch -- across the changes in the active set. Given an active set $\calA \subset \{1,\ldots,n\}$ of cardinality $p < n-m$, and writing $A_\calA := AE_\calA$ for the restriction of $A$ to the columns indexed by $\calA$, $\tilde{A}\tilde{A}^\top $ can be written
\begin{equation}\label{eq:augmented_grammat}
	\begin{bmatrix}
		AA^\top &  A_{\calA} \\  A_{\calA}^\top & I_p
	\end{bmatrix}.
\end{equation}
From~\eqref{eq:augmented_grammat}, one can show~\cite[Section~4.2.9]{golubvanloan:2013} that the $L$ factor has the corresponding block pattern
\[L = \begin{bmatrix}
	L_{11} & 0 \\ L_{21} & L_{22}
\end{bmatrix},\]
where 
$L_{11}$ is the $m \times m$ Cholesky factor of $AA^\top$, $L_{21} := (L_{11}^{-1} A_{\calA})^\top$ is of size $p \times m$, and $L_{22}$ is the $p \times p$ Cholesky factor of $G := I_p - L_{21}L_{21}^\top$.
The matrix $A$ being fixed, forming the Cholesky factor $L_{11}$ of $AA^\top$ and $L_{11}^{-1} A$ only has to be done once at the start of the algorithm. Therefore, $L_{11}$ and $L_{21}$ -- the latter given by extracting the columns of $L_{11}^{-1} A$ indexed by $\calA$ -- are both available quantities and only the $L_{22}$ part needs to be maintained. Since adding an index to $\calA$ results in adding a column to $A_\calA$, and thus in appending a row and a column to $G$, the new factor $L_{22}$ can be updated incrementally by applying a bordered update to the existing one in $O(mp+p^2)$ flops, which is less costly than the $O((m+p)^3)$ time complexity required to form the entire $L$ factor from scratch.
\paragraph*{Bound-constrained problems}
Our intention with the work presented in this paper was to design an algorithm able to directly handle linear constraints of the form~\eqref{eq:linear_constraints} but we also accept problems where linear constraints are only bounds and then
\[\Omega = \left\{x \in \RR^n \ |\ l \le x \le u\right\}.\]
The framework we have presented remains valid for this formulation, with three practical simplifications in our implementation. First, the projection of a vector $x$ onto $\Omega$ has components
\begin{equation}\label{eq:projection_bounds}
	\left(\proj{\Omega}{x}\right)_i = 
	\left\{ \begin{aligned}
		& l_i & &\text{if } x_i < l_i \\
		& x_i & &\text{if } x_i \in [l_i,u_i] \\
		& u_i & &\text{if } x_i > u_i.
	\end{aligned} \right.
\end{equation}
Second, the projections onto tangent spaces are obtained by setting to $0$ the components corresponding to active bounds. Third, we take $\varepsilon = 0$, so that the Cauchy path of Section~\ref{subsec:step_computation} is the projected gradient path~\eqref{eq:projected_gradient_path}, and we use directly the norm of the projected gradient
\[\left\| x_K - \proj{\Omega}{x_K-\nabla_x\Phi_K}\right\|,\]
as criticality measure. This measure certifies the test~\eqref{eq:algo_traulls_criticality_test}: if $p = \proj{\Omega}{x_K-\nabla_x\Phi_K} - x_K$ satisfies $\|p\|_\infty \le \varepsilon \le \bar\varepsilon$, each component of the projection of $-\nabla_x\Phi_K$ onto $\calT_\varepsilon(x_K)$ is bounded in absolute value by the corresponding component of $p$, so that $\|p\| \le \omega_K$ implies~\eqref{eq:algo_traulls_criticality_test} as soon as $\omega_K \le \min(\varepsilon_{\max},\bar\varepsilon)$.

\paragraph*{Parameter conditions} 
Algorithm~\ref{algo:traulls} requires $\mu_0 > 1$, $\tau > 1$, positive tolerance constants $\omega,\eta,\kappa_\omega,\kappa_\eta,\beta_\omega,\beta_\eta$ and positive stopping tolerances $\omega_*,\eta_*$, the tolerance sequences being updated following~\cite[Algorithm~14.4.2]{conn-etal:2000}. The threshold of the $\varepsilon$-active sets is $\varepsilon_K = \min(\varepsilon_{\max},\omega_K)$ at outer iteration $K$, for a constant $\varepsilon_{\max} > 0$; the convergence analysis requires $\varepsilon_{\max} \le \bar\varepsilon$, a threshold given by Lemma~\ref{lemma:eps_threshold} that only depends on $\Omega$ but is not computable in practice. When the $\varepsilon$-active bounds and the rows of $A$ are linearly dependent, which can happen if Assumption~\ref{assumption:full_rank_active} does not hold or if $\varepsilon > \bar\varepsilon$, the multipliers of the decomposition~\eqref{eq:ls_decomposition} are not unique, and those of the least-squares solution may have the wrong sign at a first-order critical point. Our implementation then uses the multipliers that minimize the sum of the squared violations of their signs, which vanishes if and only if multipliers of the right sign exist. The constants driving the inner minimization of Algorithm~\ref{algo:trinner_iteration} are tied to the subproblem criticality tolerance, so that at outer iteration $K$ we set $\kappa_{\mathrm{cg}} = \kappa_{\mathrm{mlt}} = \omega_K$, while the safeguard constant of~\eqref{eq:structured_sr1_safeguard} is $\kappa_{\mathrm{sds}} = 10^{-8}$, the value suggested in~\cite[Section~6.2]{nocedalwright:2006}.

At the start of each outer iteration, the initial approximation of the second-order terms of the Hessian is $B_0 = 0$ for the SR1 update and, for the BFGS update, the identity, rescaled by $\inner{\tilde{y}_k}{s_k}/\inner{s_k}{s_k}$ at its first update, the BFGS update being skipped when $\inner{\tilde{y}_k}{s_k}$ is not sufficiently positive. The constant monitoring the hybrid switching scheme~\eqref{eq:hybrid_update_strategy} is $\kappa_{\mathrm{hyb}} = 0.1$, as in~\cite{tjoabiegler:1991}.

We now discuss the trust-region handling. The initial radius value for Algorithm~\ref{algo:trinner_iteration} is set to
\begin{equation*}
	\Delta_0 = \delta_0 \| g_0\|_\infty, \quad \delta_0 > 0.
\end{equation*}
The mechanism to update the trust region given in~\eqref{eq:tr_basic_update} still leaves important flexibility. We choose to follow a standard strategy similar to the one presented in~\cite[Chapter~17]{conn-etal:2000}. We also added a refinement to handle the case where the ratio $\rho_k$ is negative, which can happen when there is very poor agreement between the function and the model. In such cases, we reduce the radius more severely, to at most $\gamma_1\Delta_k$. The complete update rule for the trust-region radius is then
\begin{equation*}
	\Delta_{k+1} = \left\{\begin{aligned}
		& \max\left(\alpha_2 \|s_k\|_\infty,\Delta_k\right) & &\text{if } \rho_k \ge \eta_2 \\
		& \Delta_k &  &\text{if } \rho_k \in [\eta_1,\eta_2) \\
		& \alpha_1\|s_k\|_\infty & &\text{if } \rho_k \in [0,\eta_1) \\
		& \min\left(\alpha_1 \|s_k\|_\infty,\gamma_1\Delta_k\right) & &\text{if } \rho_k < 0,
	\end{aligned}\right.
\end{equation*}
whose constants satisfy $0 < \alpha_1 < 1 < \alpha_2$, $0 < \eta_1 \le \eta_2 < 1$ and $0 < \gamma_1 < 1$, a variant of the generic scheme~\eqref{eq:tr_basic_update} in which the radius is never reduced after a successful iteration, and may be reduced by more than the factor $\gamma_1$ after an unsuccessful one. The specific numerical values of all parameters used in our tests are reported in Section~\ref{sec:numerical_experiments}.

	\section{Convergence analysis}\label{sec:convergence_analysis}

In this section, we study the convergence of Algorithm~\ref{algo:traulls} towards a first-order critical point of problem~\eqref{pb:cnls}. We start by showing global convergence of Algorithm~\ref{algo:trinner_iteration}, implying that the inner minimization phase of Algorithm~\ref{algo:traulls} is always well defined. We then establish global convergence of the main algorithm. Our analysis of the outer iterations is based on the criticality test~\eqref{eq:algo_traulls_criticality_test} of~\cite{conn-etal:1996b}, which differs from the measure~\eqref{eq:criticality_measure_pi} used in our implementation. At the end of this section, we show that the latter certifies the former, so that our convergence results apply when~\eqref{eq:criticality_cond_reduced_grad} is used as an inner iteration termination condition instead of~\eqref{eq:algo_traulls_criticality_test}.

\subsection{Global convergence of the inner minimization}\label{subsec:inner_convergence}

In this section, $\lambda$ and $\mu$ are fixed, $\varphi(x) = \Phi(x,\lambda,\mu)$, and we write $\pi_\varepsilon(x) = \pi_\varepsilon(x,\nabla\varphi(x))$. We analyze Algorithm~\ref{algo:trinner_iteration} with the Cauchy step computed by the procedure of Appendix~\ref{appendix:cauchy_point_computation} for a fixed threshold $\varepsilon > 0$ of the $\varepsilon$-active sets, and with the stopping test $\pi_\varepsilon(x_k) \le \omega$.

When $\Omega$ is a box and $\varepsilon = 0$, as in our implementation for bound-constrained problems, the procedure follows the projected gradient path~\eqref{eq:projected_gradient_path} and the analysis of~\cite{conn-etal:1988a} applies. Over a general polyhedron, the active sets along $\proj{\Omega}{x-tg}-x$ need not be nested. Our procedure follows instead a path along which a bound is never released once it is reached, and the analysis below is carried out for that path. The nesting of the active sets, which does not hold for the projection, then holds by construction. The price to pay is that the decrease guaranteed by the Cauchy step can no longer be related to the projected gradient: it is related to the measure $\pi_\varepsilon$, and the threshold $\varepsilon$ is what keeps the first segment of the path from being arbitrarily short. Treating as active the bounds that are nearly active corresponds to the dominant constraints of~\cite{conn-etal:1996b} on which our criticality test is based.

We first make the following standard assumption.
\begin{assumption}\label{assumption:inner_iterates_compact}
	The set $\calX = \left\{x \ | \ \varphi(x) \le \varphi(x_0)\right\} \cap \Omega$ is nonempty and compact.
\end{assumption}
By Assumption~\ref{assumption:functions_C2}, $\varphi$ is twice continuously differentiable. We recall that the quadratic model is of the form
\[q_k(s) = \dfrac{1}{2} \inner{s}{H_ks} + \inner{g_k}{s},\]
where $g_k = \nabla \varphi(x_k)$ and $H_k$ is a symmetric approximation of the Hessian, for instance one of the structured approximations of Section~\ref{subsec:hessian_approx}. \rev{We assume that the approximation is not modified after an unsuccessful iteration, that is, $H_{k+1} = H_k$ when $\rho_k < \eta_1$. This is the case in our implementation: the approximations of Section~\ref{subsec:hessian_approx} are updated from the derivatives at the trial point $x_k+s_k$, which are only evaluated when the step is accepted.} The norm used is the induced norm on matrices, and we define the nondecreasing sequence
\[b_k = 1+\max_{0 \le i \le k} \ \left\| H_i\right\|.\]
We do not require the approximations to be bounded, but only that their norms do not grow too fast compared with the decrease of the function values.
\begin{assumption}\label{assumption:hessian_norm_compared_convergence_speed}
	\[\lim\limits_{k\to \infty} b_k (\varphi_{k+1}-\varphi_k) = 0.\]
\end{assumption}
This assumption holds, in particular, when the approximations $H_k$ are bounded, since the sequence $(\varphi_k)_k$ is then nonincreasing and bounded below on $\calX$.

From now on, we assume that the threshold of the $\varepsilon$-active sets satisfies $\varepsilon \in (0,\bar\varepsilon]$, where $\bar\varepsilon$ is given by Lemma~\ref{lemma:eps_threshold}, so that the properties stated in that lemma hold at every point of $\Omega$.

We now study the path along which the Cauchy step is computed, on which the decrease guarantee established below relies. Let $x \in \Omega$, $g = \nabla\varphi(x)$ and $\Delta > 0$. Because of the choice of the $\ell_\infty$ norm, the steps within the trust region are the vectors of
\begin{equation}\label{eq:proj_grad_feasible_set}
	\left\{d \in \RR^n \ | \ Ad=0,\ l^\Delta \le d \le u^\Delta\right\},
\end{equation}
with $l^\Delta_i = \max\left(-\Delta,l_i-x_i\right)$ and $u^\Delta_i = \min\left(\Delta,u_i-x_i\right)$ for all $i=1,\ldots,n$. Let $\calA_\varepsilon = \calA_\varepsilon(x)$ and $\mathcal{V} = \{i \in \calA_\varepsilon \mid \theta^{\mathrm{w}}_i > 0\}$ be the set of $\varepsilon$-active bounds whose multiplier estimate has the wrong sign. Here $\theta$ and $\theta^{\mathrm{w}}$ come from the decomposition~\eqref{eq:ls_decomposition} of $g$ computed with the $\varepsilon$-active bounds, which is unique by Lemma~\ref{lemma:eps_threshold}(ii). For $i \in \calA_\varepsilon^-(x)\setminus\calA_\varepsilon^+(x)$ (resp. $\calA_\varepsilon^+(x)\setminus\calA_\varepsilon^-(x)$), we write $n_i = e_i$ (resp. $n_i = -e_i$) for the inward normal of the $\varepsilon$-active bound, so that $\theta_ie_i = -\theta^{\mathrm{w}}_in_i$ for every $i \in \mathcal{V}$.

The first segment of the path depends on which $\varepsilon$-active bounds are kept fixed. Keeping all of them would confine the path to $T_\varepsilon(x)$, and the path could not exploit the bounds of $\mathcal{V}$, whose multipliers indicate that the model can be decreased by moving away from them. These bounds are therefore candidates for release. Releasing the bounds of a set $F \subseteq \mathcal{V}$, while keeping the other $\varepsilon$-active bounds fixed, restricts the path to the subspace
\[T_F = \left\{v \mid Av = 0,\ v_i = 0 \text{ for } i \in \calA_\varepsilon\setminus F\right\},\]
which grows with $F$, from $T_\emptyset = T_\varepsilon(x)$ to $T_{\mathcal{V}}$, and in which the steepest descent direction is $\proj{T_F}{-g}$. As in classical active-set methods, releasing a single bound of $\mathcal{V}$ gives a direction that moves away from it, as follows from~\eqref{eq:gram_expression} below. When several bounds are released at once, however, the projection onto $T_F$ couples them through the linear equalities, and $\proj{T_F}{-g}$ may point outwards at some of them. The path would then head towards bounds that lie within distance $\varepsilon$ of $x$, possibly at distance zero, and its first segment could be arbitrarily short. The procedure of Appendix~\ref{appendix:cauchy_point_computation} therefore first releases all the bounds of $\mathcal{V}$, by setting $F_0 = \mathcal{V}$, and then fixes again those at which the direction points outwards: for $j = 0,1,\ldots$,
\begin{equation}\label{eq:refixing}
	F_{j+1} = F_j \setminus \left\{ i \in F_j \mid \inner{n_i}{\proj{T_{F_j}}{-g}} < 0 \right\}.
\end{equation}
The multiplier estimates, and hence $\mathcal{V}$, are not recomputed during this iteration. Since the sets $F_j$ are nested and decrease strictly as long as the iteration does not stop, it stops after at most $|\mathcal{V}|$ removals, with a set $F_{j+1} = F_j =: F$ at which $\proj{T_F}{-g}$ points inwards, or is tangent, at every released bound. The variables of $\calA(0) := \calA_\varepsilon \setminus F$ are fixed at their current values, and $z_0 = \proj{T_F}{g}$. The path is then defined by $s(0) = 0$ and
\[s(t) = s(t_i) - (t-t_i)z_i \quad \text{for } t \in [t_i,t_{i+1}),\]
where
\begin{equation}\label{eq:reduced_gradient_proj_grad_path}
	z_i = \proj{T_i}{g}, \qquad T_i = \left\{v \mid Av = 0,\ v_j = 0 \text{ for } j \in \calA(t_i)\right\},
\end{equation}
$t_{i+1}$ is the first value of $t$ at which a free component reaches its bound in~\eqref{eq:proj_grad_feasible_set}, and $\calA(t_{i+1})$ is obtained by adding the corresponding indices to $\calA(t_i)$. The Cauchy step $s^C = s(t^C)$ is obtained at the first value $t^C$ at which $\psi(t) = q(s(t))$ satisfies $\psi'(t^C_+) \ge 0$, or at which $z_i = 0$. The minor iterations of Section~\ref{subsec:step_computation} then provide a step $s_k$ that remains in~\eqref{eq:proj_grad_feasible_set} and satisfies
\begin{equation}\label{eq:step_drecrease_wrt_cauchy}
	q_k(s_k) \le q_k(s_k^C).
\end{equation}

\begin{proposition}\label{prop:computed_path}
	Suppose that Assumption~\ref{assumption:full_rank_active} holds. The path $s(t)$ has the following properties.
	\begin{enumerate}[label=(\roman*), leftmargin=*, itemsep=2pt]
		\item It is piecewise linear, lies in~\eqref{eq:proj_grad_feasible_set}, and has finitely many breakpoints.
		\item The working sets are nested, $\calA(t) \subseteq \calA(t')$ for $0 < t \le t'$, so that $T_{i+1} \subseteq T_i$ and $\|z_{i+1}\| \le \|z_i\|$.
		\item $s(t) = -\int_0^t z(\tau)\,d\tau$, $\inner{g}{z_i} = \|z_i\|^2$ and $\|s(t)\| \le t\|z_0\| \le t\|g\|$.
		\item $-z_0$ is a feasible descent direction at $x$: $Az_0 = 0$, $(z_0)_i = 0$ for $i \in \calA(0)$, and $\inner{n_i}{-z_0} \ge 0$ for $i \in F$.
	\end{enumerate}
\end{proposition}
\begin{proof}
	The nesting (ii) holds by construction, since a bound added to the working set is never released along the path; projecting $g$ onto a smaller subspace reduces the norm of the projection. Each segment of positive length ends by adding at least one index to the working set, so that there are at most $n$ of them. Statements (i) and (iii) follow, with $\inner{g}{z_i} = \|z_i\|^2$ because $z_i$ is an orthogonal projection of $g$. Each segment remains in~\eqref{eq:proj_grad_feasible_set} up to the next breakpoint by definition of the latter. Statement (iv) is the termination condition of~\eqref{eq:refixing}.
\end{proof}

The next lemma is a key result that compares the norm of the initial reduced gradient to the criticality measure $\pi_{\varepsilon}$.
\begin{lemma}\label{lemma:initial_direction}
	Suppose that Assumption~\ref{assumption:full_rank_active} holds. There is a constant $\kappa_0 \in (0,1]$, depending only on $A$ and on the bounds, such that, for all $x \in \Omega$ and $\varepsilon \in (0,\bar\varepsilon]$,
	\[\|z_0\| \ge \kappa_0\,\pi_\varepsilon(x).\]
\end{lemma}
\begin{proof}
	If $\mathcal{V} = \emptyset$, then $F = \emptyset$, $z_0 = \proj{T_\varepsilon(x)}{g}$ and $\theta^{\mathrm{w}} = 0$, so that $\|z_0\| \ge \pi_\varepsilon(x)$ and the result holds with $\kappa_0 = 1$. We assume from now on that $\mathcal{V} \ne \emptyset$, and we let $\tilde{g} = \proj{T_\varepsilon(x)}{g}$. We collect the sign violations of the bounds of $\mathcal{V}$ in the vector $\vartheta \in (0,\infty)^{\mathcal{V}}$ with components $\vartheta_i = \theta^{\mathrm{w}}_i$, $i \in \mathcal{V}$. For $F' \subseteq \mathcal{V}$, we denote by $\vartheta_{F'} \in (0,\infty)^{F'}$ the subvector of $\vartheta$ formed by the components $\vartheta_i$, $i \in F'$, that is, the violations carried by the bounds released in $F'$. Since $\theta^{\mathrm{w}}$ vanishes on $\calA_\varepsilon\setminus\mathcal{V}$, we have $\|\vartheta\| = \|\theta^{\mathrm{w}}\|$. We first express $\proj{T_{F'}}{g}$ by means of a Gram matrix, then show that the final set $F$ of~\eqref{eq:refixing} retains a uniform fraction of $\vartheta$, and finally conclude.

	Let $F' \subseteq \mathcal{V}$. The decomposition~\eqref{eq:ls_decomposition} of $g$ with the $\varepsilon$-active bounds reads
	\[g = \tilde{g} + A^\top\xi + \sum_{i\in\calA_\varepsilon}\theta_ie_i,\]
	and we project it onto $T_{F'}$. The rows of $A$ and the vectors $e_i$, $i \in \calA_\varepsilon\setminus F'$, are orthogonal to $T_{F'}$, while $\proj{T_{F'}}{\tilde{g}} = \tilde{g}$ since $\tilde{g} \in T_\varepsilon(x) \subseteq T_{F'}$. Writing $\theta_ie_i = -\vartheta_in_i$ for $i \in F'$ and denoting by $p_i = \proj{T_{F'}}{n_i}$ the projected inward normals, we obtain
	\[\proj{T_{F'}}{g} = \tilde{g} - \sum_{i\in F'}\vartheta_i\,p_i.\]
	The vector $\tilde{g}$ is orthogonal to every $p_i$, since $\inner{\tilde{g}}{p_i} = \inner{\proj{T_{F'}}{\tilde{g}}}{n_i} = \inner{\tilde{g}}{n_i} = \pm\tilde{g}_i = 0$, $\tilde{g}$ vanishing on $\calA_\varepsilon$. Let $G(F') \in \RR^{F'\times F'}$ be the Gram matrix of the projected normals, with entries
	\[G(F')_{ik} = \inner{p_i}{p_k} = \inner{n_i}{\proj{T_{F'}}{n_k}}, \qquad i,k \in F'.\]
	Expanding the squared norm of $\proj{T_{F'}}{g}$ gives the first identity below. The second one follows from $\inner{n_i}{\proj{T_{F'}}{-g}} = \inner{p_i}{-\proj{T_{F'}}{g}}$, which holds because the projection is self-adjoint and $p_i \in T_{F'}$:
	\begin{equation}\label{eq:gram_expression}
		\left\|\proj{T_{F'}}{g}\right\|^2 = \|\tilde{g}\|^2 + \vartheta_{F'}^\top G(F')\,\vartheta_{F'}, \qquad \inner{n_i}{\proj{T_{F'}}{-g}} = \left(G(F')\,\vartheta_{F'}\right)_i \quad (i \in F').
	\end{equation}
	The second identity gives the inward component, at each released bound, of the steepest descent direction in $T_{F'}$. Since $n_i = \pm e_i$, the matrix $G(F')$ is, up to the signs of its rows and columns, the principal submatrix indexed by $F'$ of the matrix of the orthogonal projection onto $T_{F'}$. It depends only on $A$, on the $\varepsilon$-active bounds and on $F'$, and not on $g$. Its off-diagonal entries describe how the linear equalities couple the released bounds: when $\Omega$ is a box, $p_i = n_i$ and $G(F') = I$. The matrix $G(F')$ is positive definite. It is positive semidefinite as a Gram matrix, and the vectors $p_i$, $i \in F'$, are linearly independent. Indeed, if $\sum_{i\in F'}\chi_ip_i = 0$, then $\sum_{i\in F'}\chi_in_i$ has a zero projection onto $T_{F'}$, and thus belongs to $T_{F'}^\perp$, which is spanned by the rows of $A$ and the $e_j$, $j \in \calA_\varepsilon\setminus F'$. Since $n_i = \pm e_i$, this is a linear relation between the rows of $A$ and the vectors $e_j$, $j \in \calA_\varepsilon$, so that $\chi = 0$ by Lemma~\ref{lemma:eps_threshold}(ii). In particular, for $F' = \{i\}$, the inward component $G(F')_{ii}\,\vartheta_i$ is positive, which is the property of single releases mentioned before~\eqref{eq:refixing}.

	By the second identity of~\eqref{eq:gram_expression}, the iteration~\eqref{eq:refixing} reads
	\[F_0 = \mathcal{V}, \qquad F_{j+1} = F_j \setminus \left\{i \in F_j \mid \left(G(F_j)\,\vartheta_{F_j}\right)_i < 0\right\}.\]
	In this form, it involves $g$ only through $\vartheta$, the matrices $G(F_j)$ being determined by $A$, by the $\varepsilon$-active bounds and by $F_j$. For fixed $\varepsilon$-active bounds and $\mathcal{V}$, we may thus run it from any vector $\vartheta \in (0,\infty)^{\mathcal{V}}$, and we denote by $F(\vartheta)$ the resulting final set. We show that there is $\kappa_1 > 0$ such that
	\begin{equation}\label{eq:retained_violation}
		\left\|\vartheta_{F(\vartheta)}\right\| \ge \kappa_1\,\|\vartheta\| \qquad \text{for all } \vartheta \in (0,\infty)^{\mathcal{V}}.
	\end{equation}
	Two reductions simplify the argument. First, multiplying $\vartheta$ by a positive scalar does not change the signs tested by the iteration, so that $F(t\vartheta) = F(\vartheta)$ for all $t > 0$, and it suffices to consider vectors with $\|\vartheta\| = 1$. Second, the iteration produces a chain of sets $\mathcal{V} = F_0 \supsetneq F_1 \supsetneq \cdots \supsetneq F_J = F(\vartheta)$, and there are finitely many such chains of subsets of $\mathcal{V}$. It is therefore enough to prove~\eqref{eq:retained_violation} for the vectors $\vartheta$ that produce a given chain, $\kappa_1$ being then the smallest of the finitely many constants obtained.

	Fix such a chain, with last set $F$, and assume for contradiction that there are vectors $\vartheta^\ell \in (0,\infty)^{\mathcal{V}}$ producing it, with $\|\vartheta^\ell\| = 1$ and $\vartheta^\ell_F \to 0$. By compactness of the unit sphere, a subsequence of $(\vartheta^\ell)_\ell$ converges to some $\bar\vartheta$ such that $\bar\vartheta \ge 0$, $\|\bar\vartheta\| = 1$ and $\bar\vartheta_F = 0$. We show that $\bar\vartheta$ vanishes on every set of the chain, which contradicts $\|\bar\vartheta\| = 1$ since $F_0 = \mathcal{V}$. The mechanism is the following. By~\eqref{eq:gram_expression}, a released bound $i$ is fixed again only if its inward component $G_{ii}\vartheta_i + \sum_{k\ne i}G_{ik}\vartheta_k$ is \rev{negative, which,} the diagonal term being positive, requires a coupling with other released bounds carrying positive weights. In the limit, the bounds fixed again at the last iteration can only be coupled with each other, since those of $F$ carry no weight, and positive definiteness then forces their weights to vanish. Repeating this argument backwards along the chain gives the result.

	Formally, we proceed by backward induction on $j$, the property holding for $j = J$ since $F_J = F$. Assume that $\bar\vartheta$ vanishes on $F_{j+1}$, and let $R = F_j\setminus F_{j+1}$ be the set of bounds fixed again at iteration $j$. For $i \in R$, $(G(F_j)\,\vartheta^\ell_{F_j})_i < 0$ for all $\ell$, hence $(G(F_j)\,\bar\vartheta_{F_j})_i \le 0$ in the limit. Since $\bar\vartheta_{F_j}$ vanishes outside $R$, this reads $(G_{RR}\,\bar\vartheta_R)_i \le 0$, where $G_{RR}$ is the principal submatrix of $G(F_j)$ indexed by $R$. Multiplying by $\bar\vartheta_i \ge 0$ and summing over $i \in R$ yields $\bar\vartheta_R^\top G_{RR}\,\bar\vartheta_R \le 0$. As a principal submatrix of a positive definite matrix, $G_{RR}$ is positive definite, so that $\bar\vartheta_R = 0$, and $\bar\vartheta$ vanishes on $F_j = F_{j+1}\cup R$.

	We can now conclude. Applying~\eqref{eq:retained_violation} to the vector $\vartheta$ of the sign violations at $x$, whose final set is $F$, gives $\|\vartheta_F\| \ge \kappa_1\|\theta^{\mathrm{w}}\|$, and in particular $F \ne \emptyset$. Since $z_0 = \proj{T_F}{g}$, the first identity of~\eqref{eq:gram_expression} with $F' = F$ yields
	\begin{align*}
		\|z_0\|^2 = \|\tilde{g}\|^2 + \vartheta_F^\top G(F)\,\vartheta_F &\ge \|\tilde{g}\|^2 + \lambda_{\min}(G(F))\,\kappa_1^2\,\|\theta^{\mathrm{w}}\|^2 \\
		&\ge \kappa_0^2 \max\left(\|\tilde{g}\|_\infty,\|\theta^{\mathrm{w}}\|_\infty\right)^2 = \kappa_0^2\,\pi_\varepsilon(x)^2,
	\end{align*}
	with $\kappa_0 = \min\left(1, \kappa_1\sqrt{\lambda_{\min}(G(F))}\right)$. The constants $\kappa_1$ and $\lambda_{\min}(G(F))$ only depend on $A$ and on the sets of $\varepsilon$-active bounds, $\mathcal{V}$ and $F$, which take finitely many values as $x$ ranges over $\Omega$ and $\varepsilon$ over $(0,\bar\varepsilon]$. Taking the minimum of $\kappa_0$ over these values gives a constant that does not depend on $x$ and $\varepsilon$.
\end{proof}
Releasing all the bounds of the wrong sign at once, and then fixing again those along which the direction points outwards, is therefore safe: the final set $F$ is nonempty as soon as $\mathcal{V}$ is, and it retains a uniform fraction of the violation of the sign conditions.

We can now establish an inequality on the guaranteed decrease.
\begin{theorem}\label{theo:cauchy_decrease}
	Suppose that Assumption~\ref{assumption:full_rank_active} holds and that $\varepsilon \in (0,\bar\varepsilon]$. Then, at every iteration $k$,
	\begin{equation}\label{eq:bound_step_decrease}
		-q_k(s_k) \ge \frac{\kappa_0\,\pi_\varepsilon(x_k)}{2}\min\left(\varepsilon,\ \Delta_k,\ \frac{\kappa_0\,\pi_\varepsilon(x_k)}{2(1+\|H_k\|)}\right).
	\end{equation}
\end{theorem}
\begin{proof}
	If $z_0 = 0$, then $\pi_\varepsilon(x_k) = 0$ by Lemma~\ref{lemma:initial_direction}, and there is nothing to prove. Otherwise, let $t_1 > 0$ be the first positive breakpoint. On $[0,t_1)$, $s(t) = -tz_0$, so that
	\[|s_j(t)| \le t\|z_0\| \quad \text{for every component } j.\]
	Along this first segment, the components fixed in $\calA(0)$ do not move, and a free component that is not $\varepsilon$-active lies at a distance larger than $\varepsilon$ from its bounds. A released component moves inwards by Proposition~\ref{prop:computed_path}(iv), that is, away from its $\varepsilon$-active bound and towards its other bound, which lies at a distance at least $3\varepsilon$ by Lemma~\ref{lemma:eps_threshold}(iii). Finally, the trust-region bounds lie at distance $\Delta_k$. Hence
	\[t_1 \ge \frac{\min(\varepsilon,\Delta_k)}{\|z_0\|}.\]

	On $[0,t_1)$, the model along the path is
	\[\psi(t) = -t\|z_0\|^2 + \frac{t^2}{2}\inner{z_0}{H_kz_0},\]
	so that, for $t \le \hat{t} := \min\left(t_1, \dfrac{1}{2(1+\|H_k\|)}\right)$,
	\[\psi'(t) \le -\|z_0\|^2\left(1 - t\|H_k\|\right) \le -\frac{1}{2}\|z_0\|^2.\]
	No termination condition of the path can thus be met before $\hat{t}$, so that $t^C \ge \hat{t}$. Since $\psi$ is decreasing on $[0,t^C]$, we obtain
	\[\psi(t^C) \le \psi(\hat{t}) \le -\frac{1}{2}\|z_0\|^2\,\hat{t},\]
	which, with the lower bound on $t_1$, gives
	\[-q_k(s_k^C) \ge \min\left(\frac{\|z_0\|}{2}\min(\varepsilon,\Delta_k),\ \frac{\|z_0\|^2}{4(1+\|H_k\|)}\right).\]
	We conclude with Lemma~\ref{lemma:initial_direction} and~\eqref{eq:step_drecrease_wrt_cauchy}.
\end{proof}

The following theorem states the convergence of the inner \rev{iteration} procedure.
\begin{theorem}\label{theo:inner_convergence}
	Suppose that Assumptions~\ref{assumption:functions_C2}, \ref{assumption:full_rank_active}, \ref{assumption:inner_iterates_compact} and~\ref{assumption:hessian_norm_compared_convergence_speed} hold, and that $\varepsilon \in (0,\bar\varepsilon]$. If Algorithm~\ref{algo:trinner_iteration}\rev{, with the radius update rule of Section~\ref{subsec:implementation_details},} generates an infinite sequence of iterates, then
	\[\liminf_{k\to\infty} \pi_\varepsilon(x_k) = 0.\]
	Consequently, for every $\omega > 0$, the stopping test $\pi_\varepsilon(x_k) \le \omega$ is satisfied after finitely many iterations.
\end{theorem}
\begin{proof}
	Assume, for contradiction, that $\pi_\varepsilon(x_k) \ge \underline{\pi} > 0$ for all $k \ge k_0$. Since $1+\|H_k\| \le b_k$ and $b_k \ge 1$, Theorem~\ref{theo:cauchy_decrease} gives, for all $k \ge k_0$,
	\begin{equation}\label{eq:decrease_bk}
		-q_k(s_k) \ge c\,\min\left(\Delta_k, \frac{\zeta}{b_k}\right), \qquad c = \frac{\kappa_0\underline{\pi}}{2}, \quad \zeta = \min\left(\varepsilon, \frac{\kappa_0\underline{\pi}}{2}\right).
	\end{equation}
	Let $L$ bound $\|\nabla^2\varphi\|$ on the compact set $\{x+s \mid x \in \calX,\ \|s\|_\infty \le 1\}$, and let $L' = n(L+1)/2$. If $\|s_k\|_\infty \le 1$, Taylor's theorem gives
	\[\left|\varphi(x_k+s_k) - \varphi(x_k) - q_k(s_k)\right| \le \tfrac{1}{2}\left(L + \|H_k\|\right)\|s_k\|^2 \le L'b_k\|s_k\|_\infty^2.\]
	We also define the constants
	\[c_1 = \min\left(1,\ \zeta,\ \frac{(1-\eta_1)c}{2L'}\right), \qquad c_2 = \min\left(c_1,\ \sqrt{\frac{(1-\eta_1)c\,c_1}{L'}}\right).\]
	At an unsuccessful iteration, $\rho_k < \eta_1$ implies
	\[(1-\eta_1)\left(-q_k(s_k)\right) < \varphi(x_k+s_k) - \varphi(x_k) - q_k(s_k).\]

	First, if iteration $k \ge k_0$ is unsuccessful, then
	\[\Delta_k > \frac{c_1}{b_k} \qquad \text{and} \qquad \|s_k\|_\infty \ge \frac{c_2}{b_k}.\]
	Indeed, if $\Delta_k \le c_1/b_k$, then $\|s_k\|_\infty \le \Delta_k \le 1$ and $\min(\Delta_k,\zeta/b_k) = \Delta_k$, so that
	\[(1-\eta_1)\,c\,\Delta_k < L'b_k\Delta_k^2, \qquad \text{that is,} \qquad \Delta_k > \frac{(1-\eta_1)c}{L'b_k} \ge \frac{2c_1}{b_k},\]
	a contradiction. Next, if $\|s_k\|_\infty \ge 1$, then $\|s_k\|_\infty \ge c_2/b_k$. Otherwise, since $\min(\Delta_k,\zeta/b_k) \ge c_1/b_k$, we obtain
	\[L'b_k\|s_k\|_\infty^2 > \frac{(1-\eta_1)c\,c_1}{b_k}, \qquad \text{hence} \qquad \|s_k\|_\infty \ge \frac{c_2}{b_k}.\]

	Second, we show that
	\[\Delta_k \ge \frac{c_3}{b_k} \quad \text{for all } k \ge k_0,\]
	with $c_3 = \min\left(b_{k_0}\Delta_{k_0},\ \alpha_1c_2,\ \gamma_1c_1\right)$.
	From now on, the radius is updated by the rule used in our implementation of Algorithm~\ref{algo:trinner_iteration}, given in Section~\ref{subsec:implementation_details}. It is a variant of the generic scheme~\eqref{eq:tr_basic_update} in which the radius is never reduced after a successful iteration. The property holds at $k_0$. If it holds at a successful iteration $k$, then
	\[\Delta_{k+1} \ge \Delta_k \ge \frac{c_3}{b_k} \ge \frac{c_3}{b_{k+1}}.\]
	At an unsuccessful iteration, this rule and the first step give
	\[\Delta_{k+1} \ge \min\left(\alpha_1\|s_k\|_\infty,\ \gamma_1\Delta_k\right) \ge \frac{\min(\alpha_1c_2,\gamma_1c_1)}{b_k},\]
	and, since the Hessian approximation is only updated after successful iterations, $H_{k+1} = H_k$, so that $b_{k+1} = b_k$.

	Finally, if there are finitely many successful iterations, then $x_k$, $H_k$ and $b_k = \bar{b}$ are eventually constant, and $\Delta_{k+1} \le \alpha_1\Delta_k$ for all large $k$ by the same rule, which contradicts $\Delta_k \ge c_3/\bar{b}$. Otherwise, by~\eqref{eq:decrease_bk} and the second step, every successful iteration $k \ge k_0$ satisfies
	\[b_k\left(\varphi_k - \varphi_{k+1}\right) \ge b_k\,\eta_1\left(-q_k(s_k)\right) \ge \eta_1\,c\,\min(c_3,\zeta) > 0,\]
	which contradicts Assumption~\ref{assumption:hessian_norm_compared_convergence_speed} since there are infinitely many such iterations.
\end{proof}
Theorem~\ref{theo:inner_convergence} only establishes that the measure $\pi_\varepsilon$ becomes small along a subsequence, which is what the termination of the inner iterations requires. Unlike the projected gradient, $\pi_\varepsilon$ is not continuous with respect to $x$, and we do not claim that $\pi_\varepsilon(x_k) \to 0$. The link between $\pi_\varepsilon$ and the criticality test~\eqref{eq:algo_traulls_criticality_test} used in the analysis of the outer iterations is established in Lemma~\ref{lemma:certification}.

\subsection{Global convergence of the algorithm}\label{subsec:outer_convergence}

We now establish that every limit point of the outer iteration sequence is a first-order critical point of~\eqref{pb:cnls}, for a feasible set combining nonlinear equalities with a general polyhedron. Our problem is a particular case of the one considered in~\cite{conn-etal:1996b}, whose analysis also covers linear equalities alongside the linear inequalities, as noted there, and our criticality test~\eqref{eq:algo_traulls_criticality_test} is theirs. We restate the main steps of their argument in our notation, for completeness and because the inner iteration of Section~\ref{subsec:inner_convergence} is designed to satisfy this test. 
We first make the following assumption about the iterates considered.

	\begin{assumption}\label{assumption:iterates_domain_bounded}
		The iterates produced by Algorithm~\ref{algo:traulls} lie within a closed bounded domain of $\RR^n$.
	\end{assumption}

Let $(x_K)_K$ be a sequence of such iterates. Assumption~\ref{assumption:iterates_domain_bounded} then ensures the existence of a subsequence indexed by $\calK \subseteq \NN$ converging to a limit point $x_*$. Since the iterates satisfy $x_K \in \Omega$ for all $K$, we also have $x_* \in \Omega$. We recall from~\eqref{eq:tangent_space} that $\calA(x)$ is the set of bounds active at $x$ and $T(x)$ the associated tangent space, and we abbreviate $\calA_* := \calA(x_*)$. We denote by $\tilde A_*$ the matrix gathering the linear constraints active at $x_*$,
\begin{equation*}
	\tilde A_* := \begin{pmatrix} A \\ E^\top_* \end{pmatrix},
\end{equation*}
where $E_*$ is the submatrix of the identity whose columns are $e_i$, $i \in \calA_*$. By Assumption~\ref{assumption:full_rank_active} at $x_*$, $\tilde A_*$ has full row rank, so $T(x_*) = \mathrm{null}(\tilde A_*)$ and there is an orthonormal matrix $N_*$ whose columns form a basis of $T(x_*)$. The next assumption ensures that the null space of the active linear constraints is large enough to achieve feasibility of the nonlinear constraints, and that the gradients of those constraints are linearly independent within that space.

	\begin{assumption}\label{assumption:cq_null_space_jacobian}
		The rank of the matrix $C(x_*)N_*$ is no smaller than $n_c$ at any limit point $x_*$ of the sequence $(x_K)_K$.
	\end{assumption}

	We now introduce the remaining notation and notions used in our convergence analysis.

We recall that the normal cone $\calN_\Omega(x)$ to $\Omega$ at $x \in \Omega$ is given by~\eqref{eq:normal_cone_eps} with $\varepsilon = 0$, the equality multipliers $\xi$ being free in sign while the bound multipliers are signed. The first-order criticality condition $\proj{\Omega}{x_*-\nabla_x\calL(x_*,\lambda_*)}=x_*$ of problem~\eqref{pb:cnls} is then equivalent to
\[-\nabla_x\calL(x_*,\lambda_*)\in \calN_\Omega(x_*).\]

Finally, for vectors $x$ at which the matrix $C(x)N_* N_*^\top C(x)^\top$ is nonsingular, the least-squares multipliers associated with $\tilde A_*$ are defined by
\begin{equation}\label{eq:least-squares_mult}
	\lambda(x) = -\big((C(x)N_*)^\dagger\big)^\top N_*^\top\nabla f(x),
\end{equation}
where $(C(x)N_*)^\dagger = N_*^\top C(x)^\top\big(C(x)N_* N_*^\top C(x)^\top\big)^{-1}$ is the right pseudoinverse of $C(x)N_*$. When $(C(x)N_*)^\dagger $ is well defined, Assumption~\ref{assumption:functions_C2} ensures that $\lambda(x)$ is differentiable and that its Jacobian $\nabla\lambda(x)$ is bounded in a neighborhood of $x_*$ (see~\cite[Lemma~2.1]{conn-etal:1996b}).

We can now proceed with the global convergence proof. The first lemma is the counterpart of~\cite[Lemma~4.1]{conn-etal:1996b}. It shows that the criticality test~\eqref{eq:algo_traulls_criticality_test} controls the reduced gradient at the limit point.
\begin{lemma}\label{lemma:bound_proj_algrad_sol_nullspace}
	Let $\left(x_K\right)_{K\in \calK} \subseteq \Omega$ be a sequence converging to $x_*$ and such that
	\[\left\| \proj{\calT_{\varepsilon_K}(x_K)}{-\nabla_x \Phi_K} \right\| \le \omega_K,\]
	with $0 \le \varepsilon_K \le \omega_K$ and $\lim\limits_{K\to \infty} \omega_K = 0$. Then
	\begin{equation*}
		\left\| \proj{T(x_*)}{\nabla_x \Phi_K} \right\| \le  \omega_K,
	\end{equation*}
	for all $K\in \calK$ large enough.
\end{lemma}

\begin{proof}
	To lighten the notation, we write $\calT_K = \calT_{\varepsilon_K}(x_K)$ and $\calN_K = \calN_{\varepsilon_K}(x_K)$, and $T_K = T_{\varepsilon_K}(x_K)$ for the subspace defined as in~\eqref{eq:tangent_space} with the $\varepsilon_K$-active bounds at $x_K$. We first show that, for $K \in \calK$ large enough, the bounds that are $\varepsilon_K$-active at $x_K$ are active at $x_*$, on the same side, that is,
	\[\calA^-_{\varepsilon_K}(x_K) \subseteq \calA^-(x_*) \qquad \text{and} \qquad \calA^+_{\varepsilon_K}(x_K) \subseteq \calA^+(x_*),\]
	so that 
	\[\calA_{\varepsilon_K}(x_K) \subseteq \calA_*.\]
	Let $i \notin \calA^-(x_*)$, so that $(x_*)_i - l_i > 0$. Since $x_K \to x_*$ and $\varepsilon_K \to 0$, we have $(x_K)_i - l_i > \varepsilon_K$, that is, $i \notin \calA^-_{\varepsilon_K}(x_K)$, for all $K \in \calK$ large enough. The same argument applies to the upper bounds, and the inclusions hold simultaneously for all the bounds for $K \in \calK$ large enough, since there are finitely many of them.

	These inclusions carry over to the subspaces defined by the bounds. Every $d \in T(x_*)$ satisfies $Ad = 0$ and $d_i = 0$ for $i \in \calA_*$, hence for $i \in \calA_{\varepsilon_K}(x_K)$, so that $T(x_*) \subseteq T_K$. Moreover, the components of every $d \in T_K$ indexed by $\calA_{\varepsilon_K}(x_K)$ vanish, so that $d$ satisfies the sign conditions defining the cone $\calT_K$ in~\eqref{eq:tangent_cone_eps}. Hence, for $K \in \calK$ large enough,
	\[T(x_*) \subseteq T_K \subseteq \calT_K.\]

	Since $T_K$ is a subspace contained in $\calT_K$, every $w \in \calN_K$ satisfies $\inner{w}{v} \le 0$ and $\inner{w}{-v} \le 0$ for all $v \in T_K$, that is, $\calN_K$ is orthogonal to $T_K$. The Moreau decomposition~\eqref{eq:moreau_decomposition} of $-\nabla_x\Phi_K$ then gives $\proj{T_K}{-\nabla_x\Phi_K} = \proj{T_K}{\proj{\calT_K}{-\nabla_x\Phi_K}}$, so that
	\begin{equation*}
		\left\| \proj{T(x_*)}{\nabla_x \Phi_K} \right\| \le \left\| \proj{T_K}{\nabla_x \Phi_K} \right\| \le \left\| \proj{\calT_K}{-\nabla_x \Phi_K} \right\| \le \omega_K,
	\end{equation*}
	the first inequality following from $T(x_*) \subseteq T_K$.
\end{proof}

The next lemma is a standard result about the behavior of the Lagrange multipliers in AL methods.
\begin{lemma}[{\cite[Lemma~14.4.1]{conn-etal:2000}}]\label{lemma:behavior_multipliers}
	Assume that $\lim\limits_{K\to \infty} \mu_K = \infty$ when Algorithm~\ref{algo:traulls} is executed. Then 
	\begin{equation*}
		\lim\limits_{K\to \infty} \frac{\|\lambda_K\|}{\mu_K} = 0.
	\end{equation*}
\end{lemma}

We now state the main convergence results. Their statement parallels~\cite[Lemma~4.4 and Theorem~4.6]{conn-etal:1996b}. Since our criticality test is theirs, we restate their arguments in our notation. 
\begin{lemma}\label{lemma:key_lemma}
	Suppose that Assumptions~\ref{assumption:functions_C2} and~\ref{assumption:feasible_linear_cons} hold. Let $\left(x_K\right)_{K\in \calK}$ be a sequence of iterates in $\Omega$ satisfying Assumption~\ref{assumption:iterates_domain_bounded} and converging to $x_*$ for which Assumption~\ref{assumption:cq_null_space_jacobian} holds. Let $\lambda_* = \lambda(x_*)$, $(\lambda_K)_{K \in \calK}$ be any sequence of vectors and $(\mu_K)_{K \in \calK}$ be a nondecreasing sequence of positive scalars. Assume that~\eqref{eq:algo_traulls_criticality_test} holds with $\lim\limits_{K \in \calK}\omega_K = 0$.
	
		Then, there are positive constants $\kappa_2, \kappa_3$ such that
		\begin{equation}\label{eq:bound_1st_order_mult}
			\|\bar{\lambda}_K - \lambda_* \| \le \kappa_2 \omega_K + \kappa_3 \|x_K-x_*\|,
		\end{equation}
		\begin{equation}\label{eq:bound_ls_mult}
			\|\lambda(x_K) - \lambda_* \| \le \kappa_3 \|x_K-x_*\|,
		\end{equation}
		and
		\begin{equation}\label{eq:bound_norm_cons}
			\|c(x_K)\| \le \kappa_2 \frac{\omega_K}{\mu_K} + \kappa_3 \frac{\|x_K-x_*\|}{\mu_K} + \frac{\|\lambda_K - \lambda_* \|}{\mu_K}.
		\end{equation}
		If, in addition, $c(x_*) = 0$, then $x_*$ is a first-order critical point for problem~\eqref{pb:cnls} with corresponding Lagrange multipliers $\lambda_*$ and with $\lim\limits_{K \in \calK} \bar{\lambda}_K = \lim\limits_{K \in \calK} \lambda(x_K) = \lambda_*$.
	\end{lemma}

	\begin{proof}
		By Assumptions~\ref{assumption:functions_C2} and~\ref{assumption:cq_null_space_jacobian}, for $K$ sufficiently large, the matrix $\left(C_KN_*\right)^\dagger$ is well defined, bounded and converges to $\left(C_*N_*\right)^\dagger$. Therefore, there is a constant $\kappa_2$ such that
		\begin{equation*}
			\|(\left(C_KN_*\right)^\dagger)^\top \| \le \kappa_2.
		\end{equation*}
		Furthermore, by Lemma~\ref{lemma:bound_proj_algrad_sol_nullspace}, we have
		\[ \|N_*^\top \nabla_x\Phi_K\| = \left\| \proj{T(x_*)}{\nabla_x \Phi_K} \right\| \le  \omega_K,\]
		for $K \in \calK$ large enough. Thus we may deduce
		\begin{equation}\label{eq:bound_mult_1st_ls}
			\begin{split}
				\|\bar \lambda_K - \lambda(x_K) \| & = \|(\left(C_KN_*\right)^\dagger)^\top N_*^\top\nabla f_K + \bar\lambda_K\| \\
				& = \|(\left(C_KN_*\right)^\dagger)^\top (N_*^\top\nabla f_K + (C_KN_*)^\top\bar\lambda_K)\| \\
				& = \|((C_KN_*)^\dagger)^\top (N_*^\top  \nabla_x \Phi_K)\| \\
				& \le 	\|(\left(C_KN_*\right)^\dagger)^\top \| \ \|N_*^\top \nabla_x\Phi_K\| \\
				& \le \kappa_2 \omega_K.
			\end{split}
		\end{equation}
		Furthermore, by the mean value theorem for vector-valued functions, we have
		\[\lambda(x_K) - \lambda_* = \int_{0}^{1} \nabla \lambda(x(t))(x_K-x_*)dt,\]
		where $x(t) = x_K + t(x_*-x_K)$. By our discussion of the least-squares multipliers~\eqref{eq:least-squares_mult}, one can show that the above integrand is bounded for $x_K$ sufficiently close to $x_*$, hence
		\[\| \lambda(x_K) - \lambda_*\| \le \kappa_3 \|x_K-x_*\|,\]
		for $K\in \calK$ sufficiently large and for some positive constant $\kappa_3$, giving~\eqref{eq:bound_ls_mult}.
		Then, combining the latter inequality and~\eqref{eq:bound_mult_1st_ls}, we may deduce that
		\begin{equation*}
			\begin{split}
				\|\bar \lambda_K - \lambda_* \| & \le \|\bar \lambda_K - \lambda(x_K) \| + \|\lambda(x_K) - \lambda_* \| \\
				& \le \kappa_2 \omega_K + \kappa_3 \|x_K-x_*\|,
			\end{split}
		\end{equation*}
		which is inequality~\eqref{eq:bound_1st_order_mult}. 
		Next, by expression~\eqref{eq:1st_order_mult}, we have $c(x_K) = (\bar \lambda_K - \lambda_K) / \mu_K$, which leads to
		\begin{equation*}
			\begin{split}
				\|c(x_K) \| & \le \frac{\|\bar \lambda_K - \lambda_* \|}{\mu_K} + \frac{\|\lambda_K - \lambda_* \|}{\mu_K} \\
				& \le \kappa_2 \frac{\omega_K}{\mu_K} + \kappa_3 \frac{\|x_K-x_*\|}{\mu_K} + \frac{\|\lambda_K - \lambda_* \|}{\mu_K},
			\end{split}
		\end{equation*}
		where we have used~\eqref{eq:bound_1st_order_mult}.
		This concludes the first part of the lemma. 
	
		Since $(\omega_K)_K$ tends to $0$ and $(x_K)_{K\in \calK}$ converges to $x_*$ by assumption, inequalities~\eqref{eq:bound_1st_order_mult} and~\eqref{eq:bound_ls_mult} imply that both sequences $(\bar \lambda_K)_K$ and $(\lambda(x_K))_K$ converge to $\lambda_*$ for $K \in \calK$. Therefore, from identity~\eqref{eq:identity_al_lag_grad}, we obtain that the sequence of gradients $\nabla_x \Phi_K$ converges to $\nabla_x \calL(x_*,\lambda_*)$. 
	
	Finally, assume that $c(x_*) = 0$. For $K \in \calK$, we write $\calT_K = \calT_{\varepsilon_K}(x_K)$ and $\calN_K = \calN_{\varepsilon_K}(x_K)$, and we denote by $w_K = \proj{\calN_K}{-\nabla_x\Phi_K}$ the normal cone component of the Moreau decomposition~\eqref{eq:moreau_decomposition} of $-\nabla_x\Phi_K$, so that
	\[-\nabla_x\Phi_K = \proj{\calT_K}{-\nabla_x\Phi_K} + w_K, \qquad w_K \in \calN_K.\]
	We first show that $w_K$ lies in $\calN_\Omega(x_*)$ for $K \in \calK$ large enough. Since $\varepsilon_K = \min(\varepsilon_{\max},\omega_K)$ tends to zero, the argument used in the proof of Lemma~\ref{lemma:bound_proj_algrad_sol_nullspace} shows that the $\varepsilon_K$-active bounds at $x_K$ are active at $x_*$, on the same side, for $K \in \calK$ large enough:
	\[\calA^-_{\varepsilon_K}(x_K) \subseteq \calA^-(x_*) \qquad \text{and} \qquad \calA^+_{\varepsilon_K}(x_K) \subseteq \calA^+(x_*).\]
	The vectors $-e_i$, $i \in \calA^-_{\varepsilon_K}(x_K)$, and $e_i$, $i \in \calA^+_{\varepsilon_K}(x_K)$, which generate $\calN_K$ together with the rows of $A$ and their opposites, are thus among the generators of $\calN_\Omega(x_*)$, given by~\eqref{eq:normal_cone_eps} with $\varepsilon = 0$, so that
	\[w_K \in \calN_K \subseteq \calN_\Omega(x_*) \qquad \text{for } K \in \calK \text{ large enough}.\]
	We then show that $w_K$ converges to $-\nabla_x\calL(x_*,\lambda_*)$. The tangent cone component of the decomposition tends to zero, since $\|\proj{\calT_K}{-\nabla_x\Phi_K}\| \le \omega_K$ by~\eqref{eq:algo_traulls_criticality_test} and $\omega_K \to 0$. As $\nabla_x\Phi_K \to \nabla_x\calL(x_*,\lambda_*)$, we obtain
	\[w_K = -\nabla_x\Phi_K - \proj{\calT_K}{-\nabla_x\Phi_K} \to -\nabla_x\calL(x_*,\lambda_*).\]
	The normal cone $\calN_\Omega(x_*)$ being closed,
	\[-\nabla_x\calL(x_*,\lambda_*) \in \calN_\Omega(x_*).\]
	Therefore, we have established that $(x_*,\lambda_*)$ is a KKT point and the lemma is proved.
\end{proof}

	We can now establish the global convergence of Algorithm~\ref{algo:traulls}.
	\begin{theorem}\label{theo:global_convergence_traulls}
		Suppose that Assumptions~\ref{assumption:functions_C2} and~\ref{assumption:feasible_linear_cons} hold. Let $\left(x_K\right)_{K}$ be a sequence of iterates generated by Algorithm~\ref{algo:traulls} satisfying Assumption~\ref{assumption:iterates_domain_bounded}. Further assume that there is a subsequence $(x_K)_{K \in \calK}$ converging to $x_*$, for which Assumption~\ref{assumption:cq_null_space_jacobian} holds and let $\lambda_* = \lambda(x_*)$.
		Then, the conclusions of Lemma~\ref{lemma:key_lemma} hold.
	\end{theorem}
	\begin{proof}
		As in the proof of Lemma~\ref{lemma:key_lemma}, the inequalities~\eqref{eq:bound_1st_order_mult},~\eqref{eq:bound_ls_mult},~\eqref{eq:bound_norm_cons} can be derived from the assumptions. All that remains is to show that $c(x_*) = 0$. There are two cases to distinguish. 
	
		First, if the sequence of penalty parameters is bounded above, then there is $K_0$ sufficiently large such that $\|c(x_K)\| \le \eta_K$ for all $K \ge K_0$. Since $\eta_K \to 0$, we get that $c(x_*) = 0$.
	
	Next, if the sequence of penalty parameters is unbounded, then by the update rule, we have that $(\mu_K)_K$ grows to $\infty$. Hence, by Lemma~\ref{lemma:behavior_multipliers}, the ratio $\|\lambda_K\|/\mu_K$ converges to zero. Therefore, by inequality~\eqref{eq:bound_norm_cons}, we also obtain that $c(x_*) = 0$. 
	In both cases, $x_*$ is feasible with respect to the nonlinear constraints and the remaining conclusions of Lemma~\ref{lemma:key_lemma} hold.
\end{proof}
The criticality test~\eqref{eq:algo_traulls_criticality_test} on which the above analysis relies involves a projection onto the cone $\calT_{\varepsilon_K}(x_K)$, which amounts to solving a quadratic program. Our implementation monitors criticality instead through the measure~\eqref{eq:criticality_measure_pi}, which only requires a projection onto a subspace. The next lemma shows that the two quantities are equivalent, up to constants that only depend on $\Omega$, so that the convergence results above also apply to our implementation. Its right inequality reformulates~\cite[Lemma~7.1]{conn-etal:1996b} for the least-squares multiplier estimates of~\eqref{eq:ls_decomposition}, and the reduced gradient criterion~\cite[(7.3)]{conn-etal:1996b}. We state it for a generic vector $v$, which stands for $\nabla_x\Phi_K$, and we recall that $\pi_\varepsilon(x,v)$ is the measure~\eqref{eq:criticality_measure_pi} computed from the decomposition~\eqref{eq:ls_decomposition} of $v$ with the $\varepsilon$-active bounds (Section~\ref{subsec:al_framework}).
\begin{lemma}\label{lemma:certification}
	Suppose that Assumption~\ref{assumption:full_rank_active} holds, let $\bar\varepsilon$ be the threshold of Lemma~\ref{lemma:eps_threshold}, and let $\sigma_{\min}$ be the smallest of the smallest singular values of the matrices $\tilde{A}_\varepsilon$, over $x \in \Omega$ and $\varepsilon \in [0,\bar\varepsilon]$, which is positive since these matrices take finitely many values and have full row rank by Lemma~\ref{lemma:eps_threshold}(ii). Then, for all $\varepsilon \in [0,\bar\varepsilon]$, $x \in \Omega$ and $v \in \RR^n$,
	\begin{equation}\label{eq:certification}
		\min(1,\sigma_{\min})\,\pi_\varepsilon(x,v) \le \left\|\proj{\calT_\varepsilon(x)}{-v}\right\| = \operatorname{dist}\left(-v,\calN_\varepsilon(x)\right) \le \sqrt{n}\,\pi_\varepsilon(x,v).
	\end{equation}
	In particular, $\pi_\varepsilon(x,v) = 0$ if and only if $-v \in \calN_\varepsilon(x)$. Moreover, for $\varepsilon = 0$,
	\begin{equation}\label{eq:proj_grad_le_dist}
		\|\proj{\Omega}{x-v}-x\| \le \operatorname{dist}\left(-v,\calN_\Omega(x)\right),
	\end{equation}
	so that $\pi(x,v) = 0$ if and only if $x$ is a first-order critical point of the minimization over $\Omega$ of a function whose gradient at $x$ is $v$.
\end{lemma}
\begin{proof}
	Let $\varepsilon \in [0,\bar\varepsilon]$, $x \in \Omega$ and $v \in \RR^n$. We write $S = \calA_\varepsilon(x)$ and $\tilde{v} = \proj{T_\varepsilon(x)}{v}$, and we decompose $v$ as in~\eqref{eq:ls_decomposition} with the $\varepsilon$-active bounds,
	\[v = \tilde{v} + A^\top\xi + E_S\theta,\]
	this decomposition being unique by Lemma~\ref{lemma:eps_threshold}(ii). Since $\tilde{v} \in T_\varepsilon(x)$, we have $A\tilde{v} = 0$ and $\tilde{v}_i = 0$ for all $i \in S$. In other words, $\tilde{v}$ is orthogonal to the rows of $A$ and to the vectors $e_i$, $i \in S$, that is, to the row space of $\tilde{A}_\varepsilon$. This orthogonality is used in both inequalities below. The equality in~\eqref{eq:certification} follows from the Moreau decomposition~\eqref{eq:moreau_decomposition} of $-v$ with respect to the cone $\calT_\varepsilon(x)$ and its polar $\calN_\varepsilon(x)$, since, by definition of the projection onto the closed convex cone $\calN_\varepsilon(x)$,
	\[\operatorname{dist}\left(-v,\calN_\varepsilon(x)\right) = \left\|-v - \proj{\calN_\varepsilon(x)}{-v}\right\| = \left\|\proj{\calT_\varepsilon(x)}{-v}\right\|.\]

	To prove the right inequality in~\eqref{eq:certification}, we exhibit an element of $\calN_\varepsilon(x)$ close to $-v$. Let $\theta^+$ be obtained from $\theta$ by setting to zero its components of the wrong sign, so that
	\[|\theta_i - \theta^+_i| = \theta^{\mathrm{w}}_i \qquad \text{for all } i \in S,\]
	and let $w = -A^\top\xi - E_S\theta^+$. The components of $\theta^+$ have the signs of the multipliers at a first-order critical point, nonnegative for the lower bounds and nonpositive for the upper bounds, those of the fixed variables being unrestricted. Hence $w \in \calN_\varepsilon(x)$ by~\eqref{eq:normal_cone_eps}. Moreover,
	\[v + w = \tilde{v} + E_S(\theta-\theta^+),\]
	where the second term is a combination of the vectors $e_i$, $i \in S$, and is therefore orthogonal to $\tilde{v}$. Their squared norms thus add up, with $\|E_S(\theta-\theta^+)\| = \|\theta^{\mathrm{w}}\|$. Moreover, the components of $\tilde{v}$ indexed by $S$ vanish, so that $\|\tilde{v}\|^2 \le (n-|S|)\,\|\tilde{v}\|_\infty^2$, while $\|\theta^{\mathrm{w}}\|^2 \le |S|\,\|\theta^{\mathrm{w}}\|_\infty^2$. Hence
	\[\operatorname{dist}\left(-v,\calN_\varepsilon(x)\right) \le \|v+w\| = \left(\|\tilde{v}\|^2 + \|\theta^{\mathrm{w}}\|^2\right)^{1/2} \le \sqrt{n}\max\left(\|\tilde{v}\|_\infty,\|\theta^{\mathrm{w}}\|_\infty\right) = \sqrt{n}\,\pi_\varepsilon(x,v).\]

	Conversely, to prove the left inequality, we bound $\|v+w\|$ from below for an arbitrary $w \in \calN_\varepsilon(x)$. By~\eqref{eq:normal_cone_eps}, such a vector can be written $w = A^\top\xi' + E_S\tau$, with
	\[\tau_i \ge 0 \ \text{ for } i \in \calA^+_\varepsilon(x)\setminus\calA^-_\varepsilon(x), \qquad \tau_i \le 0 \ \text{ for } i \in \calA^-_\varepsilon(x)\setminus\calA^+_\varepsilon(x).\]
	Then $v + w = \tilde{v} + \tilde{A}_\varepsilon^\top(\xi+\xi',\theta+\tau)$. Since $\tilde{v}$ is orthogonal to the row space of $\tilde{A}_\varepsilon$, and by definition of $\sigma_{\min}$,
	\[\|v+w\|^2 = \|\tilde{v}\|^2 + \left\|\tilde{A}_\varepsilon^\top(\xi+\xi',\theta+\tau)\right\|^2 \ge \|\tilde{v}\|^2 + \sigma_{\min}^2\|\theta+\tau\|^2.\]
	The multipliers of the wrong sign cannot be compensated by $\tau$. Indeed, whenever $\theta^{\mathrm{w}}_i > 0$, the components $\theta_i$ and $\tau_i$ have the same sign, so that $|\theta_i + \tau_i| \ge |\theta_i| = \theta^{\mathrm{w}}_i$. Hence, for all $w \in \calN_\varepsilon(x)$,
	\[\|v+w\|^2 \ge \|\tilde{v}\|^2 + \sigma_{\min}^2\|\theta^{\mathrm{w}}\|^2 \ge \min(1,\sigma_{\min})^2\left(\|\tilde{v}\|^2 + \|\theta^{\mathrm{w}}\|^2\right) \ge \min(1,\sigma_{\min})^2\,\pi_\varepsilon(x,v)^2,\]
	and taking the infimum over $w$ gives the left inequality in~\eqref{eq:certification}. Finally, since $\calN_\varepsilon(x)$ is closed, $-v \in \calN_\varepsilon(x)$ if and only if $\operatorname{dist}(-v,\calN_\varepsilon(x)) = 0$, which happens by~\eqref{eq:certification} if and only if $\pi_\varepsilon(x,v) = 0$.

	It remains to consider the case $\varepsilon = 0$, for which $\calN_0(x) = \calN_\Omega(x)$. Let
	\[p = \proj{\Omega}{x-v}, \qquad d = p - x.\]
	The projection onto the closed convex set $\Omega$ is characterized by $\inner{x-v-p}{y-p} \le 0$ for all $y \in \Omega$, and taking $y = x$ gives
	\[\|d\|^2 \le \inner{-v}{d}.\]
	Moreover, every $w \in \calN_\Omega(x)$ satisfies $\inner{w}{y-x} \le 0$ for all $y \in \Omega$, as can be checked on the generators in~\eqref{eq:normal_cone_eps} with $\varepsilon = 0$, and in particular $\inner{w}{d} \le 0$ since $p \in \Omega$. Hence, for all $w \in \calN_\Omega(x)$,
	\[\|d\|^2 \le \inner{-v-w}{d} \le \|v+w\|\,\|d\|,\]
	so that $\|d\| \le \|v+w\|$, and taking the infimum over $w$ yields~\eqref{eq:proj_grad_le_dist}. The last statement of the lemma follows from the case $\varepsilon = 0$ of the first part, since $x$ is a first-order critical point of the minimization over $\Omega$ of a function whose gradient at $x$ is $v$ if and only if $-v \in \calN_\Omega(x)$, as recalled in Section~\ref{subsec:outer_convergence}.
\end{proof}
Consequently, if $\varepsilon_{\max} \le \bar\varepsilon$, each inner minimization terminates by Theorem~\ref{theo:inner_convergence} at a point satisfying~\eqref{eq:criticality_cond_reduced_grad}, hence~\eqref{eq:algo_traulls_criticality_test} with $\omega_K$ replaced by $\sqrt{n}\,\omega_K$ by Lemma~\ref{lemma:certification}. This tolerance still tends to zero, so that Lemmas~\ref{lemma:bound_proj_algrad_sol_nullspace} and~\ref{lemma:key_lemma} and Theorem~\ref{theo:global_convergence_traulls} apply to the iterates of Algorithm~\ref{algo:traulls} with the inner minimization of Algorithm~\ref{algo:trinner_iteration}, without any assumption on the signs of the bound multipliers. Likewise, the final test of Algorithm~\ref{algo:traulls}, carried out as $\pi_{\varepsilon_*}(x_K,\lambda_K,\mu_K) \le \omega_*$, certifies $\|\proj{\calT_{\varepsilon_*}(x_K)}{-\nabla_x\Phi_K}\| \le \sqrt{n}\,\omega_*$. As in~\cite{conn-etal:1996b}, this certificate concerns the bounds that lie within distance $\varepsilon_* \le \omega_*$ of the final point.

	\section{Numerical experiments}\label{sec:numerical_experiments}

In this section, we report numerical experiments conducted with the Julia implementation of our algorithm, named TRAULLS. All runs were performed with Julia 1.12.1 on a machine equipped with an Intel\textsuperscript{\textregistered} Core\textsuperscript{\texttrademark} Ultra 9 processor 285 and 64 GiB of RAM. All the benchmarks we present here are also accessible in the \texttt{benchmark} folder of the GitHub repository.
The tests were carried out on a set of 52 constrained NLS problems from the collections~\cite{hockschittkowski:1980, schittkowski:1987, lucksanvlcek:1999}, and problems~2 and~3 referenced in~\cite{biegler-etal2000}. For problems with variable dimensions, we set the number of variables to $n \in \{100, 500, 1000\}$. Thus, we have a total of 82 problem instances ranging from 2 to 1001 variables, with up to 2994 residuals and 999 constraints. Problems with inequality constraints were converted to the form~\eqref{pb:cnls} by adding slack variables. Only five of the 82 instances, HS14, HS22, HS23, HS32 and HS323, have a linear constraint kept explicit, a single equality in each case once slack variables are added. On all the others, the polyhedron~\eqref{eq:linear_constraints} reduces to a box, on which our implementation takes $\varepsilon = 0$, so that the Cauchy path of Section~\ref{subsec:step_computation} is the projected gradient path of~\cite{conn-etal:1988b}. Section~\ref{subsec:linear_constraints} therefore assesses the treatment of general polyhedra separately. For our tests, we set the criticality tolerance to $\omega_* = 10^{-5}$, the feasibility tolerance to $\eta_* = 10^{-6}$, the maximum number of outer iterations to 500 and the maximum number of inner iterations to 1000. Both criticality tests are relative in our implementation: the final test of Algorithm~\ref{algo:traulls} is carried out as $\pi_{\varepsilon_*}(x_K,\lambda_K,\mu_K) \le \omega_*(1+\pi(x_0))$, and each inner minimization stops once its criticality measure falls below $\omega_K$ times one plus its value at the starting point of that minimization. These relative tests are a practical choice, which the analysis of Section~\ref{sec:convergence_analysis}, carried out with absolute tolerances, does not cover. The remaining parameters are set to $\mu_0 = 10$, $\tau = 10$, $\kappa_\omega = \beta_\omega = \eta = \omega = 1$, $\kappa_\eta = 0.1$, $\beta_\eta = 0.9$, $\kappa_{\mathrm{hyb}} = 0.1$, $\delta_0 = 0.1$, $\varepsilon_{\max} = 10^{-3}$, and the trust-region constants $\alpha_1 = 0.25$, $\alpha_2 = 2.5$, $\eta_1 = 0.05$, $\eta_2 = 0.9$, $\gamma_1 = 0.0625$ recommended in~\cite[Section~17.1]{conn-etal:2000}, in accordance with the conditions of Section~\ref{subsec:implementation_details}.

	To avoid any bias in computation time caused by Julia's just-in-time compilation, and for all solvers considered in this section, each instance was solved once to warm up, the result of that first solve being discarded, then solved three times with the minimum of the three timings reported. We present our results with Dolan--Mor\'e performance profiles~\cite{dolanmore:2002} produced with the \texttt{BenchmarkProfiles.jl} package~\cite{benchmark-profiles-julia}. We exclude from the profiles the problems on which the solvers considered for a given comparison do not all converge to the same solution, in the sense that
	\begin{equation}\label{eq:discarding_heuristic}
	\dfrac{f_{\max} - f_{\min}}{1 + \max \left(|f_{\max}|,|f_{\min}|\right)} > 10^{-2},
	\end{equation}
	where $f_{\max}$ and $f_{\min}$ are respectively the maximum and minimum values of the objective obtained on a given instance. A similar discarding heuristic is also used in~\cite{wachterbiegler:2006, orbansiqueira:2020}.

	\subsection{Comparison of Hessian approximations}\label{subsec:compare_hessians}

	Our first set of experiments compares the behavior and performance of our algorithm depending on the Hessian approximation used during the inner minimization. We compare the GN approximation~\eqref{eq:hessian_gn_approx} against structured approximations of the form~\eqref{eq:hessian_full_approx}. Based on the structured secant equation~\eqref{eq:structured_secant_eq}, we implemented the SR1 update we discussed and the BFGS update that one can derive similarly. Both structured approximations were used in their plain and hybrid variants with the switching rule~\eqref{eq:hybrid_update_strategy}.

	\subsubsection{Profiles on the whole test set}

	We first compare these five strategies with respect to three metrics: computation time, number of outer iterations and number of inner iterations. The last metric coincides with the number of residual evaluations, since the trust-region mechanism evaluates the residuals exactly once per inner iteration. Results are reported in Figure~\ref{fig:compare_hessians}.

One variant fails to reach a first-order critical point: the plain SR1 variant reaches the maximum number of outer iterations on problem HS373 and returns a point that is only feasible. There, its iterates leave the basin of the solution at the second outer iteration and drift along a valley in which $x_1$ and $x_2$ grow without bound while the constraints remain satisfied, a situation excluded by Assumption~\ref{assumption:iterates_domain_bounded}. Problem HS23 and the instances of problems~5.13 and~5.14 from~\cite{lucksanvlcek:1999} with $n=100$ were removed after evaluating the relative-gap criterion~\eqref{eq:discarding_heuristic}. The profiles are therefore built on the remaining 79 instances, for which all the variants that converge agree on the solution.

\begin{figure}[htbp]
	\centering
	\begin{minipage}[b]{0.32\textwidth}
		\centering
		\includegraphics[width=\textwidth]{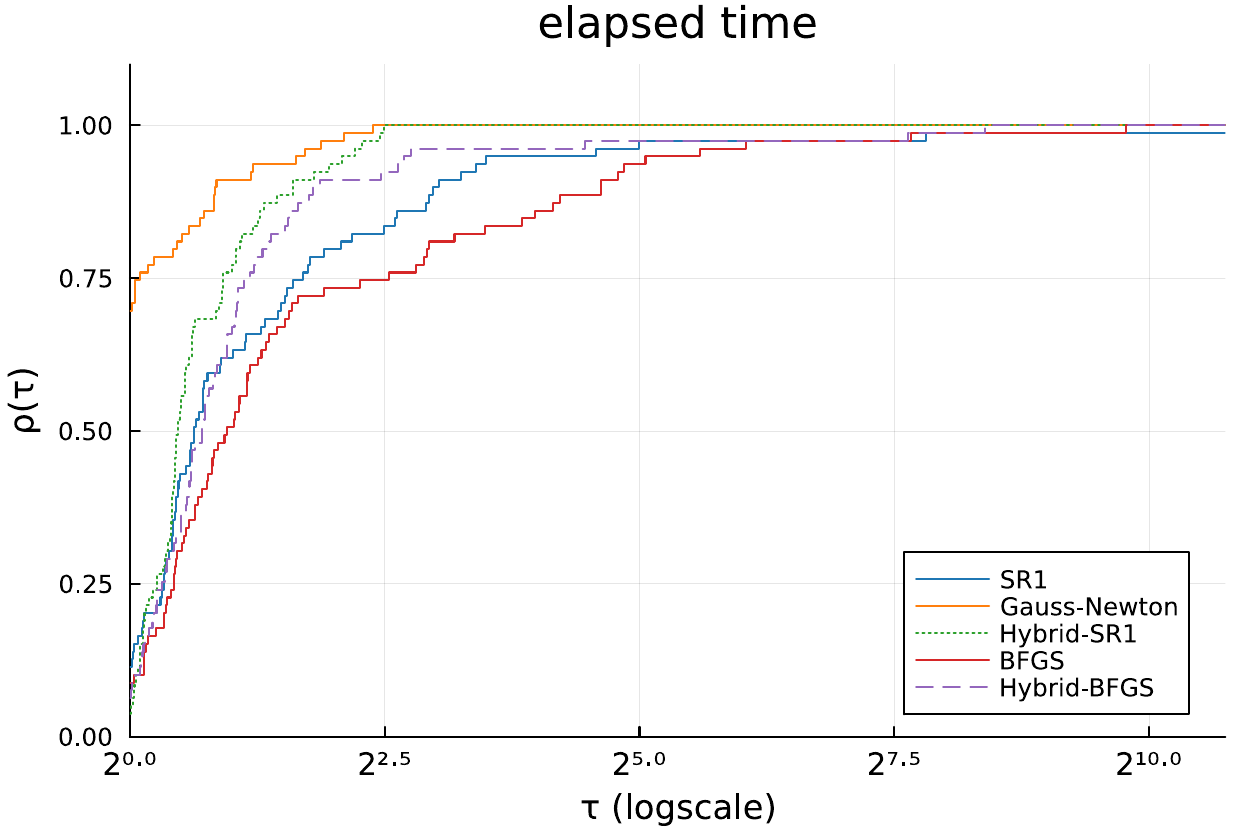}
	\end{minipage}
	\hfill
	\begin{minipage}[b]{0.32\textwidth}
		\centering
		\includegraphics[width=\textwidth]{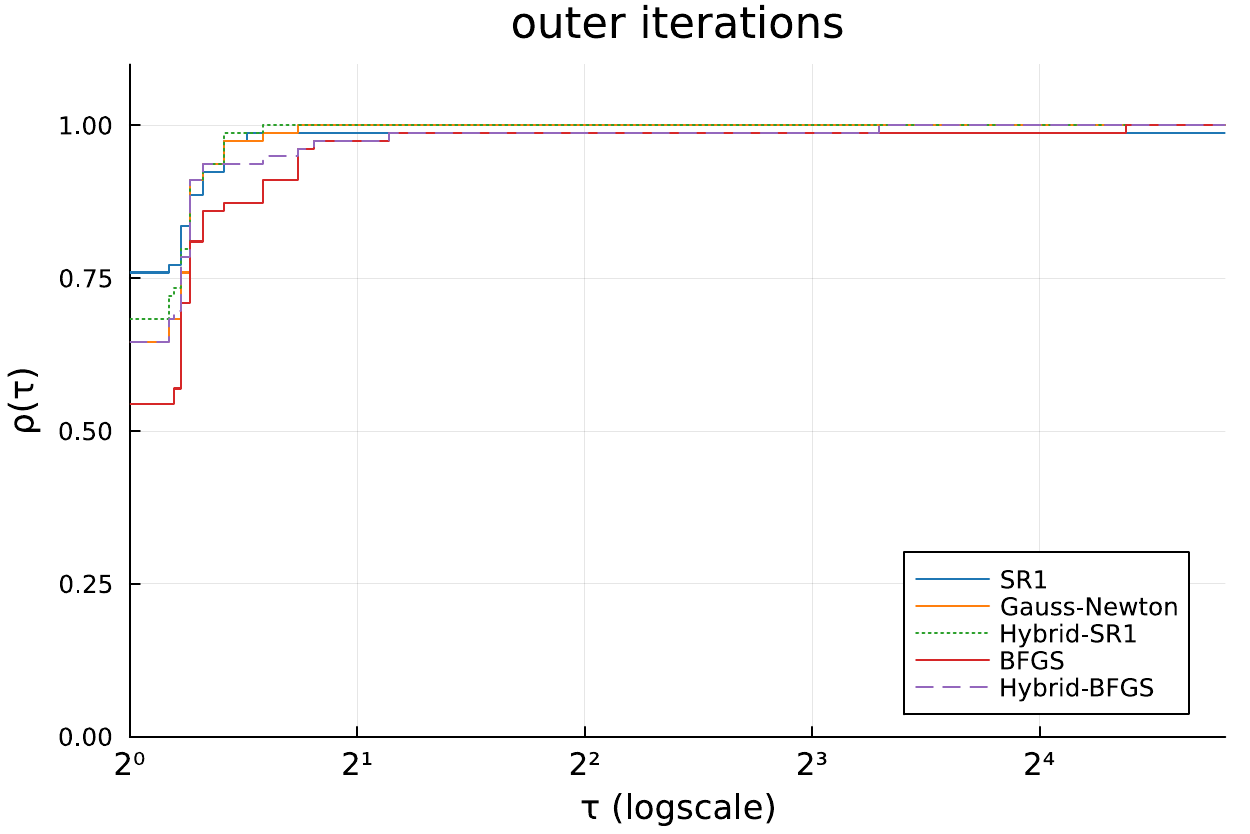}
	\end{minipage}
	\hfill
	\begin{minipage}[b]{0.32\textwidth}
		\centering
		\includegraphics[width=\textwidth]{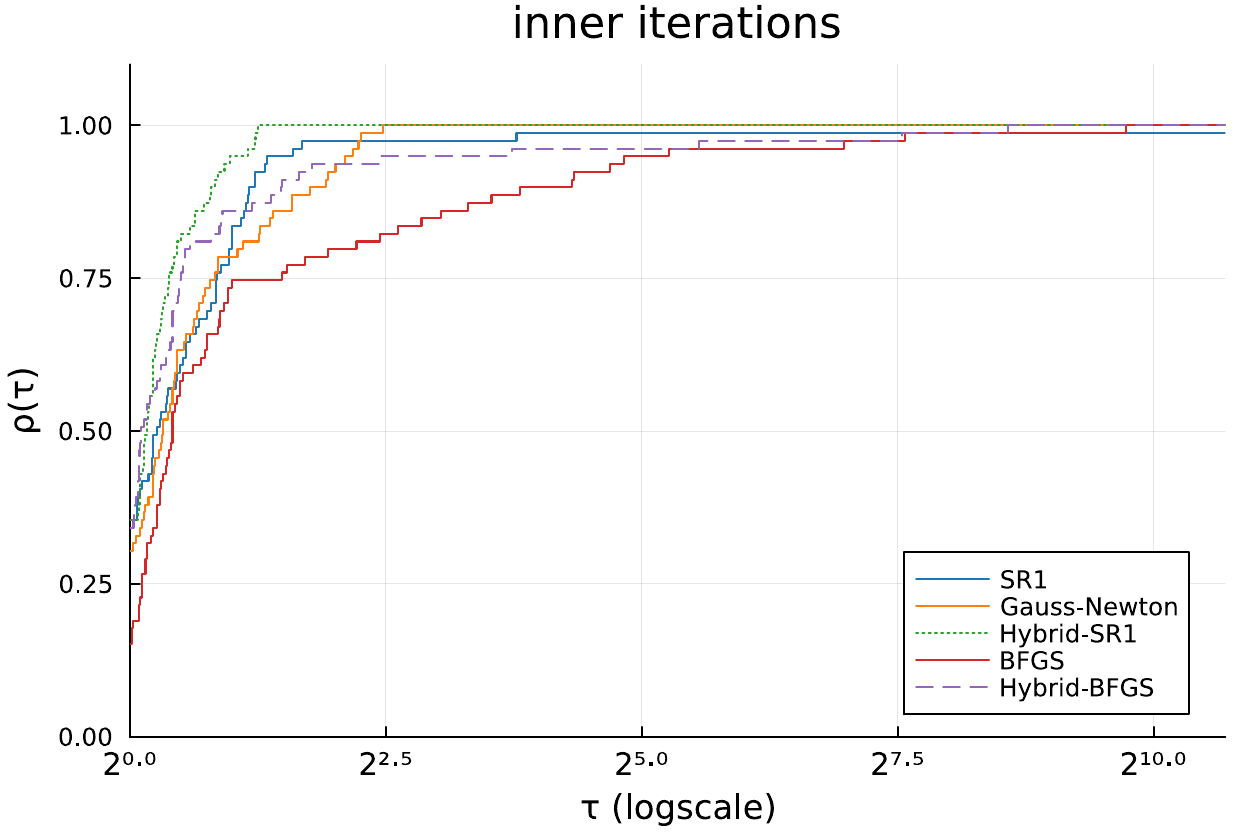}
	\end{minipage}
	\caption{Performance profiles comparing five different Hessian approximation strategies with respect to computation time (left), number of outer iterations (center),
		and number of inner iterations (right).}
	\label{fig:compare_hessians}
\end{figure}

	In terms of computation time, the GN approximation is the fastest and is also the cheapest per iteration because it requires no additional storage or update computation, and because our method mainly involves matrix-vector products. The GN approximation especially dominates on problems having sparse Jacobians and small residuals at the solution, since the second-order terms are still stored as a dense matrix. The Hybrid-SR1 variant is a close second.

Measured by outer iterations, the five strategies are almost indistinguishable, which is expected since the outer loop is governed by the AL updates rather than by the Hessian model. The number of inner iterations, on the other hand, separates them clearly. There, the hybrid variants outperform their plain counterparts and the GN approximation alike. This is consistent with the discussion of the hybrid switching rule~\eqref{eq:hybrid_update_strategy}, which -- as in the unconstrained case -- correctly identifies situations where the GN approximation must be corrected and activates the SQN correction accordingly. Overall, Hybrid-SR1 appears to be the most robust variant, which also indicates that the SR1 approximation is well suited to handling negative curvature in the trust-region setting.

	\subsubsection{Profiles stratified by residual magnitude at the solution}

To refine these observations and analyze separately the effectiveness of the approximate second-order correction and of the hybrid switching scheme, we stratify the previous profiles by the squared norm of the residuals at the solution, denoted $\|r_{\mathrm{opt}}\|^2$. We distinguish four regimes: zero residuals ($\|r_{\mathrm{opt}}\|^2 < 10^{-8}$, 18 instances), small residuals ($10^{-8} \le \|r_{\mathrm{opt}}\|^2 < 1$, 16 instances), medium residuals ($1 \le \|r_{\mathrm{opt}}\|^2 < 100$, 23 instances) and large residuals ($\|r_{\mathrm{opt}}\|^2 \ge 100$, 22 instances). For this comparison, we retain the GN approximation and the SR1 update in its plain and hybrid variants. Figure~\ref{fig:stratified_reseval} reports the profiles on the number of residual evaluations -- equivalently, as discussed above, on the number of inner iterations -- and Figure~\ref{fig:stratified_time} those on computation time.

\begin{figure}[htbp]
	\centering
	\begin{minipage}[b]{0.48\textwidth}
		\centering
		\includegraphics[width=\textwidth]{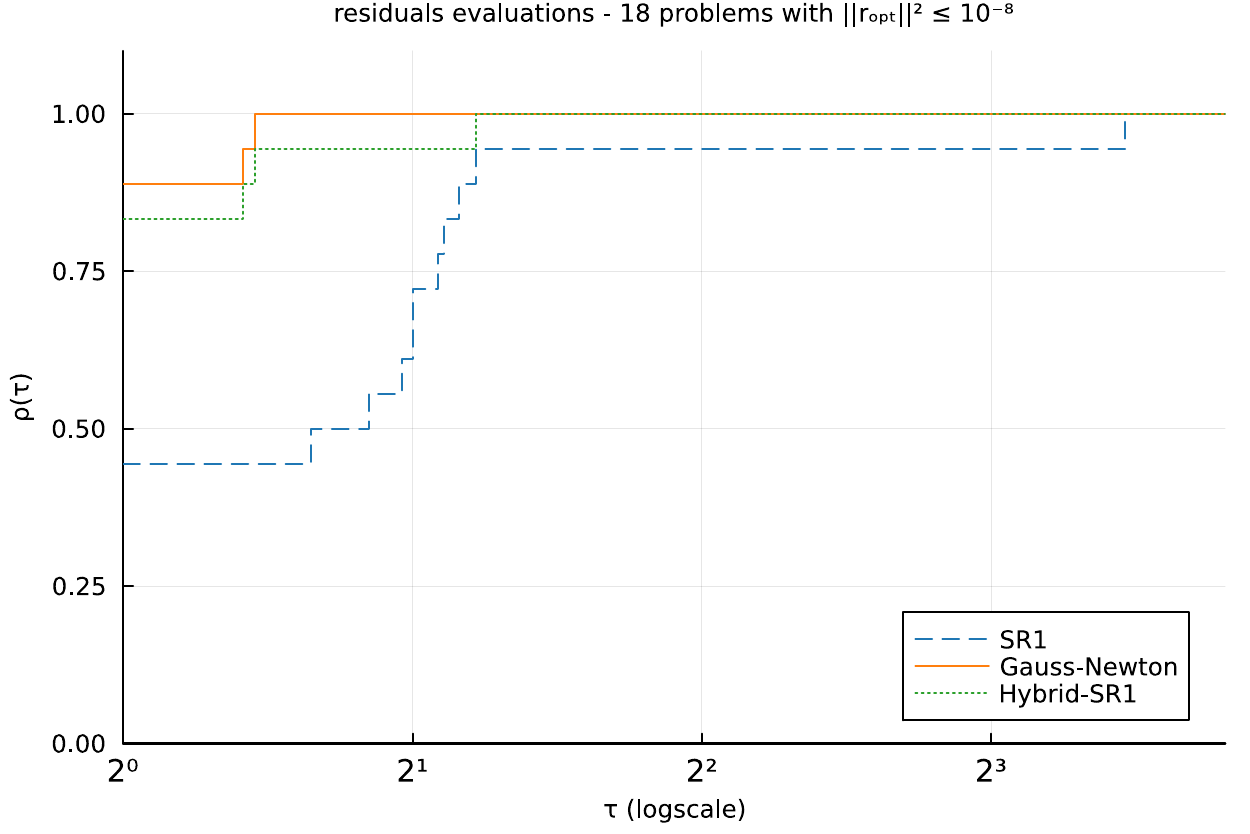}
	\end{minipage}
	\hfill
	\begin{minipage}[b]{0.48\textwidth}
		\centering
		\includegraphics[width=\textwidth]{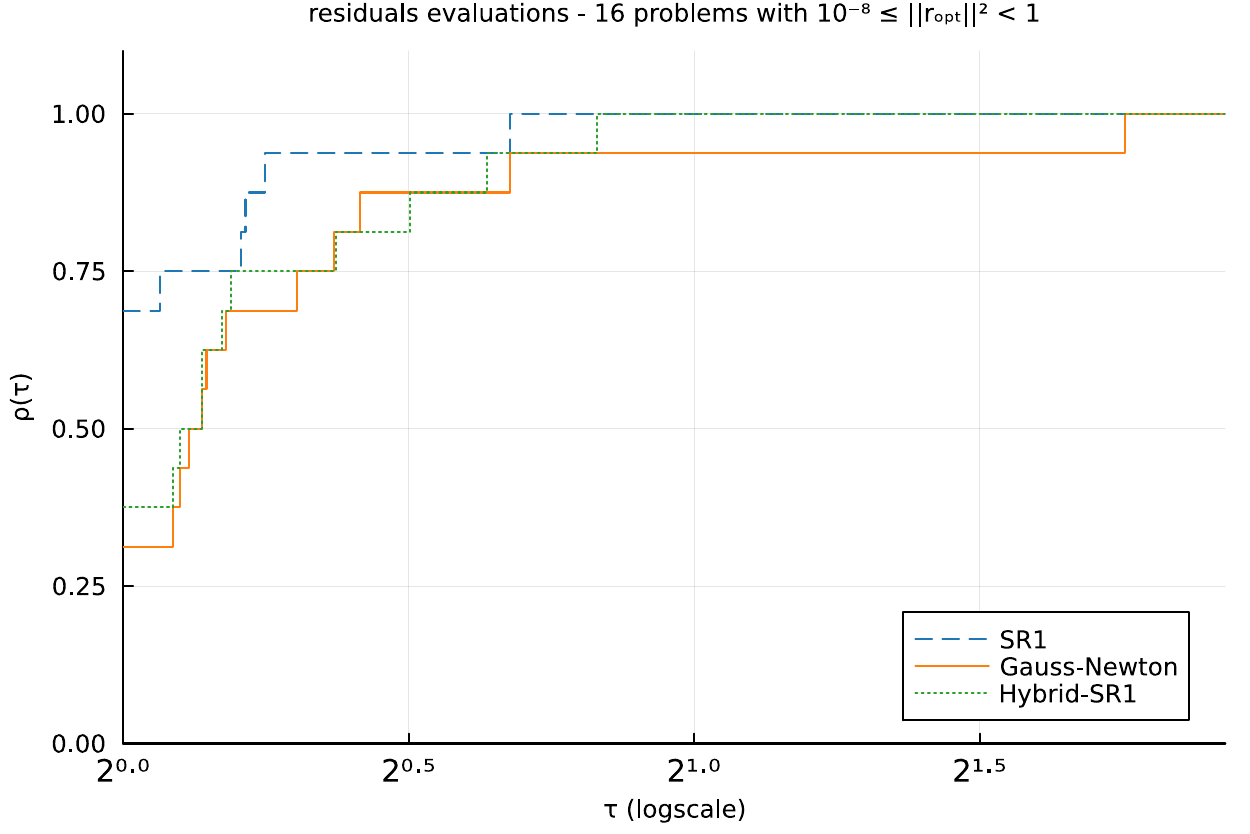}
	\end{minipage}
	
	\vspace{0.5em}
	
	\begin{minipage}[b]{0.48\textwidth}
		\centering
		\includegraphics[width=\textwidth]{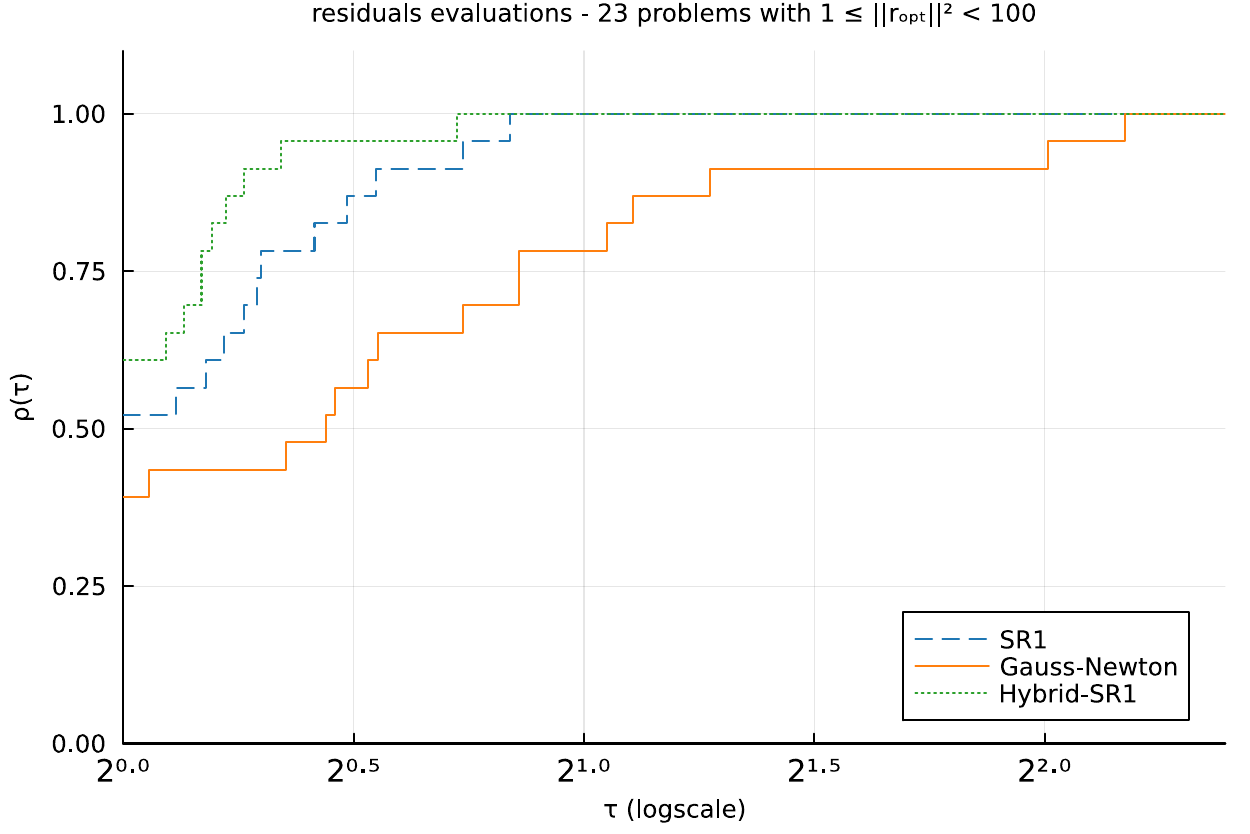}
	\end{minipage}
	\hfill
	\begin{minipage}[b]{0.48\textwidth}
		\centering
		\includegraphics[width=\textwidth]{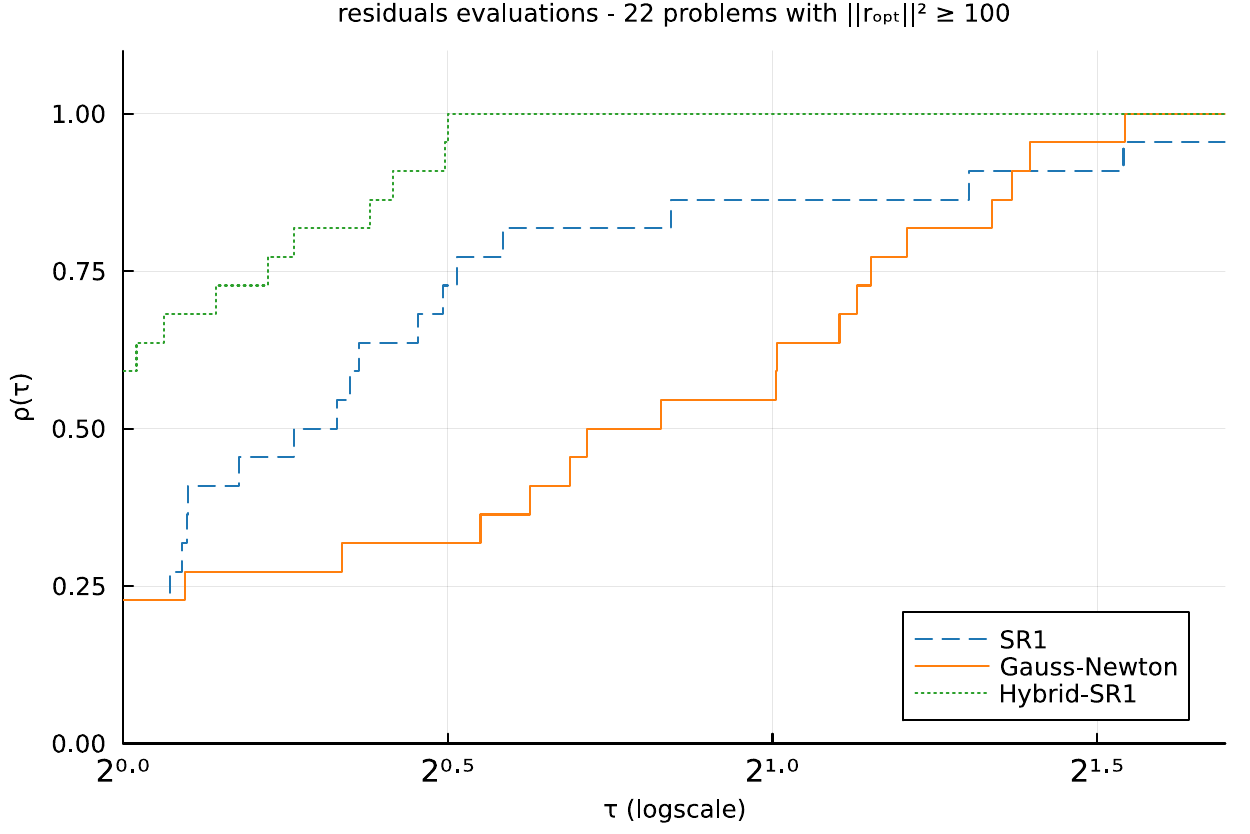}
	\end{minipage}
	\caption{Performance profiles on the number of residual evaluations, stratified by the squared norm of the residuals at the solution: zero residuals (top left), small residuals (top right), medium residuals (bottom left) and large residuals (bottom right).}
	\label{fig:stratified_reseval}
\end{figure}

	The profiles on residual evaluations display a clear trend as the residual magnitude grows. The GN approximation performs better on the zero-residual instances but worse on the nonzero-residual ones, whereas Hybrid-SR1 is consistent across all residual magnitudes. The gap between the two even widens as the residuals get larger. The hybrid scheme therefore matches GN when the residuals vanish at the solution and takes a decisive lead as soon as the residuals are significant. This is precisely the behavior the switching rule is designed to produce.

The plain SR1 update stays close to its hybrid counterpart on the small- and medium-residual instances, and is even ahead on the former. The second-order correction is thus beneficial in itself, and the hybrid rule does not always exploit it fully in this intermediate regime. A plausible explanation is that the switching test~\eqref{eq:hybrid_update_strategy} monitors the relative decrease of the AL, a quantity influenced by both the residual size and the constraint violation, so that nonzero but still small residuals are not always identified as such. In the large-residual regime, by contrast, the hybrid rule is unambiguously the better strategy.

\begin{figure}[htbp]
	\centering
	\begin{minipage}[b]{0.48\textwidth}
		\centering
		\includegraphics[width=\textwidth]{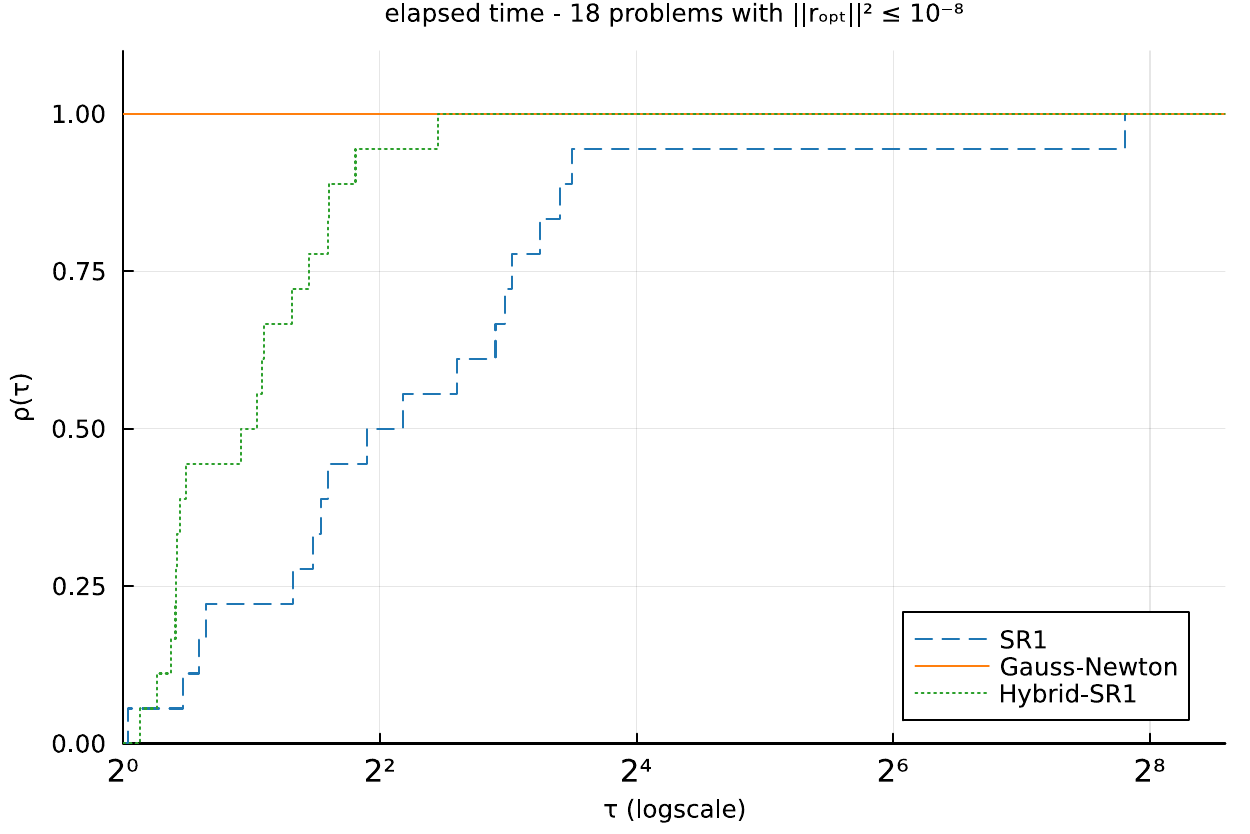}
	\end{minipage}
	\hfill
	\begin{minipage}[b]{0.48\textwidth}
		\centering
		\includegraphics[width=\textwidth]{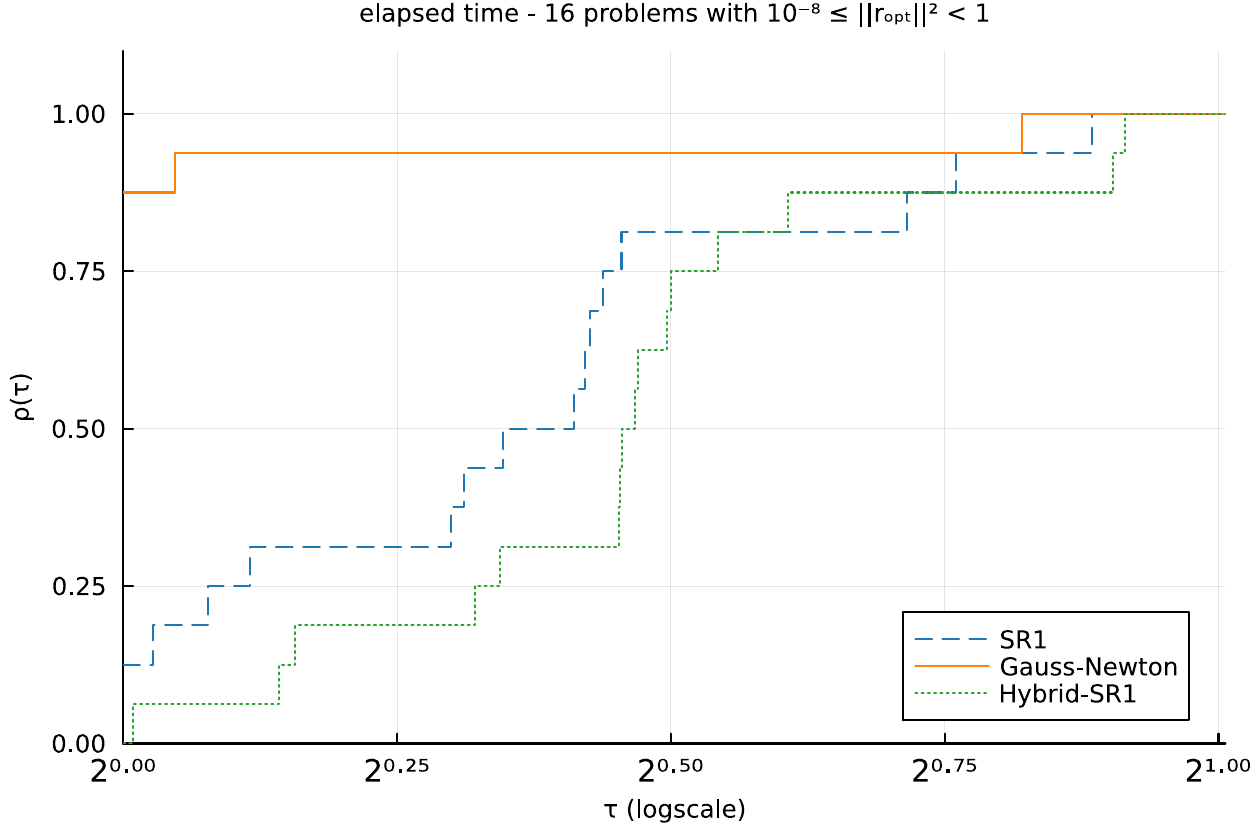}
	\end{minipage}
	
	\vspace{0.5em}
	
	\begin{minipage}[b]{0.48\textwidth}
		\centering
		\includegraphics[width=\textwidth]{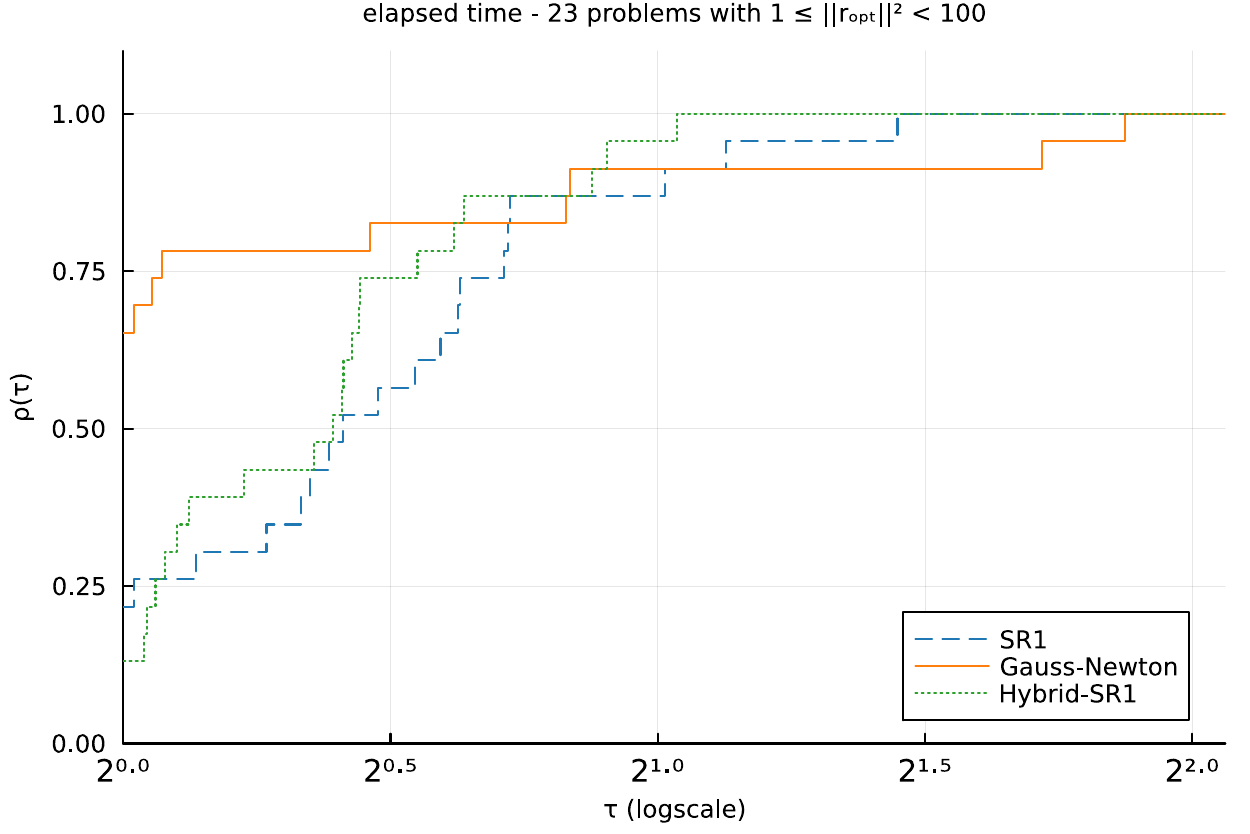}
	\end{minipage}
	\hfill
	\begin{minipage}[b]{0.48\textwidth}
		\centering
		\includegraphics[width=\textwidth]{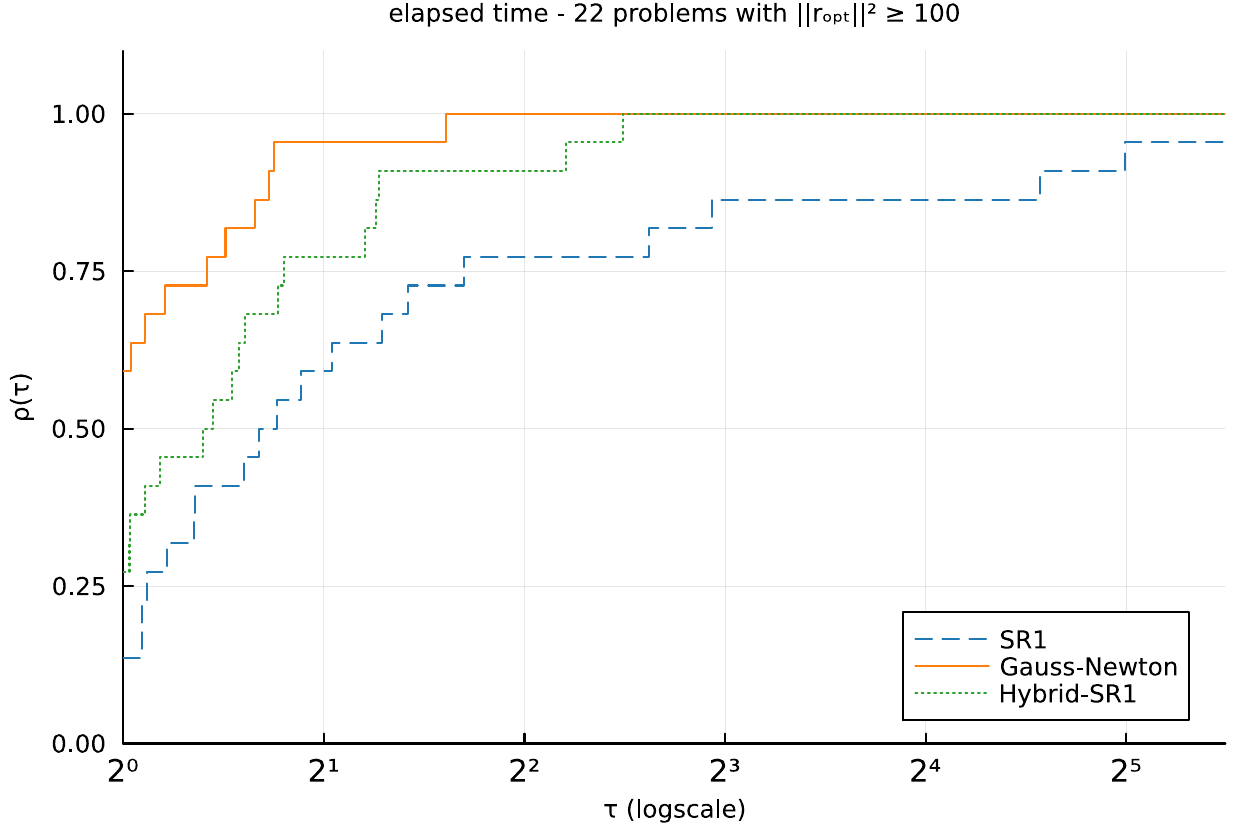}
	\end{minipage}
	\caption{Performance profiles on computation time, stratified by the squared norm of the residuals at the solution: zero residuals (top left), small residuals (top right), medium residuals (bottom left) and large residuals (bottom right).}
	\label{fig:stratified_time}
\end{figure}

Computation time, however, favors the GN approximation in every stratum, although GN is the fastest on fewer instances as the residuals grow. The savings in residual evaluations achieved by the hybrid scheme therefore do not compensate for the additional cost of forming and using a dense second-order correction in the subsequent linear-algebra computations. This is largely due to how the test set is built. For all the problems considered, residuals are available by direct computation and their evaluation is consequently very cheap. In many applications, and in particular when the residuals come from a simulation code or from the solution of a system of ordinary differential equations, their evaluation is the dominant computational cost. In such settings, the profiles of Figure~\ref{fig:stratified_reseval} are the relevant ones, and they show that the structured correction with hybrid switching does reduce the amount of work that actually matters. 

To quantify this observation, we ask at which evaluation cost the structured correction starts to pay off. Assume that one residual evaluation costs $c$ more than it does here, every other cost being unchanged. The total time of a solver becomes
\begin{equation}\label{eq:residual_cost_model}
	T(c) = T + c\, n_r^{\mathrm{eval}},
\end{equation}
where $T$ is the time we measured and $n_r^{\mathrm{eval}}$ the number of residual evaluations it performed. On a given instance, the GN approximation and Hybrid-SR1 are equally fast at the break-even cost
\begin{equation}\label{eq:breakeven_cost}
	c^\star = \dfrac{T_{\mathrm{hyb}}-T_{\mathrm{gn}}}{n^{\mathrm{eval}}_{\mathrm{gn}}-n^{\mathrm{eval}}_{\mathrm{hyb}}},
\end{equation}
which is meaningful when the hybrid variant performs fewer residual evaluations but takes more time, the only regime in which a crossover exists. Beyond $c^\star$, the hybrid variant is the faster of the two. This model keeps the iteration counts fixed and charges only the evaluations, so it describes a problem whose residuals are more expensive but whose solution path is unchanged.

Before applying it, we record in Table~\ref{tab:evaluation_ratios} how the evaluation counts of the two variants compare, stratum by stratum. The relative-gap criterion~\eqref{eq:discarding_heuristic} is applied to the objective values reached by the GN approximation and Hybrid-SR1 alone, and it discards no instance. The pairwise counts are as informative as the geometric means. On the zero-residual instances the two variants perform exactly the same number of evaluations on 17 of the 18 instances, which confirms that the switching rule is inert when the correction is not needed and that it costs nothing there. The saving then grows with the residual magnitude. The Jacobian and gradient counts follow the same trend, which matters here, since an application whose residuals are expensive to evaluate usually has expensive derivatives as well.

\begin{table}[htbp]
	\centering
	\caption{Evaluation counts of Hybrid-SR1 relative to the GN approximation, by residual magnitude at the solution. The three middle columns count the instances on which Hybrid-SR1 performs fewer, exactly as many, or more residual evaluations than GN. The last two columns report the geometric mean of the ratio of the Jacobian and gradient evaluation counts, a value below one favoring Hybrid-SR1.}\label{tab:evaluation_ratios}
	\begin{tabular}{lcccccc}
		\hline
		& & \multicolumn{3}{c}{residual evaluations} & \multicolumn{2}{c}{geometric mean} \\
		stratum & instances & fewer & equal & more & Jacobian & gradient \\
		\hline 
		zero   & 18 &  0 & 17 & 1 & 1.04 & 1.04 \\
		small  & 16 &  4 &  6 & 6 & 0.95 & 0.96 \\
		medium & 24 & 14 &  6 & 4 & 0.76 & 0.79 \\
		large  & 24 & 18 &  0 & 6 & 0.67 & 0.70 \\
		\hline
	\end{tabular}

\end{table}

We now restrict the analysis to the 48 medium- and large-residual instances, since these are the ones on which the correction is meant to act. A crossover exists on 17 of them, those on which Hybrid-SR1 performs fewer residual evaluations but takes more time. On these instances, the break-even cost~\eqref{eq:breakeven_cost} ranges from 0.003 milliseconds at the first quartile to 1.18 milliseconds at the third quartile. We do not report the median, which is too sensitive to timing noise on so few instances: it changed by a factor of nearly three between two timed runs of the same code. In other words, an extra cost of about one millisecond per residual evaluation suffices for Hybrid-SR1 to become the faster variant on three quarters of these 17 instances. An evaluation that requires a simulation run or the integration of a system of ordinary differential equations is typically one to two orders of magnitude above that threshold.

Figure~\ref{fig:breakeven} shows the proportion of the 48 medium- and large-residual instances on which Hybrid-SR1 is faster, as a function of the extra cost $c$. It starts at 31\%, the proportion obtained with the times we measured, and saturates at two thirds as $c$ increases, that is, on the 32 instances on which Hybrid-SR1 performs fewer residual evaluations. No evaluation cost can make it faster on the remaining ones. The conclusion we draw is therefore a conditional one, namely that the structured correction is the better choice on the majority of the problems with nonzero residuals as soon as the residuals are genuinely expensive, and that the threshold at which this happens is low enough to be reached by most applications in which the residuals come from a simulation.

\begin{figure}[htbp]
	\centering
	\includegraphics[width=0.62\textwidth]{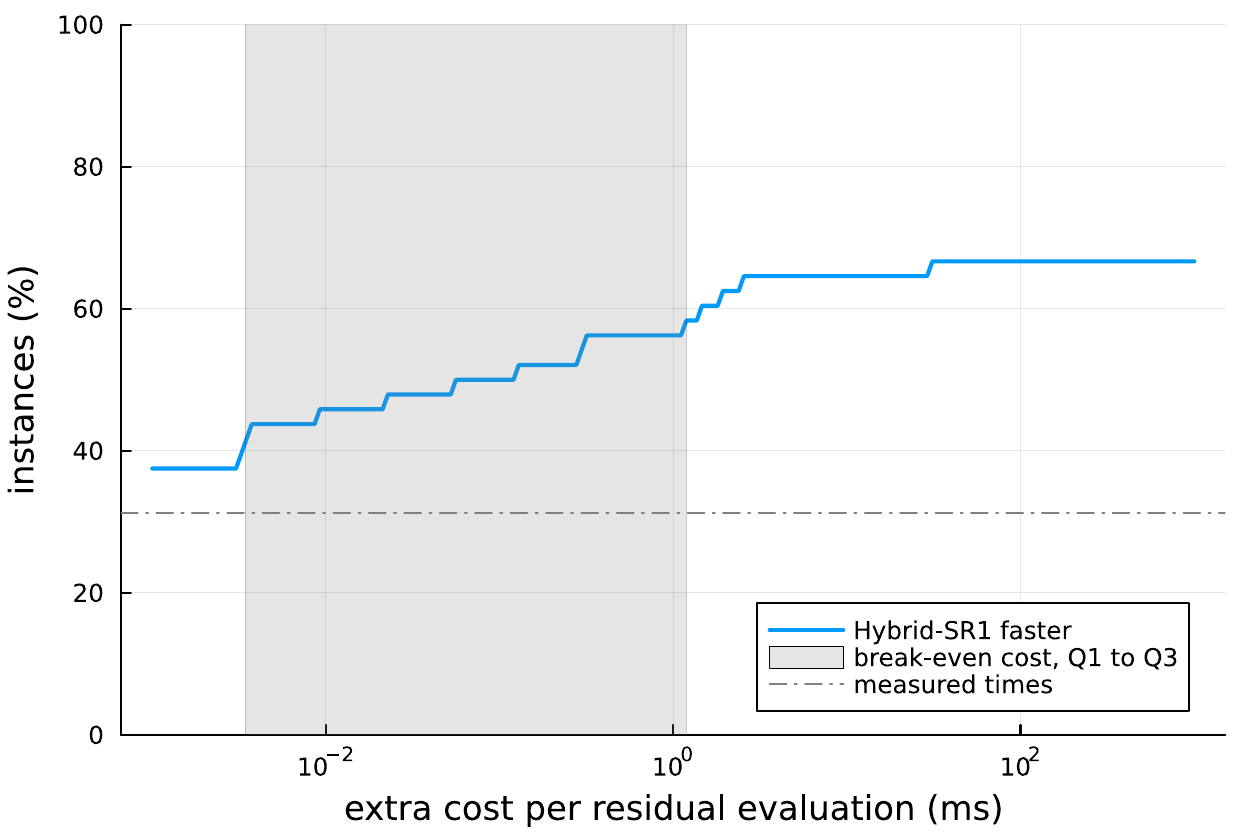}
	\caption{Proportion of the 48 medium- and large-residual instances on which Hybrid-SR1 is faster than the GN approximation, as a function of the extra cost charged to every residual evaluation in model~\eqref{eq:residual_cost_model}. The shaded band spans the break-even costs~\eqref{eq:breakeven_cost} between the first and third quartiles, and the dash-dotted line marks the proportion obtained with the measured times.}
	\label{fig:breakeven}
\end{figure}

	Based on these results, we retain the Hybrid-SR1 variant as the best overall strategy for TRAULLS, and use it in the comparisons that follow.

	\subsection{Comparison with IPOPT and Percival}\label{subsec:compare_solvers}

We now compare TRAULLS (with the Hybrid-SR1 approximation) against the solvers IPOPT~\cite{wachterbiegler:2006} and Percival~\cite{arreckx-etal:2016}. The latter is also an AL algorithm and has a native Julia implementation. Its inner minimization uses a Julia version of the TRON~\cite{linmore:1999a} algorithm -- part of the \texttt{JSOSolvers.jl} package~\cite{migot-etal:2026}. To benefit from the benchmarking facilities provided by the JuliaSmoothOptimizers (JSO)\footnote{\url{https://jso.dev}} ecosystem, we used the wrapper \texttt{NLPModelsIpopt.jl}~\cite{nlp-models-ipopt-julia}. This allows us to run both reference solvers on the same test set as TRAULLS -- all models being encoded in \texttt{NLSProblems.jl}~\cite{nls-problems-julia} -- and to export the benchmark metrics using \texttt{SolverBenchmark.jl}~\cite{solver-benchmark-julia}. All solvers shared the experimental setup and the tolerances described at the beginning of this section. Percival was additionally given a time limit of 10 minutes per instance, a limit consistent with the dimensions of the problems considered; no run of TRAULLS or IPOPT took more than 10 seconds.

To make the comparison with IPOPT relevant, we set the parameters \texttt{nlp\_scaling\_method} = \texttt{none}, \texttt{dual\_inf\_tol} = \texttt{Inf}, \texttt{constr\_viol\_tol} = \texttt{Inf}, \texttt{compl\_inf\_tol} = \texttt{Inf}, \texttt{acceptable\_iter} = \texttt{0} and the maximum number of iterations to 1000. We also set \texttt{hessian\_approximation} = \texttt{limited-memory}, so that IPOPT builds its Hessian model with an L-BFGS update of memory 6 instead of exact second derivatives. This choice makes the comparison meaningful with respect to the derivative information used. For Percival, since its outer iteration follows the same update rules as Algorithm~\ref{algo:traulls}, we set its tolerances and parameters to the same values as ours. We note that Percival, through its TRON inner solver, uses exact Hessian-vector products, and is therefore the only solver of the three that accesses second derivatives.

For these tests, we compare the three solvers with respect to the computation time, the number of residual evaluations and the number of gradient evaluations. After evaluating criterion~\eqref{eq:discarding_heuristic}, we removed from the profiles the following nine instances: HS61, HS77, HS264, the instances with $n=100$ of problems~5.1, 5.4, 5.13 and~5.14 from~\cite{lucksanvlcek:1999}, and the instances with $n=500$ and $n=1000$ of problem~5.16 from the same collection. The profiles of Figure~\ref{fig:perf_solvers} are thus built on the remaining 73 instances. On these, TRAULLS and IPOPT converge on every instance, while Percival reaches its time limit on the two instances of problem~5.11 from~\cite{lucksanvlcek:1999} with $n=500$ and $n=1000$. As a side note, we also ran IPOPT with exact second derivatives, and this configuration failed on two instances, reaching the maximum number of iterations on HS27 and terminating with a restoration-phase failure on HS70.

\begin{figure}[htbp]
	\centering
	\begin{minipage}[b]{0.32\textwidth}
		\centering
		\includegraphics[width=\textwidth]{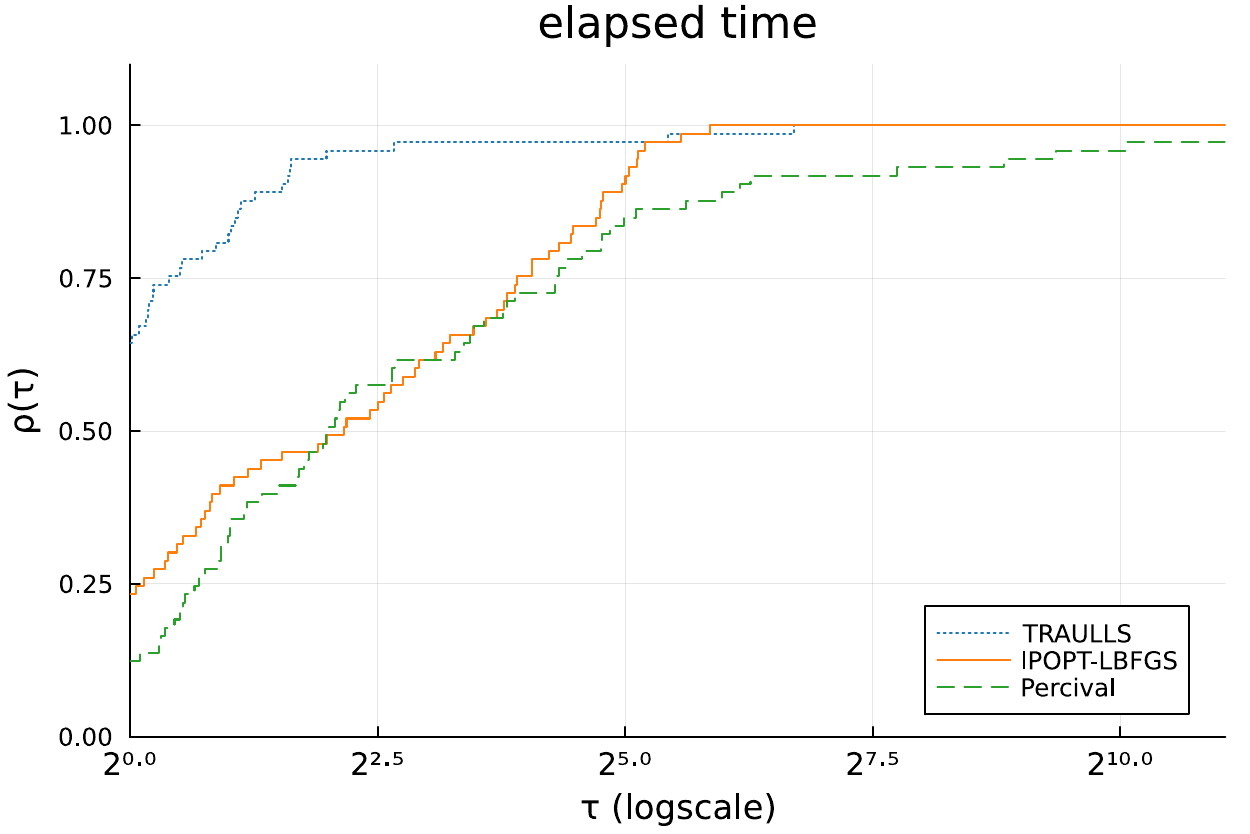}
	\end{minipage}
	\hfill
	\begin{minipage}[b]{0.32\textwidth}
		\centering
		\includegraphics[width=\textwidth]{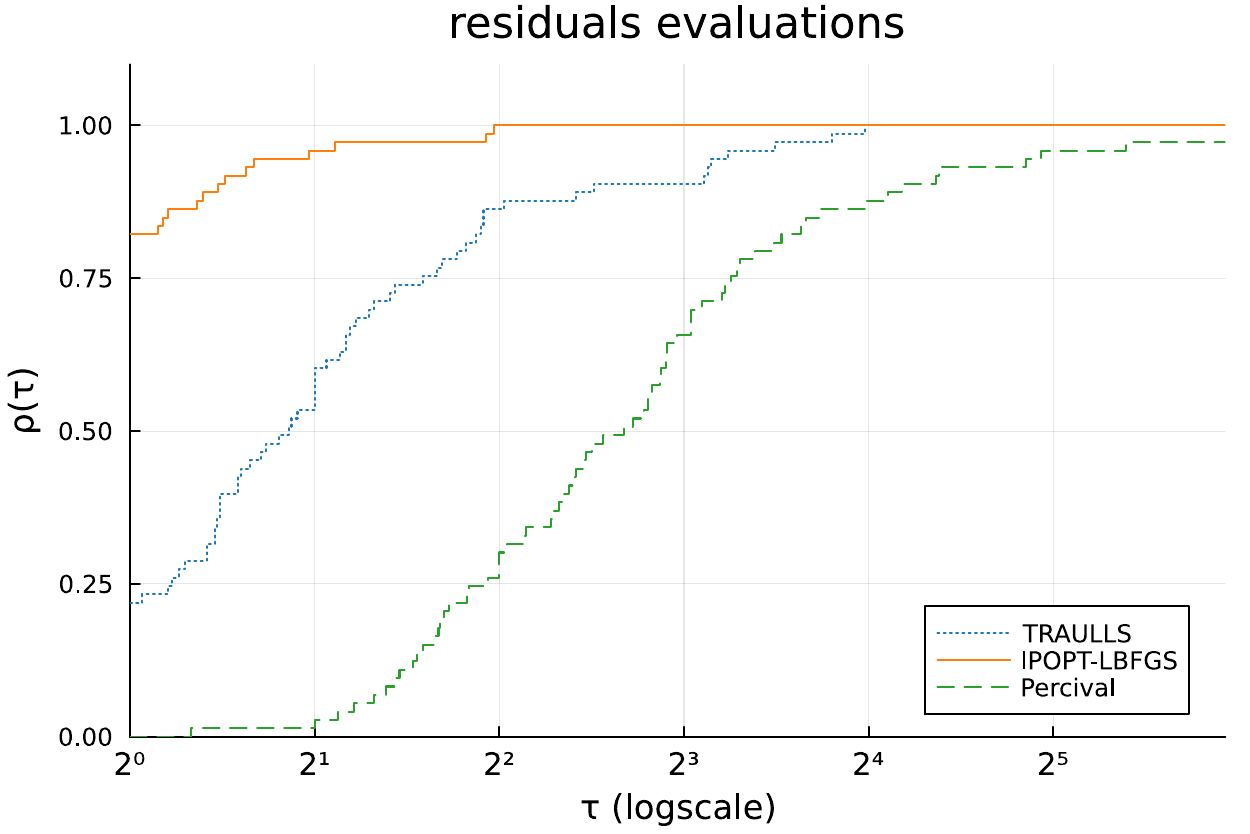}
	\end{minipage}
	\hfill
	\begin{minipage}[b]{0.32\textwidth}
		\centering
		\includegraphics[width=\textwidth]{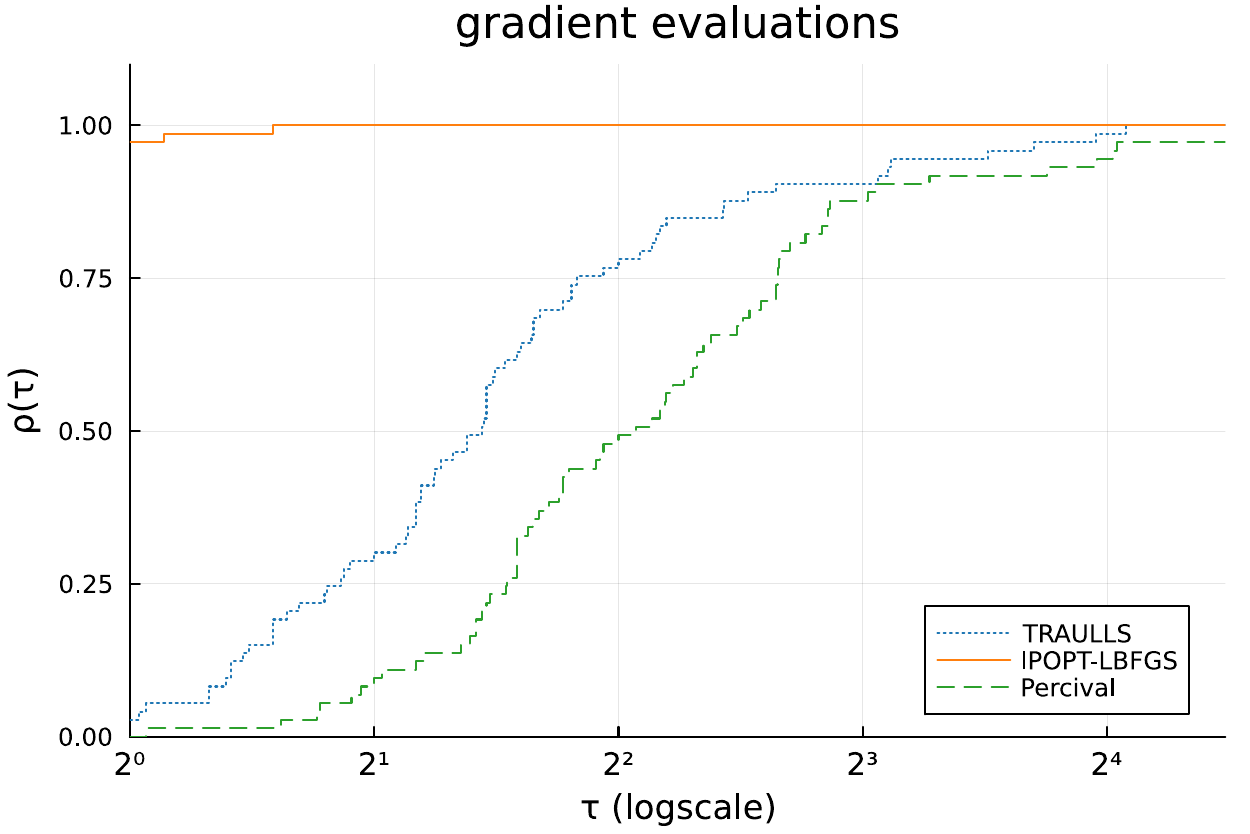}
	\end{minipage}
	\caption{Performance profiles comparing TRAULLS (Hybrid-SR1), IPOPT, and Percival
		on the test set, with respect to computation time (left), residual evaluations (center),
		and gradient evaluations (right).}
	\label{fig:perf_solvers}
\end{figure}

In terms of computation time, TRAULLS is the fastest solver on more than 50\% of the instances. This advantage is, however, strongly dependent on the problem size. For instance, if the comparison is restricted to IPOPT, TRAULLS is faster on all 41 instances with $n \le 100$, but on only 10 of the 26 instances with $n \ge 500$, where our computation time is on average -- in the geometric mean sense -- 1.67 times that of IPOPT. The margin observed on the smallest instances should be read with care, as both codes solve them in a few tens of milliseconds at most, with median times of 0.12 and 1.5 milliseconds respectively, so that the ratio there mostly reflects a lower fixed cost per solve rather than a difference in the optimization work performed. The large instances are where the gap is algorithmically meaningful, and there the cause is the dense storage and manipulation of the structured second-order correction, whose cost grows quadratically with the number of variables while the rest of the method only involves matrix-vector products. The example of problem~5.12 from~\cite{lucksanvlcek:1999} with $n=1000$ illustrates that point: Hybrid-SR1 requires 142 inner iterations against 354 for the GN approximation, yet takes 1.93 seconds against 1.41. This is the limitation we discuss in the conclusion, and it is the reason a limited-memory form of the structured update is the natural next step for large-scale problems.

Regarding the evaluation metrics, IPOPT dominates both profiles -- with a larger gap on gradients than on residuals -- and this dominance decomposes into two distinct effects. TRAULLS performs on average 1.89 times as many residual evaluations as IPOPT, which reflects a genuine difference in the number of steps needed. On top of that, the two methods do not evaluate gradients at the same rate, since our trust-region scheme performs one gradient evaluation per inner iteration, whereas the line search in IPOPT reuses a single gradient across the trial points of a backtracking sequence. Part of the gap is thus structural, inherent to the globalization strategy rather than to the quality of our Hessian model.

These results should be read with the role of each solver in mind. IPOPT is a mature interior-point implementation, and it is included here as an external reference point rather than as the target of our contribution. The comparison that isolates what this work proposes is the one with Percival, which follows the same AL framework, is run with the same tolerances and parameters, and differs from TRAULLS essentially in that it treats the problem as a generic nonlinear program. There, the advantage of TRAULLS is unambiguous and holds on all three metrics. On the 71 retained instances on which Percival converges, it takes in geometric mean 4.28 times as long as TRAULLS, and performs 2.94 times as many residual evaluations and 1.47 times as many gradient evaluations. Also, unlike the comparison with IPOPT, this advantage does not rest on the smallest instances; it is moreover obtained even though TRAULLS only approximates second-order information, whereas Percival evaluates exact Hessian-vector products. Within a common algorithmic framework, exploiting the least-squares structure through the GN term and its structured correction therefore leads to a clear improvement, which is the claim this paper makes. That TRAULLS is at the same time faster than IPOPT on most of the test set, while requiring more function evaluations, situates our implementation as a competitive solver whose remaining bottleneck -- the dense second-order correction on large instances -- is identified and addressable.

\subsection{Linear constraints kept explicit}\label{subsec:linear_constraints}

To assess the treatment of general polyhedra, we solved 15 problems with linear constraints in two ways: with the linear constraints kept explicit, as in~\eqref{pb:cnls}, and with the same constraints passed to the AL as nonlinear constraints. Besides the five instances mentioned at the beginning of this section, the set contains problems HS21, HS42, HS48, HS49, HS50, HS51, HS52 and HS53 from~\cite{hockschittkowski:1980} and HS224 and HS268 from~\cite{schittkowski:1987}, with the residuals of~\cite{nls-problems-julia}. Once slack variables are added, the explicit formulations have up to 10 variables and 5 linear constraints. Table~\ref{tab:linear_constraints} reports the iteration counts of Hybrid-SR1 under both formulations and the violation of the linear constraints at the final point.

\begin{table}[htbp]
	\centering
	\caption{Linear constraints kept explicit (E) or passed to the augmented Lagrangian (P), with Hybrid-SR1: number $n$ of variables of the explicit formulation, slack variables included, number $m$ of linear constraints, numbers of outer and inner iterations, and largest violation of the linear constraints at the final point.}\label{tab:linear_constraints}
	\begin{tabular}{lcccccccc}
		\hline
		& & & \multicolumn{2}{c}{outer iterations} & \multicolumn{2}{c}{inner iterations} & \multicolumn{2}{c}{linear violation} \\
		problem & $n$ & $m$ & E & P & E & P & E & P \\
		\hline
		HS14 & 3 & 1 & 8 & 6 & 12 & 10 & $8\times10^{-16}$ & $1\times10^{-8}$ \\
		HS21 & 3 & 1 & 4 & 6 & 13 & 15 & 0 & 0 \\
		HS22 & 4 & 1 & 6 & 5 & 9 & 8 & 0 & $9\times10^{-7}$ \\
		HS23 & 7 & 1 & 6 & 7 & 23 & 69 & 0 & 0 \\
		HS32 & 4 & 1 & 5 & 7 & 7 & 9 & $1\times10^{-15}$ & $6\times10^{-7}$ \\
		HS42 & 4 & 1 & 6 & 6 & 15 & 16 & 0 & $6\times10^{-7}$ \\
		HS48 & 5 & 2 & 4 & 6 & 12 & 12 & $9\times10^{-16}$ & $1\times10^{-7}$ \\
		HS49 & 5 & 2 & 4 & 4 & 17 & 10 & $3\times10^{-14}$ & $1\times10^{-11}$ \\
		HS50 & 5 & 3 & 5 & 5 & 8 & 7 & $2\times10^{-15}$ & $1\times10^{-7}$ \\
		HS51 & 5 & 3 & 5 & 7 & 13 & 12 & $2\times10^{-16}$ & $9\times10^{-7}$ \\
		HS52 & 5 & 3 & 1 & 6 & 2 & 6 & $5\times10^{-15}$ & $1\times10^{-7}$ \\
		HS53 & 5 & 3 & 1 & 6 & 3 & 8 & $1\times10^{-14}$ & $2\times10^{-7}$ \\
		HS224 & 4 & 2 & 1 & 5 & 4 & 10 & $1\times10^{-14}$ & $1\times10^{-7}$ \\
		HS268 & 10 & 5 & 2 & 7 & 2 & 8 & $2\times10^{-14}$ & $2\times10^{-11}$ \\
		HS323 & 4 & 1 & 6 & 6 & 14 & 18 & 0 & 0 \\
		\hline
	\end{tabular}

\end{table}

Both formulations reach the same solution on 14 of the 15 problems. Kept explicit, the linear constraints hold to machine precision at every iterate, whereas the AL only enforces them up to the feasibility tolerance. The explicit treatment also needs fewer iterations, about 40\% fewer outer and 30\% fewer inner iterations in geometric mean, the largest gains occurring when all the constraints are linear, as on HS52, HS53 and HS224, which are then solved within a single outer iteration. The exception is HS268, whose criticality measure at the starting point is large, about $7\times10^3$. The relative stopping test is then satisfied after two outer iterations of the explicit formulation, before the minimum is reached, whereas the penalized formulation gets closer to it during the additional outer iterations needed to reach feasibility.

\section*{Conclusion}
In this paper, we have presented an algorithm able to solve a class of problems, namely NLS objectives under nonlinear equalities combined with bounds and linear equalities kept explicit, for which, to our knowledge, no dedicated method was available. Our method fully exploits the structure of the problem through the AL reformulation and a structured Hessian approximation. The resulting implementation is faster and more robust than a traditional AL solver, and remains competitive in computation time with a state-of-the-art solver for nonlinear programming. Despite this promising performance, the main limitation for large-scale applications lies in the dense storage of the structured second-order correction of the Hessian, in addition to the subsequent computations that become prohibitive as the number of variables grows. A limited-memory variant of the structured update, following~\cite{byrd-etal:1994}, would be a natural feature to implement. To tackle problems with a large number of residuals, another avenue of research could be to combine the AL framework with an adaptive sampling strategy applied to the residuals, lowering the dimensions of the inner minimization subproblems. This would extend to the constrained case techniques currently investigated for unconstrained problems~\cite{bellavia-etal:2026}.

\vspace{1em}\noindent\textbf{Acknowledgments:} The authors would like to thank Mitacs and the Institute for Data Valorization (IVADO).

\section*{Statements and Declarations}

\paragraph*{Funding} This work was mainly funded by Mitacs through projects IT45065 and IT51449, in partnership with Hydro-Qu\'ebec. It was also funded by Mitacs and IVADO through project IT38322, and the second author is supported by the Natural Sciences and Engineering Research Council of Canada (NSERC) through Discovery Grant RGPIN-2022-04400.

\paragraph*{Competing interests} The third author is an employee of Hydro-Qu\'ebec, the partner organization of the Mitacs projects that funded this work and an industrial partner of IVADO. The authors have no other competing interests to declare that are relevant to the content of this article.

\paragraph*{Data availability} The test problems come from the collections cited in Section~\ref{sec:numerical_experiments}. Their implementations, the result files of all the runs reported in this paper, together with the commit of the solver code that produced them, and the scripts that rebuild its figures, tables and numbers from these files are available in the \texttt{benchmark} folder of \url{https://github.com/UncertainLab/Traulls.jl}, in the version tagged \texttt{coap-2026}, and described in the documentation of the package.

\paragraph*{Code availability} TRAULLS is open source and distributed under the MIT license at \url{https://github.com/UncertainLab/Traulls.jl}. The version used in this paper is tagged \texttt{coap-2026}.

	%%%%%%%%%%%%%%%%%% BIBLIOGRAPHY
	\bibliographystyle{plainnat}
	\bibliography{refs}
	
	\clearpage
	
	\section*{Appendix}
	
	\appendix

	\section{Cauchy point computation}\label{appendix:cauchy_point_computation}

For the sake of clarity, we omit the iteration index during the rest of this appendix. 
We describe the procedure used to compute the first local minimizer of the quadratic $q(s) = \frac{1}{2} \inner{s}{Hs} + \inner{g}{s}$ along a piecewise linear path $s(t)$, $t \ge 0$, which coincides with the projected gradient path $\proj{\Omega}{x-tg}-x$ when $\Omega$ is a box and $\varepsilon = 0$. This is an adaptation of~\cite[Algorithm~17.3.1]{conn-etal:2000} to the case where linear equalities are present. Over a general polyhedron, the active sets along the projected gradient path need not be nested, and the path below is instead defined by nested working sets, a bound being kept fixed once it is reached.
For each $t \ge 0$, the step $s(t)$ satisfies the constraints
\[As(t) = 0, \qquad s^{(l)} \le s(t) \le s^{(u)},\]
where the matrix $A$ is $m\times n$ with $m < n$ and the vectors $s^{(l)}$, $s^{(u)}$ are iteration-dependent bounds having finite components.
The breakpoints
\begin{equation*}
	0=t_0 < t_1 < \ldots < t_p
\end{equation*} 
correspond to the successive scalars at which at least one component of $s(t)$ satisfies a bound with equality. Assuming $A$ has full row rank, there are $n-m$ degrees of freedom remaining and thus at most $n-m$ breakpoints before the path is constant. We recall the recursive expression
\[s(t) = -(t-t_i)z_i + s(t_i),\]
for $t\in [t_i, t_{i+1})$ and with the reduced gradient $z_i$ given by~\eqref{eq:reduced_gradient_proj_grad_path}, the tangent space being now defined by the working set $\calA(t_i)$.

The working set $\calA(0)$ is selected among the $\varepsilon$-active bounds at $x$, $\varepsilon \ge 0$ being the threshold of Section~\ref{subsec:al_framework} ($\varepsilon = 0$ corresponds to the bounds active at $x$). The corresponding variables are fixed at their current values, and the least-squares estimates $\theta$ of their multipliers are given by the decomposition~\eqref{eq:ls_decomposition} of $g$, which the projection provides at no extra cost. A bound is kept in $\calA(0)$ if its multiplier has the sign required at a first-order critical point, and freed otherwise. The freed bounds along which the resulting direction $-z_0$ points outwards are fixed again, and this is repeated until, at every freed bound, $-z_0$ points inwards or has a zero component (iteration~\eqref{eq:refixing}). When $\varepsilon = 0$, this amounts to zero-length segments of the path. When $\varepsilon > 0$, a freed bound may lie at a positive distance from $x$, and it is fixed again explicitly. When $\Omega$ is a box and $\varepsilon = 0$, $\theta_i = g_i$ and this rule reduces to fixing the variables at a bound along which $-g$ points outwards. With this choice, $-z_0$ is a feasible descent direction with $\inner{g}{z_0} = \|z_0\|^2$ and $\|z_0\| \ge \kappa_0\,\pi_\varepsilon(x,g)$ (Proposition~\ref{prop:computed_path} and Lemma~\ref{lemma:initial_direction}).

To find the first local minimizer of the scalar function $\psi(t) = q(s(t))$, we successively study each interval $[t_i, t_{i+1})$ to determine whether or not it contains a local minimizer.
Assume we have not found a local minimizer on $[t_0, t_i)$ and thus look at the interval $[t_i, t_{i+1})$. The model along this arc can be written
\begin{equation}\label{eq:model_projected_gradient_interval}
	\psi(t) = \psi(t_i) + \psi_i'\Delta t + \dfrac{\psi_i''}{2}(\Delta t)^2
\end{equation}
with
$\psi_i'' = \inner{z_i}{Hz_i}$, $\psi_i' = -\inner{s(t_i)}{Hz_i} - \inner{g}{z_i}$ and $\Delta t = t-t_i$.
Different cases can occur depending on the values of the slope $\psi_i'$ and the curvature $\psi_i''$. 
First, if
\[\psi_i' \ge 0,\]
we stop at $t_i$. This is the required minimizer, except in the degenerate case where $\psi_i' = 0$ and $\psi_i'' < 0$. Stopping there does not affect the decrease guarantee, which only involves the first segment of the path, on which $\psi'$ is negative unless $z_0 = 0$. In floating-point arithmetic, we also stop when $\|z_i\|_\infty \le 10\,\epsilon_{\mathrm{M}}\|g\|_\infty$, where $\epsilon_{\mathrm{M}}$ is the relative machine precision. The direction $z_i$ is then made of rounding errors, and the ratio $-\psi_i'/\psi_i''$ formed from it is meaningless and may even be negative. Next, if
\[\psi_i' < 0 \text{ and } \psi_i'' > 0,\]
the quadratic~\eqref{eq:model_projected_gradient_interval} has a strict minimizer at
\[t_i- \dfrac{\psi_i'}{\psi_i''}.\]
If the latter belongs to the interval of interest, i.e.,
\[t_i-\psi_i'/\psi_i'' < t_{i+1},\]
then it is the required minimizer. In other cases, the minimizer is at or beyond $t_{i+1}$. To prepare for the study of the next interval, we first need to find the next breakpoint, given by the smallest scalar $t_{i+1}>t_i$ such that, at $s(t_{i+1})$, one of the free components hits one of its bounds. By introducing
\[\delta_i^- = \min_{(z_i)_j > 0} \left\{ \dfrac{(s(t_i))_j - s^{(l)}_j}{(z_i)_j}\right\}, \qquad \delta_i^+ = \min_{(z_i)_j < 0} \left\{ \dfrac{(s(t_i))_j - s^{(u)}_j}{(z_i)_j}\right\},\]
the next breakpoint is given by 
\begin{equation}\label{eq:next_breakpoint}
	t_{i+1} = t_i + \delta_i,
\end{equation} 
with $\delta_i = \min\left(\delta_i^-, \delta_i^+\right)$. We then add the corresponding variable index to the list of fixed components, compute the next reduced gradient $z_{i+1}$ and finally update the slope and curvature of the model along the next interval. A last termination case can occur. Since $A$ is of rank $m$, at most $n-m$ bounds can become active, so if we never find a local minimizer, the procedure still ends whenever it reaches the last breakpoint $t_{p}$. Indeed, past this breakpoint, the model along the path is constant so we can return the last accumulated step $s(t_p)$ as the Cauchy step. Note that, in this case, it is also the total step, because no more variables can be modified. The finiteness of the box $\left[s^{(l)}, s^{(u)}\right]$ guarantees that, whenever $t$ passes the last breakpoint, the working set leaves no degree of freedom: $n-m$ components of the step $s(t)$ are fixed at one of their bounds, and the others are determined by $As(t) = 0$.
The procedure for the Cauchy step computation is outlined in Algorithm~\ref{algo:cauchy_point}.
\begin{algorithm}
	\caption{Cauchy step computation}\label{algo:cauchy_point}
	\begin{algorithmic}[1]
		\Require Point $x$, threshold $\varepsilon \ge 0$, Hessian $H$, gradient $g$, bounds on the direction $s^{(l)},\ s^{(u)}$
		\State Fix the variables of the $\varepsilon$-active bounds at $x$, compute their multiplier estimates $\theta$, free those of the wrong sign, and compute the direction $z_0$, fixing again the freed bounds along which $-z_0$ points outwards
		\State Set \textbf{found} $\gets$ \textbf{false} 
		\State Set $t_0\gets 0$, $s^{(0)}\gets 0$ and initialize counter $i\gets 0$
		\Repeat
		\State $\psi_i' \gets -\inner{s^{(i)}}{Hz_i}-\inner{g}{z_i},\ \psi_i'' \gets \inner{z_i}{Hz_i}$
		\State Find the next breakpoint $t_{i+1}$ by~\eqref{eq:next_breakpoint}
		\If{$\psi_i' \ge 0$ or $\|z_i\|_\infty \le 10\,\epsilon_{\mathrm{M}}\|g\|_\infty$}
		\State Set $s^C \gets s^{(i)}$
		\State \textbf{found} $\gets$ \textbf{true}
		\ElsIf{$\psi_i' < 0$ and $\psi_i'' > 0$ and $t_i- \psi_i'/\psi_i'' < t_{i+1}$}
		\State Set $\Delta t \gets - \psi_i'/\psi_i''$
		\State Set $s^C \gets s^{(i)} - \Delta t z_i$, \textbf{found} $\gets$ \textbf{true}
		\Else
		\State Set $\Delta t \gets t_{i+1}-t_i$
		\State Compute the next direction $z_{i+1}$
		\State Set $s^{(i+1)} \gets s^{(i)} - \Delta t z_i$
		\If{$T(t_{i+1}) = \{0\}$}
		\State Set $s^C \gets s^{(i+1)}$, \textbf{found} $\gets$ \textbf{true}
		\EndIf
		\EndIf
		\State Increment $i \gets i+1$
		\Until{\textbf{found}}
		\State \Return $s^C$
	\end{algorithmic}
\end{algorithm}

\end{document}